\documentclass[article, 12pt]{amsart}

\usepackage{amsmath,amssymb,amsthm,graphicx}
\usepackage{latexsym}
\usepackage{amscd}
\usepackage{amsfonts}
\usepackage[usenames,dvipsnames]{color}
\usepackage{verbatim}
\usepackage{float}
\usepackage{url}
\usepackage{xypic}
\usepackage{marginnote}\usepackage{mathtools}
\usepackage[pdftex,colorlinks]{hyperref}
\usepackage[mathscr]{euscript}

\DeclareMathOperator{\loc}{loc}
\DeclareMathOperator{\dR}{dR}

\DeclareMathOperator{\cpt}{cpt}

\DeclareMathOperator{\sing}{sing}
\DeclareMathOperator{\str}{str}
\DeclareMathOperator{\even}{even}

\DeclareMathOperator{\rank}{rank}
\DeclareMathOperator{\Irr}{Irr}

\DeclareMathOperator{\codim}{codim}

\DeclareMathOperator{\Ker}{Ker}
\DeclareMathOperator{\Sym}{Sym}

\DeclareMathOperator{\Seg}{Seg}

\DeclareMathOperator{\Tr}{Tr}
\DeclareMathOperator{\an}{an}
\DeclareMathOperator{\ch}{ch}

\DeclareMathOperator{\Real}{Re}

\DeclareMathOperator{\Ex}{Exc}
\DeclareMathOperator{\ST}{ST}
\DeclareMathOperator{\TP}{TP}

\DeclareMathOperator{\Imag}{Im}

\DeclareMathOperator{\Lip}{Lip}
\DeclareMathOperator{\EBl}{EBl}

\DeclareMathOperator{\Tc}{Tc}

\DeclareMathOperator{\gen}{gen}

\newtheorem{prop}{Proposition}[section]
\newtheorem{theorem}[prop]{Theorem}

\newtheorem{lem}[prop]{Lemma}
\newtheorem{cor}[prop]{Corollary}
\theoremstyle{definition}

\newtheorem{definition}[prop]{Definition}
\newtheorem{rem}[prop]{Remark}
\newtheorem{example}[prop]{Example}

\newtheorem*{rem*}{Remark}
\newtheorem*{notation*}{Notation}

\numberwithin{equation}{section}

\newcommand{\bR}{\mathbb{R}}
\newcommand{\ra}{\rightarrow}
\newcommand{\bC}{\mathbb{C}}
\newcommand{\bA}{\mathbb{A}}
\newcommand{\bZ}{\mathbb{Z}}

\newcommand{\bP}{\mathbb{P}}

\newcommand{\End}{\mathrm{End}}
\DeclareMathOperator{\id}{id}

\DeclareMathOperator{\Bl}{Bl}
\DeclareMathOperator{\Div}{Div}
\DeclareMathOperator{\Pic}{Pic}

\DeclareMathOperator{\Gr}{Gr}

\DeclareMathOperator{\supp}{supp}
\DeclareMathOperator{\mer}{mer}

\DeclareMathOperator{\Proj}{Proj}

\newcommand{\Hom}{\mathrm{Hom}}

\begin{document}

\title{Double transgressions and Bott-Chern Duality}
\author{Daniel Cibotaru}
\address{Universidade Federal do Cear\'a, Fortaleza, CE, Brasil}
\email{daniel@mat.ufc.br} 
\author{Vincent Grandjean} 
\address{Universidade Federal do Cear\'a, Fortaleza, CE, Brasil}
\email{vgrandjean@mat.ufc.br} 
\author{Blaine Lawson, Jr.}
\address{Stony Brook University, Stony Brook, NY, USA}
\email{blaine@math.stonybrook.edu}
\begin{abstract}
We present a general framework for obtaining currential double transgression formulas on complex manifolds which can be seen as manifestations of Bott-Chern Duality. These results complement on one hand the simple transgression formulas obtained by Harvey -Lawson and on the other hand the double transgression formulas of Bismut-Gillet-Soul\'e. 

Among the applications  we mention a Gysin isomorphism for Bott-Chern cohomology, an abstract Poincar\'e-Lelong formula for sections of holomorphic and Hermitian vector bundles implying Andersson's generalization  of the standard Poincar\'e-Lelong, a Bott-Chern duality formula for the Chern-Fulton classes of singular varieties or a refinement of the first author's simple transgression formula for the Chern character of a Quillen superconnection associated to a self-adjoint, odd endomorphism. The existence of a Bismut-Gillet-Soul\'e double transgression without the  hypothesis of degeneration along a submanifold stands out.  As a by-product   we  also obtain a statement about the \emph{pointwise localization}  of the Samuel multiplicity of an analytic subvariety of a complex manifold along an irreducible component. 
\end{abstract}

\maketitle
\tableofcontents

\section{Introduction}

Simple, currential transgression formulas are  manifestations of Poincar\'e Duality on a smooth manifold $M$. They typically take the following form
\begin{equation}\label{eq1} \omega-Z=dT
\end{equation}
where  $\omega$  a smooth closed differential form, $Z$ is a (flat) current with support along some (locally) rectifiable set  and $dT$ is an exact current. The naive idea  to obtain such formulas is to take the closed form $\omega$ and deform it via some continuous process $\omega_t$ such that $\omega_t$ converges weakly to $Z$.  For example, Laudenbach \cite{BZ} and Harvey-Lawson \cite{HL1,HL2, HL3}   used Morse gradient flows $\varphi:\bR\times M\ra M$ in order to construct a homotopy between the deRham complex and the Morse complex via the deformation $\omega\ra \varphi_t^*\omega$. When the Morse flows are considered along the fibers of a fiber bundle $P\ra M$  more applications can be given, like for example the Gauss-Bonnet formula, or an answer to Quillen's question about the the deformation of the Chern supercharacter associated to an odd self-adjoint endomorphism (see for example \cite{CiHl3}). 

The name  transgression  comes from standard Chern-Weil theory and the well-known relation 
\begin{equation}\label{eq2} P(\nabla_1)-P(\nabla_2)=d\TP(\nabla_1,\nabla_2)
\end{equation}
which proves the cohomological invariance of the characteristic forms $P(\nabla)$ associated to a fixed invariant polynomial $P$ and a connection $\nabla$. The holomorphic counterpart to this equality is the Bott-Chern formula \cite{BC} which we now recall. If $M$ is a complex manifold of dimension $n$ and $E\ra M$ is a holomorphic vector bundle endowed with a family of Hermitian metrics $(h_t)_{t\in[1,2]}$ and $\nabla_t$ is the corresponding family of unique (Chern) connections compatible with both  the metric and the holomorphic structure then  one can refine \eqref{eq2} to 
\[P(\nabla_1)-P(\nabla_2)=\partial\bar{\partial} T(\nabla_1,\nabla_2)
\]
for some smooth form $T(\nabla_1,\nabla_2)$.  

This parallel raises the question of whether one can adapt the deformation theoretic techniques to the complex setting in order to obtain currential double transgression formulas, i.e. expressions of type
\begin{equation}\label{eq20}\omega-Z=\partial\bar{\partial} T
\end{equation}
where again $\omega$ is smooth and $Z$ is flat. The homological context for such formulas is Bott-Chern duality. 

The Bott-Chern cohomology groups $H_{BC}^{p,q}:=\frac{\Ker \partial\cap \Ker{\bar{\partial}}}{\Imag \partial{\bar\partial}}$ and their Hodge $*$-duals, the Aeppli groups $H_{A}^{p,q}:=\frac{\Ker \partial{\bar\partial}}{\Imag \partial+ \Imag{\bar{\partial}}}$ were introduced independently \cite{A,BC} in the mid sixties and have become indispensable tools to complex geometers \cite{Ang,Dem}. On a compact K\"ahler manifold they are isomorphic with some older acquaintances, the Dolbeault groups, consequence of the $\partial\bar{\partial}$ Lemma (see \cite{Ang,DGMS}). On a general compact complex manifold this is not the case, but in recent years, substantial  progress has been made in understanding and computing these groups on certain classes of non-K\"ahler manifolds \cite{Ang,AngK,AngT}.

 Bismut's impressive work \cite{Bi00} on the  Grothendieck-Riemann-Roch Theorem in Bott-Chern cohomology for proper submersions under the condition that the (higher) direct image sheaf is  locally free stands out as a major achievement in this area. Smooth double transgression formulas appear in \cite{Bi00}, albeit with a different formalism than in the present article.

 In a separate direction, currential double transgressions and Bott-Chern currents  appeared in the vast work of Gillet-Soul\'e \cite{GiS} and Bismut-Gillet-Soul\'e \cite{BiGiS,BiGiS2} and indeed they are fundamental in the definition of Green currents \cite{GiS,Ha} and the arithmetic Chow groups.  One of the first motivations for this article was the introduction of new tools that eventually should produce  more  direct proofs of the transgression formulas of Bismut-Gillet-Soul\'e like Theorem 2.5 from  \cite{BiGiS} that fit within a  natural theory of double transgressions.\footnote{The word natural is used here  with an eye towards  the facility  to come up with such expressions, rather than  emphasizing certain algebraic properties.} The importance of such formulas cannot be overstated. For example, it is used by the same authors in \cite{BiGiS} to explicitly compute the arithmetic Chern character of a chain complex of vector bundles under certain technical hypothesis. More recently, Garcia and Sankaran  constructed in \cite{GaSan} natural Green forms on orthogonal and unitary Shimura varieties.

The influential Bismut (simple transgression) Theorem 3.2 from \cite{Bi} from which the above mentioned result from \cite{BiGiS} is derived  is based on H\"ormander's theory of wave front sets. One also gets as a bonus "rates" of convergence of the relevant currents in an appropriate topology,  something not  available with the standard tools of Geometric Measure Theory.  However, this approach  has  limitations when dealing with singularities as the "degeneration locus" is usually considered to be  a complex submanifold $M'$.  The techniques we introduce here, of a completely different nature, try to overcome these technical limitations. In fact, as one might expect, the analytic geometry language encodes enough information that allows  to deal with more singular situations. 
 
One of the main new results of this article is a proof in Theorem \ref{mainc10} of the existence  of a Bismut-Gillet-Soul\'e double transgression under very general hypothesis, removing for example the condition of degeneration along a submanifold.  Explicit computations can be performed, but in order to keep  the article to a reasonable length we defer them to  future work. \vspace{0.1cm}

Returning to more elementary considerations, a relation of type (\ref{eq20}) immediately brings to mind a standard result in complex geometry, the celebrated Poincar\'e-Lelong formula that was proved, in fact, by Griffiths and King in \cite{GK} (Proposition 1.14).

\begin{theorem} \label{t1} If $L\ra M$ is a Hermitian and holomorphic line bundle and $s:M\ra L$ is a non-identically zero, holomorphic section then
\[ [s^{-1}(0)]-c_1(\nabla)=\frac{i}{\pi}\partial\bar{\partial}[\log{|s|}]
\]
where $\nabla$ is the Chern connection of $L$ and $[s^{-1}(0)]$ is the current of integration over the regular part of $s^{-1}(0)$.
\end{theorem} 
The Poincar\'e-Lelong formula is  the holomorphic counterpart of the statement that the zero locus of a smooth section of a complex line bundle transverse to the zero section differs from the first Chern form by an exact current. This article says that this analogy can be extended much  further. In fact, if we replace the flow-deformation techniques of \cite{Ci2} with certain algebraic $\bC^*$ actions then the simple transgression formulas turn to double transgressions. Nevertheless, the techniques for proving such results are completely different. We use the idea of deformation to the normal cone (\cite{F}, Ch. 5)   together with Hardt's Slicing Theorem and the standard Poincar\'e-Lelong formula in order to reach the following conclusion.

\begin{theorem}\label{t3} Let $E\ra M$ be a Hermitian and holomorphic vector bundle and $s:M\ra E\subset\bP(\bC\oplus E) $ be a  non-identically zero, holomorphic section. Let
$\omega\in \Omega^*(\bP(\bC\oplus E))$ be a $\partial$ and $\bar{\partial}$ closed form. Then the following holds
\[ s^*\omega- s_{\infty}^*\omega - Z(s,\omega)=\partial\bar{\partial}T(s,\omega)
\] 
where $s_{\infty}: M\setminus s^{-1}(0)\ra \bP(E)$ is  $s_{\infty}(m)=[s(m)]$,  $Z(s,\omega)$ is an explicit flat current with support on $s^{-1}(0)$ and $T(s,\omega)$ is a current with  locally integrable coefficients.
\end{theorem}

As a consequence we obtain the next generalizations of Poincar\'e-Lelong.

\begin{theorem}\label{M1} Let $E$ be of rank $k$ and $s:M\ra E$ be a section such that $s^{-1}(0)$ is of pure complex codimension $l>0$, $l\leq k$. Let $\delta_j^{l}$ be the Kronecker symbol. The following  hold for all $j\leq l$

\begin{equation}\label{E10} c_j(\nabla^E) - c_{j}(\nabla^{Q_s}) - \delta_j^l[s^{-1}(0)] = -\frac{i}{\pi} \partial\bar{\partial}\left( \log{|s|}\left(c_{j-1}{(\nabla^{Q_s}})+ \partial \bar{\partial}\eta_1\right)\right)\end{equation}
\begin{equation}\label{E20}-c_1(\nabla^{L_s^*})^j -\delta_j^l [s^{-1}(0)]= -\frac{i}{\pi}\partial\bar{\partial}\left( \log{|s|}\left(c_1(\nabla^{L_s^*})^{j-1}+\partial \bar{\partial}\eta_2\right)\right)\end{equation}
where $\eta_1$ and $\eta_2$ are smooth forms on $U:=M\setminus s^{-1}(0)$, all equal to $0$ for $j=1$. Here, $L_s^*$ is the dual to $L_s:=\langle s\rangle$ and  $Q_s:=E/L_s$, all vector bundles over $U$.

If $s$ does not vanish, then the two equations are identities of smooth forms on $M$ and hold for all $j$. The analytic current $[s^{-1}(0)]$ is the sum of the irreducible components of $s^{-1}(0)$, each multiplied with its Samuel multiplicity. 
\end{theorem}

Formula (\ref{E10})  was also obtained by M. Andersson in \cite{An}, using completely different techniques from Residue Theory. The computation of the right hand side of (\ref{E10}) (resp. (\ref{E20})) is new, albeit incomplete as it involves the presence of the form $\eta_1$ (resp. $\eta_2$). We believe this "leftover" can also be expressed as polynomials in the entries of the curvature tensors of $Q_s$ (resp. $L_s$).

We obtain Theorem \ref{M1} by using certain particular closed forms $\omega$ on $\bP(\bC\oplus E)$ in Theorem \ref{t3}. In fact, the choice of these forms is really natural in view of Fulton's "topological" description of the "embedded" multiplicities of a subscheme along an irreducible variety. The algebraic operation on  Chow groups given by a Chern class corresponding to a polynomial $P$ becomes naturally the operation of wedging with the Chern form associated to $P$ and the Chern connection  in the Bott-Chern groups. Naturally one can think of other applications here and we have selected two that are rather direct.

Fulton defines in \cite{F}, Example 4.2.6 the total Chern class $c^F_{ * }(X)$ of a scheme $X$  as a certain class in the Chow group $A_*(X)$ by using an embedding of $X$ into a regular variety $M$. These classes do not depend on the embedding. There exists a natural map
\[A_*(M)\ra H^{BC}_{*,*}(M)
\]
which of course can be composed with the pushforward $A_*(X)\ra A_*(M)$. If $X$  is as before the zero section of a holomorphic vector bundle over $M$,  we get the following statement.

\begin{theorem} Let $s:M\ra E$ be a holomorphic section such that  $s^{-1}(0)$ is pure  of codimension $c$ and dimension $d$. Then the  Chern-Fulton classes  satisfy in $H_{d-*,d-*}^{BC}(M)$:
\begin{equation}\label{CF0}  c^F_{ * }(s^{-1}(0))=-c_1({L_s}^*)^c(c_{ * }( {TM\otimes L_s^*})),
\end{equation}
where $L_s:=\Imag s\bigr|_{M\setminus s^{-1}(0)}$ is the line subbundle determined by $s$. The form on the right hand side is smooth over $M\setminus s^{-1}(0)$, with locally integrable coefficients on $M$.
\end{theorem}

We continue the applications inspired by \cite{F}. The generalized Poincar\'e-Lelong formula (\ref{E10}) contains on the left hand side a rather undesired correction term, namely $c_{j}(\nabla^{Q_s})$. One is then left wondering what exactly happened with the Poincar\'e duality between the zero locus of a section and the top Chern class. On one hand, the correction term is absent when the dimension of $s^{-1}(0)$ equals the expected dimension which is the rank of the vector bundle. We infer then that formula (\ref{E10})  provides a Bott-Chern (and implicitly a Poincar\'e) dual for $s^{-1}(0)$ when this generic condition is not met. 
On the other hand, one can separately obtain a Bott-Chern duality statement for the top Chern class $c_k(E)$, which we discuss next.

Let $C_{s^{-1}(0)}M$ be the affine normal cone of $s^{-1}(0)$. It is naturally a complex subspace of $E$. Let $\tau$ be a Thom form for $E\ra M$, i.e. a smooth form with compact support, such that $\int_{E/M}\tau =1$. 

Then $C_{s^{-1}(0)}M\wedge\tau$ makes perfect sense as a current on $E$. Let $\bZ(s):=\pi_*(C_{s^{-1}(0)}M\wedge\tau))$ be the corresponding current on $M$. Then

\begin{theorem} Let $k$ be the rank of $E$. The following holds in $H^{BC}_{n-k,n-k}(M)$:
\[ \bZ(s)=c_k(E).
\]
\end{theorem}

This formula led us to investigate a natural question about the Bott-Chern cohomology groups: the existence of a Gysin isomorphism. The natural integration of forms over the fiber is just the restriction of the operation of push-forward of currents to forms. Hence $\pi_*$ commutes with $\partial $ and $\bar{\partial}$ and we have a well-defined morphism
\begin{equation}\label{GysBC} \pi_*:H^{p,q}_{BC,\cpt}(E)\ra H^{p-k,q-k}_{BC}(M).
\end{equation}
 We  have the following.
\begin{theorem} Let $M$ be a compact complex manifold. Then the map $\pi_*$ in (\ref{GysBC}) is an isomorphism.
 \end{theorem}
 
 The proof of this result goes through Bott-Chern Duality with supports proved by Bigolin \cite{Big} and a certain "universal" transgression formula that holds on $\bP(\bC\oplus E)$, namely if $\omega$ is a $\partial$ and $\bar{\partial}$ closed form of  $\bP(\bC\oplus E)$, then one has:
 \[  \omega\bigr|_{E}-\pi_*(\omega)\wedge[M]-(\pi^{E})^*(\omega\bigr|_{\bP(E)})=\partial\bar{\partial}[T(\omega)].
\]
where $[M]$ is the zero section, $T(\omega)$ is a current of finite mass (locally) and $\pi^E:E\setminus \{0\}\ra \bP(E)$ is the natural projection.

We mentioned earlier that the simple transgression formulas have a dynamical point of view associated to them. Double transgressions make no exception. One of the  motivating questions for this work was how can one  compute limits of smooth forms in the weak sense
\[  \lim_{\lambda \ra \infty}(\lambda s)^*\omega
\]
when $s$ is a section of a vector bundle $E\ra M$ and $\omega$ is a form on the total space. It turns out that it is technically convenient to work on a complex compactification of $E$ and $\bP(\bC\oplus E)$ is here the natural choice. As noticed by Harvey and Lawson \cite{HL1}, the Schwartz kernel of the pull-back operator $(\lambda s)^*$  is the graph of $\lambda s$. So we get confronted with the geometric question of computing the limits of the graphs of $\lambda s$, i.e. 
\begin{equation}\label{E001} \lim_{\lambda\ra \infty}\{(s(m),[1:\lambda s(m)])\in E\times_{M}\bP(\bC\oplus E))~|~m\in M\}
\end{equation}
As one might expect, this is where the deformation to the normal cone procedure comes in. It turns out that everything can be expressed in terms of the blow-up of the regular subspace $\infty\times[0]  \subset \bP^1\times E$ and strict transforms with respect to this blow-up. To make things more palpable we use a concrete description of this blow-up:
\[\Bl_{\infty\times [0]}(\bP^1\times E)=\{([\mu:\lambda],v,[\theta:w])\in \bP^1\times E\times_M\bP(\bC\oplus E)~|~(\mu,\lambda  v) \wedge(\theta,w)=0\}
\]

 We need to look at the strict transform $\tilde{S}$ of $\bP^1\times s(M)$ in order to read  the limit in (\ref{E001}). On a closer look we also see a different point of view here that points to a generalization. The blow-up $\Bl_{\infty\times [0]}(\bP^1\times E)$ is in fact the closure of the graph of the $\bC^*$-action  on $E$, i.e. the closure  in $\bP^1\times E\times \bP(\bC\oplus E)$ of
 \[\{ (\lambda,v,[1:\lambda v])\in \bC^*\times E\times \bP(\bC\oplus E) \}\]
 One can  consider weighted homogeneous actions as follows.
 
  Let $E=E_0\oplus E_1\ldots \oplus E_k$ be a splitting of $E$ into a direct sum and consider that $\bC^*$ acts by $\lambda^{\beta_i}$ on $E_i$ where $\beta_0=0$ and $\beta_i$ is an increasing sequence of non-negative integers. Then this induces an action of $\bC^*$ on $\bP(E)$. It turns out that the closure of the graph of the action in $\bP^1\times \bP(E)\times_M \bP(E)$ is the blow-up of an explicit sheaf of (non-reduced) ideals $\mathcal{I}$ on $\bP^1\times \bP(E)$.
  
  In this article we give explicit equations for the closure of the graph of the $\bC^*$ action (see Proposition \ref{Pn}).  
  
   Using the same theoretical set-up as in the homogeneous case we can show that limits of type
 \begin{align}\label{0eq1} \lim_{\lambda\ra \infty}\{(s(m),[\lambda*s(m)])\in \bP(E)\times_{M}\bP(E))~|~m\in M\}\\
\label{0eq2}\lim_{\lambda\ra 0}\{(s(m),[\lambda*s(m)])\in \bP(E)\times_{M}\bP(E))~|~m\in M\}
\end{align} 
exist and produce subsequent double transgression formulas. Here $*$ represents the weighted homogenenous action and $s:M\ra \bP(E)$ is a section whose image is not completely contained in the fixed point set of the action.

 These type of situations are by no means artificial, since the natural rescaling action of $\bC^*$ on $\Hom(E'\oplus F)$ which compactifies to an action on $\Gr_k(E'\oplus F)$, ($k=\rank E'$) can be seen via the Pl\"ucker embedding as the restriction of a weighted homogeneous  action on $\bP(\Lambda^{k}(E'\oplus F))$.
 
One can compute  quite explicitly the limits (\ref{0eq1}) and (\ref{0eq2}) and fit them into a double transgression formula by following pretty much the same strategy as for $k=1$. It turns out that each limit is now a sum of (at most) $k+1$ terms, each of these corresponding to one irreducible component of the blow-up of the ideal $\mathcal{I}$. In Section \ref{Sec8} we give the equations of this blow-up and describe the components of the exceptional divisors. In fact, if $\tilde{S}$ is the strict transform of $\bP^1\times s(M)$ as before and we let $\tilde{S}^{\infty}$ to be the part of $\tilde{S}$ lying over $\{\infty\}\times \bP(E)\times_M\bP(E)$, then  $\tilde{S}^{\infty}$ is made from $k+1$ pieces obtained as intersections
\[ \tilde{S}^{\infty}_i:=\tilde{S}^{\infty}\cap \overline{S(F_i)\times_{F_i}U(F_i)}
\] 
where $F_i=\bP(E_i)$ is one connected component of  the fixed points of the $\bC^*$ action and $S(F_i)/U(F_i)$ are the stable/unstable associated sets. 

 In the homogeneous case, when $k=1$, the two terms have a clear  description in terms of classical analytic processes, blow-ups, exceptional divisors, cones (see Proposition \ref{Psp1}) and in a sense this is what accounts for the simplicity of the abstract Poincar\'e-Lelong  Theorem \ref{t3}. The situation is quite a bit more complicated for general $k$, owing to the fact that the orbit space of a $\bC^*$ action on a projective space is far from being Hausdorff.

We summarize  the results  of Sections \ref{Sec8} and \ref{Sec9}  in the following  rather imprecise statement (see the more precise Theorem \ref{Th95}).

\noindent
{\bf{"Theorem``}:} \; \emph{The closure of the graph of the action $\bC^*\times \bP(E)\ra \bP(E)$ in $\bP^1\times \bP(E)\times \bP(E)$ coincides with the blow-up of a sheaf of ideals $\mathcal{I}$ on $\bP^1\times \bP(E)$ and the limits (\ref{0eq1}) and (\ref{0eq2}) are the two components of the exceptional divisor\footnote{see the convention of Remark \ref{Remexcdiv}} of the strict transform (in this blow-up) of $\bP^1\times s(M)$ lying over $\infty$ and $0$ respectively in $\bP^1$. Moreover, each limit can be written as a sum of $k+1$ terms. The affine part of $\tilde{S}^{\infty}_i$, i.e. $\tilde{S}_i^{\infty}\cap [S(F_i)\times_{F_i}U(F_i)]$ (and similarly for $\tilde{S}^0_i$) can be described as the exceptional divisor of a weighted affine cone of a weighted blow-up.}

\vspace{0.2cm}

 This is the gate for  more applications.

On one hand we recover as a corollary the Bott-Chern classical transgression formula (subsection \ref{BCtheorem}) which says that given an exact sequence 
 \[ 0\ra E'\ra E\ra E''\ra 0
 \]
of holomorphic vector bundles with Hermitian metrics then at the level of Bott-Chern cohomology groups one has:
\[ c_*(E)=c_*(E'\oplus E'').
\]

Second, we turn to Quillen's question in \cite{Q0}, of computing limits of superconnections Chern character forms associated to an odd self-adjoint endomorphism of a vector bundle. The result is as follows.

\begin{theorem}\label{T002}
Let $E:=E^+\oplus E^-\ra M$ be a holomorphic super vector bundle with a compatible Hermitian metric and corresponding Chern connection $\nabla=\nabla^+\oplus \nabla^-$. Let $\tilde{A}:M\ra\Hom(E^+,E^-)$ be a holomorphic section and $\bA_z:=\left(\begin{array}{cc} \nabla^+& (z\tilde{A})^*\\
 z\tilde{A} & \nabla^-\end{array}\right)$ be the corresponding family of superconnections on $E$. Then there exists a double transgression formula
\[ \ch(\nabla)-\lim_{z\ra\infty}\ch(\bA_z)=\partial{\bar{\partial}} T(\tilde{A})
\]
where $\lim_{z\ra\infty}\ch(\bA_z)$ is a sum of flat currents, each  supported on $\{p\in M~|~\dim{\Ker\tilde{A}}= i\}$, where $1\leq i\leq k^+$.
\end{theorem}
The proof is based on a  Theorem of Quillen \cite{Q1} that extends the universal Chern supercharacter form from the morphism bundle to the Grassmannian bundle. We remark that under appropriate  transversality conditions the limit $\lim_{z\ra \infty} \ch(\bA_z)$ has already been computed in \cite{CiHl3}.

Finally we  turn to retrieving the Bismut-Gillet-Soul\'e singular double transgression.  
\begin{theorem}\label{mainc01} Let $v_i:E_i\ra E_{i-1}$ be a finite sequence of holomorphic chain morphisms between holomorphic  vector bundles  endowed with Hermitian metrics. Let $E^+=\oplus_{i} E_{2i}$ and $E^-:=\oplus_{i} E_{2i-1}$ with Chern connections $\nabla^{E^{\pm}}$ and $v^{\pm}:E^{\pm}\ra E^{\mp}$ be the corresponding morphisms induced from the $v_i$'s. For $\lambda\in\bC$ denote by
$\bA_{\lambda}$ the superconnection on $E=E^+\oplus E^-$ induced by the odd morphism $v=v^++v^-$ and defined by:
\[ \bA_{\lambda}=\nabla^{E^+}\oplus \nabla^{E^-}+(\lambda v+\bar{\lambda}v^*)
\] Then there exists a double transgression formula:
\[ \lim_{\lambda\ra \infty} \ch(\bA_{\lambda})-\ch(\bA_1)=\partial\bar{\partial} T(v).
\]
\end{theorem}
In order to prove this result we notice that the Bismut-Gillet-Soul\'e superconnection and Quillen superconnection of Theorem \ref{T002} are  closely related. More exactly the latter is  a particular case of the former for $v^-\equiv 0$.  It is therefore natural to try and use the same strategy used for Theorem \ref{T002}. However,  Quillen's Extension Theorem does not apply directly as  now one has to deal with a sum of morphisms.  We need a result that might be also of a separate interest.   We prove that the addition operation of linear morphisms of vector spaces extends holomorphically via the  graph map  to an operation defined on a bigger open space of  pairs of linear subspaces (correspondences) in the cartesian product of the Grassmannian with itself. Quite magically the closure (in this product) of the chain morphism equations $v\circ v=0$ is contained in this bigger open set. Hence the "universal Chern character superconnection forms" exist on the corresponding open subset in the vector bundle context. Since the limiting kernels have support in the closure of $v\circ v=0$ we can apply the theory.


\vspace{0.5cm}

A few more comments are in order.  Double transgression formulas for forms in this article are a consequence of double transgression formulas for graphs of sections of some holomorphic fiber bundle. These in turn are  consequences of the standard Poincar\'e-Lelong and the description of the closure of the graph of the algebraic $\bC^*$-action on the section as a certain strict transform. The idea of the deformation to the normal cone of MacPherson appears absolutely naturally in this context. The use in \cite{BiGiS2} of the same idea is unrelated to the context we work in here. In fact, the reader can distill a generalization of this process in the context of weighted $\bC^*$ actions that we treat in Section \ref{Sec8}.

Hardt's Slicing Theorem ensures that there are no surprises when taking limits.  Interesting enough, the standard Poincar\'e-Lelong plays an analogous role as Stokes Theorem does for the case of simple transgressions.


 We will be working  with closed analytic subspaces $X$ of an ambient analytic manifold $M$. These closed analytic subspaces when they are pure $p$-dimensional determine currents of (real) dimension $2p$.  The analytic subspaces will not generally be reduced nor irreducible hence their irreducible components $Z\subset X$ will come with multiplicities. These are the "embedded" multiplicities $(e_XM)_Z$ of $M$ along $X$ at $Z$ as defined by Fulton in \cite{F} (Example 4.3.4 ) and algebraically by Samuel in  \cite{Sa}. In other words the analytic current determined by $X$ will be:
\[ [X]=\sum_{Z}(e_XM)_Z[Z]
\]
as the irreducible components $Z$ determine currents $[Z]$ via integration over the regular part, using a well-known result of Lelong. We will occasionally drop $[\cdot]$ from notation, hoping that the context is clear whether we refer to the analytic subspace or the current it determines, or both. The inclusion symbols $\subset,\supset$ are occasionally used with the meaning "naturally embedded as  complex spaces" while $\cap$ will sometimes mean the relevant fiber product in the analytic category.

 It is probably apparent that this paper has been influenced by the circle of ideas developed in  \cite{HL1,HL2,HL2ii,HL3}.  We would like to thank Reese Harvey for  interesting conversations about this project.

\section{Bott-Chern Duality and the Chow groups}\label{sec1}
 
 We start by reviewing some more or less well-known facts. Let $M$ be a  smooth, oriented manifold of dimension $n$. The space of currents of degree $k$ is denoted by $\mathscr{D}'_k(M)$. One has a chain complex $(\mathscr{D}'_*(M),d)$ where the convention is that $d$ is the adjoint to the exterior derivative with no sign attached.  We  recall that Poincar\'e Duality  takes up the following form. 
\begin{theorem}\label{T0} The natural inclusion chain inclusion map $\Omega^*(M)\hookrightarrow \mathscr{D}'_{n-*}(M)$ commutes (up to a sign) with the  chain differentials and is a quasi-isomorphism.

Moreover, if $\Phi$ is any family of (closed) supports then the same map is a quasi-isomorphism between $\Omega^*_{\Phi}(M)$ and $(\mathscr{D}_{n-*}')^{\Phi}(M)$, the space of forms, respectively currents with support in $\Phi$.
\end{theorem}
This statement follows immediately from the process of regularization of currents \cite{Fe}. These isomorphisms do not require that $M$ be of finite type. 

The standard version of Poincar\'e Duality,  found in standard Differential Geometry textbooks   is recovered by combining Theorem  \ref{T0} with the Homological Duality:
\[ H_k(\mathscr{D}'_{*}(M))\simeq (H^k_{\dR,\cpt}(M))^*
\]
which also holds in general.  Notice though that the analogous isomorphism for compact supports $H_{k}(\mathscr{E}'_{*}(M))\simeq (H^k_{\dR}(M))^*$ will not hold for non-finite type manifolds as the example of a countable union of disjoint real lines shows. We look at Theorem \ref{T0} as putting cohomology objects like  differential forms  in duality to homological objects like oriented, closed submanifolds and their generalizations like closed rectifiable or flat currents via  equalities of type $\omega-Z=dT$.

If $M$ is a complex manifold then $\Omega^k(M)=\bigoplus_{p+q=k} \Omega^{p,q}(M)$ and $d=\partial+\bar{\partial}$ where $\partial:\Omega^{p,q}(M)\ra \Omega^{p+1,q}(M)$ and $\bar{\partial}:\Omega^{p,q}(M)\ra \Omega^{p,q+1}(M)$.   One defines the Bott-Chern cohomology groups 
\[H^{p,q}_{BC}(M):=\frac{\Ker{\partial}\cap \Ker {\bar{\partial} } }{\Imag \partial\bar{\partial}}\] and also, by using the dual $\partial$ and $\bar{\partial}$ operators the Bott-Chern (currential) homology groups
 \[H_{p,q}^{BC}(M):=H_{p,q}^{BC}(\mathscr{D}'_{*,*}(M)).\]  The holomorphic counterpart to Theorem \ref{T0} has a less than straightforward proof. 
\begin{theorem}[Bigolin-Demailly] \label{T1} Let $M$ be a complex manifold of dimension $n$. The inclusion map $\Omega^{p,q}(M)\ra \mathscr{D}'_{n-p,n-q}(M)$ induces isomorphisms:
\[ H_{BC}^{p,q}(M)\simeq H_{n-p,n-q}^{BC}(M).
\]
This isomorphism stays true when using compact supports.
\end{theorem}
\begin{proof} Both statements are contained in Proposizione 2.2 from \cite{Big}. The statement without restriction on supports can be found also in \cite{Dem}, consequence of  Lemma VI.12.1.
\end{proof}
The equalities of type $\omega-Z=\partial\bar{\partial}T$ appearing in this article can therefore be interpreted via Theorem \ref{T1} as saying that $\omega$ and $Z$ are Bott-Chern dual to each other. 

\vspace{0.5cm}

For the rest of the section we explain another reason  why one might be interested in the homology groups $H_{p,q}^{BC}(M)$ by relating them with the Chow groups of algebraic/analytic geometry.

Any complex manifold $M$ is naturally a reduced complex analytic space and therefore it has  well defined analytic Chow groups $\bA_*^{\an}(M)$. A holomorphic $k$-cycle on $M$ is a linear combination (with integer coefficients) of irreducible analytic subspaces of dimension $k$.

The rational equivalence  between two holomorphic $k$-cycles $[Z_1]$ and $[Z_2]$ is defined as follows. We say $Z_1\sim Z_2$ if there exist $k+1$-holomorphic cycles $V_1, \ldots V_l$ in $M\times \bP^1$  such that
\begin{itemize}
\item[(i)] the restrictions of the  projection $\pi_2: V_i\ra \bP^1$ are all dominant;
\item[(ii)] $[Z_1]-[Z_2]=\sum_{i=1}^l[V_i(0)]-[V_i(\infty)]$, where $V_i(0)=\pi_2^{-1}([1:0])$, $V_i(\infty):=\pi_2^{-1}([0:1])$.
\end{itemize}

It is well known (see Fulton \cite{F} Section 1.6) that this definition of rational equivalence coincides with the standard one via divisors and rational maps. The analytic Chow group $\bA_k^{\an}(M)$ is defined as $k$-cycles modulo rational equivalence.

Notice that holomorphic cycles are particular types of currents of locally finite mass. This suggests we consider $H^{BC}_{k,k}(M)$, the Bott-Chern group of currents of bidimension $(k,k)$. Notice that there exists a natural map:
\begin{equation} \label{kk}H^{BC}_{k,k}(M)\ra H_{2k}^{\sing}(M)
\end{equation}
since every $\partial \bar{\partial}$-exact  is also $d$-exact. It turns out that the natural map $\bA_k(M)\ra H_{2k}^{\sing}(M)$ factors out through this one.

\begin{lem}\label{prefL} There exists a natural map:
\[\bA_k^{an}(M)\ra H^{BC}_{k,k}(M),\qquad [Z]\ra [Z].
\]
Moreover, this map commutes with proper push-forward.
\end{lem}
\begin{proof} A holomorphic $k$-cycle is $d$-closed hence $\partial$ and $\bar{\partial}$ closed. On the other hand, by the classical Poincar\'e-Lelong the difference
\[[V_i(0)]-[V_i(\infty)]
\] 
is $\partial\bar{\partial}$-exact.
\end{proof}
One can be even more precise, although we will not need this here. The group $H^{BC}_{k,k}(M)$ we considered above has real coefficients. But one can look at the group $H^{BC,\flat}_{k,k}(M;\bZ)$ generated by closed $(k,k)$ locally integrally flat currents, modulo the $\partial{\bar{\partial}}$-exact ones. The image of the above map lands in $H^{BC,\flat}_{k,k}(M;\bZ)$.

Suppose $L\ra M$ is a holomorphic line bundle over $M$. Then Fulton  defines in Ch.2 of \cite{F} the operation of "capping" with $c_1(L)$ on the algebraic Chow group that can be replicated in the analytic context under certain conditions. 

Recall the standard map between Cartier divisors  and isomorphism classes of line bundles:
\[\Div(M)\ra \Pic(M).
\]
This map is almost always surjective in the algebraic context, for example when $M$ is an algebraic variety this is true (Prop. 6.15 in \cite{Ha}). However this need not be true in the analytic context even when $M$ is regular, the reason being  the possible absence of meromorphic sections of a given line bundle $L$. We will therefore restrict the attention for the rest of this section to \emph{projective}  complex manifolds $M$ for which the existence of meromorphic sections is guaranteed  (see the application to Theorem B at page 161 in \cite{GH}).

 Then
 \[c_1(L)\cap:\bA_{*}^{an}(M)\ra \bA_{*-1}^{an}(M)\]
 is defined as follows. Let $D$ be a Cartier divisor such that $L\simeq \mathcal{O}(D)$. We will denote by $[D]$ the corresponding Weil divisor determined by $D$. Let $V\subset M$ be an irreducible analytic subvariety of dimension $k$.
 
  If $V\not\subset\supp D$, then $D\cap V$ makes sense as a Cartier divisor over $V$, i.e. as the restriction $D\bigr|_{V}$. As such it has an associated Weil divisor $\left[D\bigr|_{V}\right]\in \bA^{an}_{k-1}(V)$ and hence it determines via the natural map $\bA_*^{an}(V)\ra \bA_*^{an}(M)$ an element in $\bA^{\an}_{k-1}(M)$. This will be $c_1(L)\cap [V]$.
  
  If $V\subset \supp{D}$, then consider the restriction $L\bigr|_{V}$. Since $V$ is projective we have that there exists a Cartier divisor $C$ on $V$ such that $L\bigr|_V\simeq \mathcal{O}(C)$. In this case, define $c_1(L)\cap [V]$ as $[C]$.
  \vspace{0.3cm}

A Hermitian metric on $L$ determines a connection $\nabla^L$ and hence by standard Chern-Weil theory a closed $(1,1)$-form $c_1(\nabla^{L})$. We take $c_1(\nabla^{L})$ to mean the integral form.  This of course induces an operation
\[c_1(\nabla^{L})\wedge:H^{BC}_{k,k}(M)\ra H^{BC}_{k-1,k-1}(M),\qquad T\ra c_1(\nabla^{L})\wedge T.
\]

The next result while not used in the rest of the article is illustrative of the correspondence between known results in algebraic geometry and certain applications of double transgressions we give here.
\begin{prop}\label{chbch} Let $M$ be projective. The following diagram commutes
\[\xymatrix{\bA_{k}^{an}(M)\ar[rr]^{c_1(L)\cap }\ar[d]&& \bA_{k-1}^{an}(M) \ar[d]\\ 
H^{BC}_{k,k}(M)\ar[rr]^{c_1(\nabla^{L})\wedge\quad} && H^{BC}_{k-1,k-1}(M)}.
\]
 The vertical arrows are the maps of  Lemma \ref{prefL}.
\end{prop}
\begin{proof}  Every holomorphic line bundle on $M$ can be written non-canonically as:
\begin{equation}\label{eqtln} L\simeq L_1\otimes L_2^*
\end{equation}
where $L_1$ and $L_2$ are \emph{effective} line bundles, i.e.  $L_i\simeq \mathcal{O}(D_i)$ for some $D_i$ effective divisors.  Hermitian metrics on two of the three line bundles $L,L_1,L_2$ naturally determine a metric on the remaining line bundle. Hence we can choose any metric on $L_1$ and fix the metric on $L_2$ via (\ref{eqtln}). We will therefore have three Chern connections $\nabla^{L}$, $\nabla^{L_1}$ and $\nabla^{L_2}$ and then one checks the equality of forms:
\begin{equation}\label{algeq1} c_1(\nabla^L)=c_1(\nabla^{L_1})-c_1(\nabla^{L_2})
\end{equation}

 One has (Prop 2.5 in \cite{F}, (b) and (e)) that
\begin{equation}\label{algeq2} c_1(L^*)\cap\cdot =-c_1(L)\cap\cdot
\end{equation}
 \begin{equation} \label{algeq3} c_1(L_1) \cap c_1(L_2)\cap \cdot=c_1(L_2) \cap c_1(L_1)\cap \cdot
 \end{equation}
 
 It follows from \eqref{algeq1},  (\ref{algeq2}) and  (\ref{algeq3}) and their analogues for forms that it is enough to prove the general statement for effective line bundles $L$. Each such line bundle has a holomorphic section $s:M\ra L$ which is naturally associated to some gluing data $g_{\alpha\beta}$ that describes $L$.

The zero   locus of $s$ is the Weil divisor associated to $L$, call it $W$. Then use the Poincar\'e-Lelong formula of Griffiths-King
\begin{equation}\label{leq10} c_1(\nabla^L)-W=\frac{i}{\pi}\bar{\partial}{\partial} \log|s|.
\end{equation}
 Take now any irreducible subvariety $V$ of $M$ such that $V \not\subset W$. Then we can intersect the equality of flat currents (\ref{leq10}) with $V$.
\[c_1(\nabla^L)\wedge V -W\wedge V=\frac{i}{\pi}\bar{\partial}{\partial} \left(\log{\left|s_{|_{V}}\right|}\right)
\]
On the other hand, $W\wedge V=W\cap V$ is the analytic cycle that represents $c_1(L)\cap V$ and this finishes the proof when $V \not\subset W$.

If $V\subset W$, then consider $L\bigr|_{V}$. If $V$ is non-singular then Poincar\'e-Lelong for a section of $L\bigr|_{V}$ gives what is needed. If $V$ is singular, we need to pass to a desingularization $\Pi:\tilde{V}\ra V$
and use Poincar\'e-Lelong for a section of $\Pi^*L$ and push-forward the resulting double transgression from $\tilde{V}$ to $M$.
\end{proof}
We remark that wedging with the first Chern class is a well-defined operation on the Bott-Chern groups of a complex manifold $M$, irrespective of whether $M$ is projective or not.

The Gysin isomorphism is well-known for the algebraic Chow groups, or for deRham cohomology. It might come as no surprise that it holds also for Bott-Chern cohomology groups. This will be the first application of double transgressions in the next Section.

\section{The homogeneous action and the Gysin isomorphism}

Let $\pi:E\ra M$ be a holomorphic vector bundle and $\bP^1$ be the projective line. In order to avoid introducing (essentially) unnecessary notation we will use the same greek letter $\pi$ for other bundle projections to $M$ that are related to $E$, such as $\bP(\bC\oplus E)\ra M$ and $\bP(E)\ra M$. 

\vspace{0.3cm}

 Denote by $\Bl_{\infty\times [0]}(\bP^1\times E)$ the blow-up of $\{[0:1]\}\times [0]=:\infty\times [0]$ where $[0]$ is the zero section of $E$. It has an explicit description as a regular subspace of $\bP^1\times E\times_M\bP(\bC\oplus E)$\footnote{The notation $\bC$ will alternatively represent the trivial line bundle  and the field of complex numbers}:
\begin{equation}\label{blex}\Bl_{\infty\times [0]}(\bP^1\times E)=\{([\mu:\lambda],v,[\theta:w])\in \bP^1\times E\times_M\bP(\bC\oplus E)~|~(\mu,\lambda  v) \wedge(\theta,w)=0\}
\end{equation}
with the restriction of $\pi_{1,2}$,  the projection onto the first two coordinates, being the blow-down map.

Let 
\[F_{[\mu:\lambda]}:=\pi^{-1}_1\{[\mu:\lambda]\}\]
 be the fiber of the  projection onto the first coordinate. For $\mu\neq 0$, one has $[\theta:w]=[\mu:\lambda v]$ and therefore $\pi_{2,3}\bigr|_{F_{[\mu:\lambda]}}$ is a biholomorphism onto $\{(v,[1:\lambda/\mu v])\}$ which in turn is biholomorphic with $E$ via projection onto the first component.

 For $\mu=0$ one can decompose 
\begin{equation}
\label{21} F_{[0:1]}=\Ex_{\infty\times[0]}(\bP^1\times E)\cup \Bl_{[0]}(E)
\end{equation}
into the exceptional divisor of the blow-up 
\[ \Ex_{\infty\times [0]}(\bP^1\times E)=\{[0:1]\}\times[0]\times_M\bP(\bC\oplus E)
\]
and the blow-up $\Bl_{[0]}(E)\subset E\times_M \bP(E)$ of the zero section of $E$ seen inside $E\times_M\bP(\bC\oplus E)$ via the natural inclusion
\[E\times_M\bP(E)\hookrightarrow E\times_M\bP(\bC\oplus E).
\]
To see that ({\ref{21}}) holds, separate into $v=0$ which corresponds to $\Ex_{\infty\times[0]}(\bP^1\times E)$ and $v\neq 0$. For $v\neq 0$, one gets that $[\theta:w]=[0:v]$. The closure gives $\Bl_{[0]}(E)$. Notice also that
\[ \Ex_{\infty\times [0]}(\bP^1\times E)\cap \{[0:1]\}\times \Bl_{[0]}(E)=\{[0:1]\}\times\Ex_{[0]}(E).
\]

\vspace{0.3cm}

We will switch now to currents and recall   the following version of Poincar\'e-Lelong.

\begin{theorem}\label{PL1} Let $N$ be a connected, complex analytic manifold and let $f:N\ra \bP^1$, $f:=[f_0:f_1]$ be a non-constant, holomorphic map where $(f_0,f_1)$ is a pair of  $\bC$-valued functions that do not vanish simultaneously anywhere. Then 
\begin{equation}\label{neq10} f^{-1}[1:0]-f^{-1}[0:1]=\frac{i}{\pi}\partial \bar{\partial} \log\left|\frac{f_1}{f_0}\right|
\end{equation}
\end{theorem}
\begin{proof} Let $U_0\cup U_1$ be the usual covering with affine charts of $\bP^1$. Then the identity (\ref{neq10}) reduced to  $f^{-1}(U_0)$ and to $f^{-1}(U_1)$ is a consequence of  Theorem \ref{t1} for $L=\bC$. Use then the sheaf property of currents to deduce it is true on $N$. 
\end{proof}
\begin{rem} The current determined by $\log\left|\frac{f_1}{f_0}\right|$ should actually be written $\log\left|\frac{f_1}{f_0}\right|\cdot [N]$ as it is an $(n,n)$-current, but we stick to the standard convention (see page 388 in \cite{GH}). The definition of the current $\partial \bar{\partial}(g\cdot N)$ is, again using standard conventions
\begin{equation}\label{srem}\partial \bar{\partial}(g\cdot N)(\omega):=\int_Ng\partial \bar{\partial}\omega,\qquad \omega \in \Omega^{n,n}_{\cpt}(N).
\end{equation}
One could argue that the  appropriate definition for  $\partial \bar{\partial} T$ for a current $T$ should be $ \partial (\bar{\partial} T)$, i.e.
\[ \partial \bar{\partial} T(\omega):=T(\bar{\partial}\partial\omega).
\]
This will introduce a sign in (\ref{srem}).   The reason for the use of (\ref{srem}) is that if $g$ is a smooth function, then $\partial \bar{\partial}(g\cdot N)=(\partial \bar{\partial}g)\wedge N$. We will therefore define  in this article
\[\partial \bar{\partial} T=-\partial (\bar{\partial} T).\] 

Since we are at this,  let us mention the notational convention that 
\[\omega\wedge T(\gamma):=T(\omega\wedge \gamma)\qquad T\wedge \omega(\gamma):=T(\gamma\wedge\omega).\] Of course, if $\omega$ is of even degree then both are equal.
\end{rem}

\begin{rem}  It is well-known that $\log{|f|}$ is locally integrable whenever $f$ is meromorphic (see Lemma 1.4 in \cite {GK} or \cite{Ch} page 213).
\end{rem}
The points $[1:0]$ and $[0:1]$ have nothing special.  In fact, by composing $f:N\ra \bP^1$ with a linear biholomorphism $A$ of $\bP^1$ which  takes a point $[\mu_0:\lambda_0]$  to $[1:0]$ and $[\mu_1:\lambda_1]\neq [\mu_0:\lambda_0]$ to $[0:1]$, e.g. $A=\left(\begin{array}{cc} \lambda_1&-\mu_1\\
 -\lambda_0&\mu_0\end{array}\right)$, 
we get that
 \begin{equation}\label{eqg}g:= A\circ f\end{equation}  satisfies $g^{-1}([1:0])=f^{-1}([\mu_0:\lambda_0])$ and $g^{-1}([0:1])=f^{-1}([\mu_1:\lambda_1])$. Unfortunately the correspondence 
 \[([\mu_1:\lambda_1], [\mu_0:\lambda_0])\ra A
 \]
 is not well-defined but one can make a rather canonical choice if we fix the charts. Hence we take $A=\left(\begin{array}{cc} 1& -\theta\\
  -\gamma &1 \end{array}\right)$ and obtain  a ``dynamical" Poincar\'e-Lelong by applying Theorem \ref{PL1} to $g$ and a bit more.
\begin{theorem}\label{PLmn} Let $f:N\ra \bP^1$, $f=[f_0:f_1]$ be holomorphic, non-constant. Then for every $(\gamma,\theta)\in \bC^2$, every constant multiple of the  function 
\[ h^f(\gamma,\theta):=\frac{f_1-\gamma f_0}{f_0-\theta f_1}\]  satisfies
\begin{equation}\label{P1PL} f^{-1}([1:\gamma])-f^{-1}([\theta:1])=\frac{i}{\pi}\partial \bar{\partial} \log{|h^f(\gamma,\theta)|}
\end{equation}
Formula (\ref{P1PL}) varies continuously with $(\theta,\gamma)$ in the topology of locally flat currents and moreover 
\[(\gamma,\theta)\ra  \log|h^f(\gamma,\theta)|
\]
 is continuous on every compact of $N$ in the mass norm topology, i.e. continuous in $L^1_{\loc}(N)$.
\end{theorem}
\begin{proof} 

The continuity of the slices in the flat topology follows from the classical Slicing Theorem of Hardt (see Theorem \ref{Hardt} below or \cite{Har}) hence the left hand side of (\ref{P1PL}) varies continuously. 

 For the $L^1_{\loc}$ continuity of $\log|h^f|$ it is enough to prove the $L^1_{\loc}$ continuity of
\[ \gamma\ra \log|{ f_1-\gamma f_0}|
\]
and by an obvious substitution it is enough to prove the continuity  for $\gamma =0$, i.e. that the following holds
\begin{equation}\label{hleq} \lim_{\gamma\ra 0}\int_{K}\big|\log|\gamma f_0(x)-f_1(x)|-\log|f_1(x)|\big|~dx=0.
\end{equation}
for a compact $K$ in $N$. It is also enough to prove this for small polydisks in $\bC^n$. The "problematic" polydisks $\Delta'$ are those centered at points  in $p\in f^{-1}(0)$. Taking $\gamma$ small and $a$ a point outside the common zero locus of $ \gamma f_0-f_1$, and a polydisk $\Delta\supset \Delta'$ centered at $a$, Jensen's inequality (\cite{Gu}, Theorem 8 Cap A) for $\Delta(a)$ shows that one can apply Lebesgue dominated convergence in order to justify (\ref{hleq}).

\end{proof}

\begin{theorem}[Hardt,\cite{Har}] \label{Hardt}Let $N$ be a real analytic manifold and $f:N\ra \bR^n$ be analytic. Let $T$ be an analytic current in $N$, $\dim{T}=t$. Let 
\[Y:=\{y\in \bR^n~|~\dim {(f^{-1}(y)\cap  \supp{T})}\leq t-n,\;\; \dim{(f^{-1}(y)\cap \supp{\partial T})}\leq t-n-1\}.\] 
Then the function
\[ Y\ni y\longrightarrow \langle T,f,y\rangle
\]
is well-defined and continuous in the topology of locally flat currents. 
\end{theorem}

\vspace{0.3cm}

We apply now Theorem \ref{PL1} to  $N=\Bl_{\infty\times [0]}(\bP^1\times E)$,  $f=\pi_1=[\pi_1^0:\pi_1^1]$,  the restriction to $\Bl_{\infty\times [0]}(\bP^1\times E)$  of the projection onto $\bP^1$ and $T=\Bl_{\infty\times [0]}(\bP^1\times E)$.  The dimensional condition of Hardt's Theorem is fulfilled everywhere. Hence we get
\begin{prop}\label{cormain} For every $\gamma\in \bC$, the following equality of currents holds within the complex manifold $\Bl_{\infty\times [0]}(\bP^1\times E)$:
\begin{equation}\label{ftrf} F_{[1:\gamma]}-F_{[0:1]}=\frac{i}{\pi}\partial \bar{\partial}\log{| h^{\pi_1}_0(\gamma)|}
\end{equation}
where $h^{\pi_1}_0(\gamma):=h^{\pi_1}(\gamma,0)=\frac{\pi_1^1}{\pi_1^0}-\gamma$. 

Moreover, $F_{[0:1]}=\lim_{\theta\ra \infty}F_{[1:\theta]}$ as locally flat currents and (\ref{ftrf}) is the limit of  the corresponding double transgression formulas (\ref{P1PL}).

\end{prop}
As a consequence of the last proposition we already get a double transgression formula for forms on $\bP(\bC\oplus E)$.
\begin{cor}\label{core1} Let $\omega\in \bP(\bC\oplus E)$ be a smooth $\partial$ and $\bar{\partial}$ closed form. Let $\pi:\bP(\bC\oplus E)\ra M$ and $\pi^E:E\setminus\{0\}\ra \bP(E)$ be the canonical projections. Then the following equality of currents holds on $E$:
\begin{equation}\label{ftrf2} \omega\bigr|_{E}-\iota_*(\pi_*(\omega)\wedge[M])-(\pi^{E})^*(\omega\bigr|_{\bP(E)})\footnote{The current $(\pi^{E})^*(\omega\bigr|_{\bP(E)})$ a priori well-defined on $E\setminus \{0\}$ extends to a current on $E$ via (\ref{exteq1}).}=\partial\bar{\partial}[T(\omega)].
\end{equation}
where $M$ is the zero section, $\iota:M\ra E$ the zero section embedding, $\pi_*$ is integration along the fiber (push-forward) and $T(\omega)$ is a current of locally finite mass.
\end{cor}
\begin{proof} Take $\varphi_z:E\ra \bP(\bC\oplus E)$ to be the family of embeddings $\varphi_z(v)=[1:zv]$. Then 
\begin{equation}\label{someeq}\varphi_z^*\omega(\eta):=\int_E\varphi_z^*\omega\wedge\eta=F_{[1:z]}((\pi_2)^*\omega\wedge \pi_1^*(\eta))=(\pi_1)_*(\pi_2^*\omega\wedge F_{[1:z]})(\eta)
\end{equation}
for any form $\eta$ with compact support on $E$. Here $\pi_1,\pi_2$ are the projections onto the components of $E\times_M\bP(\bC\oplus E)$. Relation (\ref{someeq}) is clear in view of the fact that  $F_{[1:z]}=\{(v,[1:zv])\}$ is the graph of $\varphi_z$. We can  therefore wedge (\ref{ftrf}) for $\gamma=1$ with $\pi_2^*\omega$, compute  $\pi_2^*\omega\wedge F_{[0:1]}$ and then apply $(\pi_1)_*$. 

First we evaluate
\[ [0]\times_M\bP(\bC\oplus E)(\pi_2^*\omega\wedge\pi_1^*\eta)=\bP(\bC\oplus E)\left(\omega\wedge \pi^*\left(\eta\bigr|_{[0]}\right)\right)=\pi_*(\omega)\wedge [M](\eta\bigr|_{[0]}).
\]
where $\bigr|_{[0]}$ is restriction to the zero section. Then, since $\Bl_{[0]}(E)\subset E\times_M \bP(E)$:
\begin{equation}\label{exteq1} \Bl_{[0]}(E)\left(\pi_2^*\left(\omega\bigr|_{\bP(E)}\right)\wedge\pi_1^*\eta\right)=\Bl_{[0]}(E)\setminus \Ex_{[0]}(E)\left(\pi_2^*\left(\omega\bigr|_{\bP(E)}\right)\wedge\pi_1^*\eta\right)=\end{equation}\[=[E\setminus\{0\}](\zeta^*\pi_2^*\left(\omega\bigr|_{\bP(E)}\right)\wedge\eta)=\zeta^*\pi_2^*(\omega\bigr|_{\bP(E)})(\eta)=(\pi^E)^*\left(\omega\bigr|_{\bP(E)}\right)(\eta).
\]
where $\zeta:E\setminus\{0\}\ra E\setminus\{0\}\times_M\bP(E)$, $\zeta(v)=(v,[v])$. 

We therefore get  (\ref{ftrf2}) with 
\[T(\omega)=(\pi_2)_*\left(\log|h^{\pi_1}_0(1)|\pi_3^*\omega\right).\]
where $\pi_2$ and $\pi_3$ are the projections of $\Bl_{\infty\times[0]}(\bP^1\times E)$ onto $E$ and $\bP(\bC\oplus E)$. 

As a bonus we also get that
\[ \lim_{z\ra\infty}\varphi_z^*\omega=\pi_*(\omega)\wedge[M]+(\pi^{E})^*(\omega\bigr|_{\bP(E)})
\]
consequence essentially of Hardt's theorem.
\end{proof}

\begin{rem} We used a closed form $\omega$ so that on the r.h.s. we obtain a $\partial\bar{\partial}$ exact current, as Leibniz will then imply that $\partial\bar{\partial}(T\wedge \omega)=(\partial\bar{\partial}T)\wedge\omega$. Then holomorphic push-forward commutes with $\partial$ and $\bar{\partial}$. Nevertheless, one can derive from (\ref{ftrf})   double transgression formulas for \emph{all forms}, the price to pay being  three more terms on the r.h.s. of (\ref{ftrf2}).
\end{rem}

 We have the following consequence of this Corollary and Bott-Chern Duality Theorem \ref{T1}.
 
  \begin{prop}\label{Gysin1} Let $M$ be a compact, complex manifold and $E\ra M$ be a holomorphic vector bundle of rank $k$.  Then  the Gysin maps $\pi_*:\Omega^{p,q}_{\cpt}(E)\ra \Omega^{p-k,q-k}(M)$ given by integration along the fiber induce isomorphisms:
 \[ H^{p,q}_{BC,\cpt}(E)\simeq H^{p-k,q-k}_{BC}(M).
 \] 
 \end{prop}
 \begin{proof} Take $\tau^E$ to be a Thom form for $E$, i.e. a closed form of degree $2k$ with  compact support such that $\int_{E/M}\tau^E\equiv1$. Clearly $\tau$ can be considered as a form on $\bP(\bC\oplus E)$. Applying Corollary \ref{core1} to $\omega=\tau^E$ we get that $\tau^E=[M]$ in $H_{*,*}^{BC}(E) $. Since $M$ has bi-dimension $(n,n)$ it follows that we can choose $\tau^E$ of bidegree $(k,k)$, i.e. discard the components of bidegree $(k-i,k+i), \; i\neq 0$ as they will be exact and hence also of zero integral. With this choice, let
  \[\psi:\Omega^{p-k,q-k}(M)\ra \Omega^{p,q}_{\cpt}(E),\qquad \psi(\omega)= \pi^*\omega\wedge \tau^{E}.
 \]
 Clearly $\pi_*\circ \psi=\id$ directly at the level of forms. On the other hand, as currents in $E$ for a form $\omega$ with compact support in $E$ we have:
 \[ \pi^*(\pi_*\omega)\wedge [M]=\pi_*(\omega)\wedge [M]
 \]  
Corollary \ref{core1} gives thus  the following equality in $H^{BC}_{*,*,\cpt}(E)$:
 \[ \omega=\pi_*(\omega)\wedge [M]= \pi^*(\pi_*\omega)\wedge [M]= \pi^*(\pi_*\omega)\wedge\tau^E=\psi\circ \pi_*(\omega).
 \]
 The conclusion then follows by Bott-Chern duality.
  \end{proof}
  
The formula (\ref{ftrf2}) is valid not only for smooth forms on $\bP(\bC\oplus E)$ but also for most analytic currents in $\bP(\bC\oplus E)$, i.e. currents induced by closed analytic subvarieties.  We have thus another consequence of Proposition \ref{cormain} using the same notations as in Corollary \ref{core1}.
\begin{cor}\label{core2} Let $R$ be a complex analytic current in $\bP(\bC\oplus E)$ of dimension $d$ such that the analytic space $R\cap \bP(E)$ has the expected dimension $d-1$ or is empty. Then the following equality of currents holds on $E$:
\[ R\bigr|_{E}-\iota_*\pi_*R-(\pi^E)^*(R\wedge \bP(E))\footnote{Pull-back via $\pi^E$ delivers a current in $E\setminus\{0\}$ increasing dimension by $1$. A posteriori, it extends to $E$. }=\partial\bar{\partial}[T(R)]
\]
where $T(R)$ is a  current of locally finite mass.
\end{cor}
\begin{proof} We use the same notation as in  Corollary \ref{core1}. First, the (pull-back) current $\pi_2^*R$ is well-defined and represented by the analytic space $E\times_MR$. Then one can check easily that $(\pi_1)_*(\pi_2^*R\wedge F_{[1:z]})$ is a well-defined analytic current on $E$ which coincides with $\varphi_z^*R=\left[(\tilde{\varphi}_{z^{-1}})_*R\right]\bigr|_{E}$, where for $z\neq0$, $\tilde{\varphi_z}:\bP(\bC\oplus E)\ra \bP(\bC\oplus E)$ is the bijection $[1:v]\ra [1:zv]$. We can therefore wedge (\ref{ftrf}) with $\pi_2^*R$ for $\gamma=1$ and apply $(\pi_1)_*$.

 If we let $\psi:[0]\times_{M}\bP(\bC\oplus E)\ra \bP(\bC\oplus E)$ to be the restriction of $\pi_2$ then $\psi$ is a biholomorphism and $\pi_1\circ \psi^{-1}=\iota\circ \pi$ where $\pi:\bP(\bC\oplus E)\ra M$ is the projection and $\iota:M\ra E$ is the zero section. This is used to identify the current $(\pi_1)_*([0]\times_M[R]):=(\pi_1)_*((E\times_MR)\wedge [0]\times_M\bP(\bC\oplus E))$ with $\pi_*R$. For the other wedge product, i.e. $\pi_2^*R\wedge F_{[0:1]}$ let $R':=R\wedge \bP(E)$. Then
 \[\pi_2^*R\wedge (E\times_M\bP(E)) = E\times_M R'=\pi_2^*(R')\]
 Then since $\Bl_{[0]}(E)\subset E\times_M\bP(E)$ it follows that
 \[\pi_2^*R\wedge \Bl_{[0]}(E)=\pi_2^*R'\wedge\Bl_{[0]}(E).
 \]
Recall that $\pi:\bP(E)\ra M$ and therefore, just as before
 \begin{equation}\label{ctrf}(\pi_1)_*(\pi_2^*R'\wedge([0]\times_M\bP(E)))=(\pi_1)_*([0]\times_MR')=\pi_*(R')
 \end{equation}
On the other hand $\pi_1\circ\zeta=\id_{E\setminus[0]}$ and therefore 
 \[(\pi_1)_*(\pi_2^*R'\wedge(\Bl_{[0]}(E)\setminus \Ex_{[0]}(E))=\zeta^*(\pi_2^*R'\wedge(\Bl_{[0]}(E)\setminus \Ex_{[0}](E))=\zeta^*(\pi_2^*R'\wedge \Gamma_{\pi^E})=(\pi^E)^*(R').\]
  By hypothesis, the dimension of the current appearing in (\ref{ctrf}) is smaller than the dimension of $R$ and so the term (\ref{ctrf}) can be discarded.
 
 We have that $T(R)=(\pi_2)_*\left(\log|h^{\pi_1}_0(1)|\pi_3^*R\right)$ where the projections are now defined on $\Bl_{\infty\times [0]}(\bP^1\times E)$. We used here the fact that when $f$ is holomorphic then, $\log|f|$ is locally integrable when restricted to any complex analytic subset. 
  \end{proof}
 
 Note that the dimensional condition also implies via Hardt Theorem that:
 \[ \lim_{z\ra\infty} \varphi_z^*R=\left[\lim_{z\ra 0}(\tilde{\varphi}_{z})_*R\right]\bigr|_{E}=\iota_*\pi_*R+(\pi^E)^*(R\wedge \bP(E)).
 \]

\vspace{0.2cm}

\section{Holomorphic sections} \label{sec3}

Let now $s:M\ra E$ be a non-zero holomorphic section. We are interested in studying  the current induced by the action-graph of $s$, namely:
\[\{(s(m),[1:\lambda s(m)])~|~ m\in M, \lambda\in \bC\}\subset E\times_M\bP(\bC\oplus E)
\]

Denote by $\tilde{S}$  the strict transform of $\bP^1\times s(M)$  in the blow-up $\Bl_{\infty\times [0]}(\bP^1\times E)$.  Since $s$ is non-vanishing,  $\tilde{S}$ is a complex analytic space of dimension $n+1$. 
\vspace{0.3cm}

\begin{lem} Let $\pi_{2,3}:\Bl_{\infty\times [0]}(\bP^1\times E)\ra E\times_M\bP(\bC\oplus E)$ be the projection onto the last two coordinates. 
Then $\pi_{2,3}(\tilde{S})=\overline{\{(s(m),[1:\lambda s(m)])~|~ m\in M, \lambda\in \bC\}}$ is a complex analytic set of dimension $n+1$ in $E\times_M\bP(\bC\oplus E)$.
\end{lem}
\begin{proof} Clearly $\pi_{2,3}$ is a proper map and hence the image of $\tilde{S}$ is a closed analytic set. It is enough then to prove that over an open dense set $D$ of $\tilde{S}$, $\pi_{2,3}\bigr|_{D}$ maps biholomorphically onto an open dense set of $\{(s(m),[1:\lambda s(m)])~|~ m\in M\}.$ But this is obvious by the definition of the strict transform $\tilde{S}$.
\end{proof}
We adapt Proposition \ref{cormain} to the new situation:

\begin{theorem} \label{PL3} The following equality of currents in $\Bl_{\infty\times [0]}(\bP^1\times E)$ holds:
\begin{equation} \label{PL31} F_{[1:\gamma]}\wedge \tilde{S}-F_{[0:1]}\wedge\tilde{S}=\frac{i}{\pi}\partial{\bar{\partial }}[\log{\left|h^{\pi_1}_0(\gamma)\right|\bigr|_{ \tilde{S}}}]
\end{equation}
where the restriction $[\log{\left|h^{\pi_1}_0(\gamma)\right|\bigr|_{ \tilde{S}}}]$ is a well-defined, $\mathcal{H}^{n+1}\llcorner \tilde{S}$-locally integrable function.

The formula varies continuously  with $\gamma\in \bC$ in the locally flat topology of currents. 

\end{theorem}

\begin{proof}  Consider $\tilde{S}$ as a current in $\Bl_{\infty\times [0]}(\bP^1\times E)$. Then the left hand side of (\ref{PL31}) is the difference of slices with $[\mu:\lambda]=[1:\gamma]$
\[ \langle \tilde{S}, \pi_1,[\mu:\lambda]\rangle-\langle \tilde{S}, \pi_1,[0:1]\rangle.
\]
It is easy to see that  Hardt's dimensional transversality condition is fulfilled for both terms. This is clearly true for $\mu\neq 0$ while for $\mu=0$  Proposition \ref{Psp1} is used together with the fact that $\pi_{2,3}$ restricted to any slice is an embedding into $E\times_M\bP(\bC\oplus E).$

 The current $\log{|h^{\pi_1}_0(\gamma)|}\bigr|_{ \tilde{S}}$ is nothing but the intersection or wedge of the flat currents $\log{|h^{\pi_1}(\gamma)}$ and $\tilde{S}$. Formally, since $\tilde{S}$ is both $\partial$ and $\bar{\partial}$ -closed it is to be expected via Leibniz that
\[[\partial\bar{\partial}\log{|h^{\pi_1}_0(\gamma)|]\wedge{ \tilde{S}}}= \partial\bar{\partial}[\log{|h^{\pi_1}_0(\gamma)|\cdot{ \tilde{S}}}].
\]
However, since $\log{|h^{\pi_1}_0(\gamma)|)}$ is not smooth everywhere, this needs a proof. For this we use Theorem 2, page 216 in \cite{Ch}. The only thing one needs to check here is that $h^{\pi_1}_0(\gamma)$ does not vanish identically on $\tilde{S}$. This is straightforward.
\end{proof}

We analyze  in more detail
 \[\tilde{S}_{\infty}:=F_{[0:1]}\wedge \tilde{S}=\langle \tilde{S}, \pi_1,[0:1]\rangle.\]
 
 \vspace{0.1cm}
 
 Let $X:=s^{-1}(0)$, endowed with the complex analytic structure induced by the components of $s$ in a local trivialization. This does not depend on the trivialization. We will  often make no distinction between $X$ and $s(X)\subset [0]$ and therefore we will also write $X=s(M)\cap [0]$.

The blow-up of $X$ in $s(M)$, denoted by $\Bl_X(s(M))$ coincides with the strict transform of $s(M)$ (see for example \cite{EH}, Prop IV-21) in $\Bl_{[0]}(E)\subset E\times_M\bP(E)$.  The projection $\pi:E\ra M$ induces a biholomorphism of pairs of analytic spaces $(s(M),X)\simeq (M,X)$. In fact we have that $\pi_2\bigr|_{\Bl_X(s(M))}$ induces an isomorphism $\Bl_X(s(M))\simeq \Bl_X(M)$  where $\pi_2$ here is the projection onto the second component of $E\times_M\bP(E)$ (see (\ref{blid})). 
 
For the next statement we need the fiberwise cone $C^fA$ of an analytic space $A\subset E\times_M\bP(E)$. This is an analytic subspace of $E\times_M\bP(\bC\oplus E)$ induced by the same equations as $A$. Some rather elementary properties of blow-ups and cones are included in the Appendix \ref{Ap1}.

\begin{prop}\label{Psp1} The following decomposition into closed analytic spaces of dimension $n$ holds
\[ \pi_{2,3}(\tilde{S}_{\infty})=\Bl_X(s(M))\cup C^f\Ex_X(s(M))
\]
with the intersection given by
\[C^f\Ex_X(s(M))\cap \Bl_X(s(M))=\Ex_X(s(M))\subset [0]\times_M\bP(E)
\]
\end{prop}
\begin{proof} First $\pi_{2,3}$ is proper hence the image is an analytic space. Use the decomposition of (\ref{21}). On one hand, we claim that
\begin{equation}\label{neq1} \tilde{S}\cap \Bl_{[0]}(E)=\Bl_{X}(s(M))
\end{equation}
where $\Bl_{X}(s(M))$ is naturally embedded in $\Bl_{[0]}(E)=\{\infty\}\times \Bl_{[0]}(E)$. Indeed, consider $E$ inside $\bP^1\times E$ as $\{\infty\}\times E$.   Let   $X_1:=E$,  $X_2:=\bP^1\times s(M)$, and $W:=\bP^1\times [0]$, be three closed subvarieties of $\bP^1\times E$. We need to prove that as complex spaces
\[ \Bl_{X_1\cap W}(X_1)\cap \Bl_{X_2\cap W}(X_2)=\Bl_{X_1\cap X_2\cap W}(X_1\cap X_2).
\] 
All the blow-ups involved come with closed embeddings into $\Bl_{\infty\times[0]}(\bP^1\times E)$ so it is in this sense that the previous equality has to be understood.
The fact that the blow-up process commutes with taking intersections (or more generally fiber products) does not hold in all generality. To see that it does hold in our case we turn this into a statement about strict transforms. We need to prove that
\[ \ST (X_1)\cap \ST (X_2)=\ST(X_1\cap X_2)
\]
where $\ST$ refers to strict transforms in $\Bl_{\infty\times[0]}(\bP^1\times E)$.  Let $Y_1$ and $Y_2$  be the \emph{total transforms}  of $X_1$ and $X_2$ and $\Ex$ be the exceptional divisor of the blow-up of $\infty\times [0]$ in $\bP^1\times E$.

Then 
\begin{equation}\label{eq22} \overline{Y_1\setminus \Ex}\cap \overline{Y_2\setminus \Ex}=\overline{(Y_1\cap Y_2)\setminus \Ex}.
\end{equation}
where closure means the analytic closure (see \cite{Ch}, Corollary 5.1). The inclusion $\supset$ is obvious. If $\Ex$ does not contain any irreducible component of $Y_1\cap Y_2$, the other inclusion holds as well.  Notice that
\[ Y_1\cap Y_2\cap \Ex=\{\infty\}\times (s(M)\cap[0])\times_M\bP(\bC\oplus E) 
\]
while
\[Y_1\cap Y_2=\{\infty\}\times \{(s(m),[\theta:w])~|~m\in M,\; w\in E_m, \; s(m)\wedge w=0\}.
\]
Since $s$ is not identically zero on any of the connected components of $M$, every point in $Y_1\cap Y_2\cap \Ex$ can be reached by a sequence of points from $Y_1\cap Y_2\setminus \Ex$, hence $\Ex$ does not contain any irreducible component of $Y_1\cap Y_2$ and (\ref{neq1}) is justified.

On the other hand, we can write $\tilde{S}\cap \Ex$  as
\[ \Bl_{\infty\times X}(\bP^1\times s(M))\cap\Ex_{\infty\times[0]}(\bP^1\times E)=\Ex_{\infty\times X}(\bP^1\times s(M)),
\]
a natural consequence about exceptional divisors of Proposition IV-21 in \cite{EH}. By Corollary \ref{C1} the last space is the fiberwise cone $C^f\Ex_0(s(M)).$
\end{proof}

 Recall that we want an equality of kernels in $E\times_M\bP(\bC\oplus E)$.
We go back to Theorem \ref{PL3} and push-forward (\ref{PL31}) via $\pi_{2,3}$. Both terms on the left hand side of (\ref{PL31}) lie in different  level sets of $\pi_1$ so $\pi_{2,3}$ does not alter them. We simplify the expression on the right hand side. But first we recall the following.

\begin{definition} A locally rectifiable $n$-dimensional current with $\bR$-coefficients is a current  which in every  covering with relatively compact open subsets, equals a triple $(X, {\overrightarrow{\xi}}_X, f)$ where $X$ is a rectifiable $n$-dimensional set, ${\overrightarrow{\xi}}_X:X\ra \Lambda^nTX$ is an orientation multivector   defined $\mathcal{H}^n\llcorner X$ a.e.  and $f:X\ra \bR$ 
is  an $\mathcal{H}^n\llcorner X$-integrable function. \end{definition}
\begin{prop}\label{pflog} The current $(\pi_{2,3})_*([\log{|h^{\pi_1}_0(\gamma)|\bigr|_{ \tilde{S}}}])$ is locally of finite mass. More precisely, $(\pi_{2,3})_*([\log{|h^{\pi_1}_0(\gamma)|\bigr|_{ \tilde{S}}}])$ is a locally rectifiable $(2n+2)$-dimensional current with $\bR$-coefficients.
\end{prop}
\begin{proof} The projection map $\pi_{2,3}$ is proper, and $\log{|h^{\pi_1}_0(\gamma)|\bigr|_{ \tilde{S}}}$ is a compactly supported current once we restrict it to the preimage of compact subsets of $M$. Moreover $\log{|h^{\pi_1}_0(\gamma)|\bigr|_{ \tilde{S}}}$ on such subsets is of finite mass (as an integrable function along a subvariety). We then use the following basic inequality for push-forward of finite mass currents
\begin{equation}\label{leq1} {\textbf {M}}(f_*T)\leq \Lip\left(f\bigr|_K\right)^k{\textbf{M}}(T),\qquad k=\deg{T},\; \supp{T}\subset K.
\end{equation}
to conclude that the push-forward is also of finite mass (locally).

In order to justify rectifiability, notice that the push-forward $(\pi_{2,3})_*\tilde{S}$ is locally $2n+2$-rectifiable as $\pi_{2,3}\bigr|_{\tilde{S}}$ is a proper, birational map. \end{proof}

Denote by $S_{\gamma}$ the following current in $E\times_{M}\bP(\bC\oplus E)$:
\[ {S}_{\gamma}:=(\pi_{2,3})_*(F_{[1:\gamma}]\wedge \tilde{S})= \left\{\left(s(m),\left[1:\gamma s(m)\right]\right)~\bigr |~m\in M\right\}.
\]

 Now, Proposition \ref{Psp1} makes possible invoking the continuity part of Hardt's Theorem and so   $\langle \tilde{S},\pi_1,[1:\lambda]\rangle\ra\langle \tilde{S},\pi_1,[0:1]\rangle$ when $\lambda\ra \infty$. This is another way of writing
\[ F_{[1:\lambda]}\wedge \tilde{S}\ra \tilde{S}_{\infty}
\]
Therefore,  by the continuity of the proper push-forward induced by $\pi_{2,3}$ we get:
\[\lim_{\lambda\ra \infty} S_{\lambda}=(\pi_{2,3})_*(\tilde{S}_{\infty}).
\]
Taking into account that the holomorphic push-forward commutes with $\partial$ and $\bar{\partial}$ we have thus reached the main result of this section:
\begin{theorem}\label{t4r}   For every $\gamma\in \bC$, there exists a double transgression formula of kernels in $E\times_M\bP(\bC\oplus E)$
\begin{equation}\label{dtker}S_{\gamma}-\lim_{\lambda\ra \infty} S_{\lambda} =\partial\bar{\partial} T(\gamma,s).
\end{equation}
\begin{itemize} 
\item[(i)] The limit on the left hand side of (\ref{dtker}) equals the sum of two currents:
\[ \tilde{s}_{\infty}(M)+C^f\Ex_{X}(s(M))
\]
where  \[\tilde{s}_{\infty}(M):=\{(s(m),[0:s(m)])\in E_m\times \bP(E_m)~|~m\in M\setminus s^{-1}(0)\} \;\mbox{and}\] 
 \[\supp C^f\Ex_{X}(s(M))\subset [0]\times_M\bP(\bC\oplus E).\]
\item[(ii)] The current 
\[T(\gamma,s)=\frac{i}{\pi}(\pi_{2,3})_*\left(\log\left|h^{\pi_1}_0(\gamma)\right|\bigr|_{\tilde{S}}\right)\] is a locally rectifiable  with $\bR$-coefficients of Hausdorff  (real) dimension $2n+2$  and bidimension $(n+1,n+1)$. 
\item[(iii)] $\supp{T(\gamma,s)}\subset s(M)\times_M\bP(\bC\oplus E)$.
\end{itemize} 
\end{theorem}
\begin{proof} The splitting into a sum of two currents  is a trivial consequence  of Proposition \ref{Psp1}. Since $C^f\Ex_{X}(s(M))\cap \Bl_{X}(s(M))=\Ex_{X}(s(M))$ is an analytic subspace of dimension strictly smaller than $\Bl_{X}(s(M))$, it can be ignored when integrating over  $\Bl_X(s(M))$.  Notice that $\tilde{s}_{\infty}(M)$ is the graph of $\pi^E\circ s\bigr|_{M\setminus X}$, the complement of the exceptional divisor in $\Bl_X(s(M))$. 

Item (ii) is the object of Proposition \ref{pflog}, while item (iii) is obvious.
\end{proof}

Let $C_XM$ be the affine normal cone of the analytic space $X$ in $M$. We explain in Appendix A, that $\bP(\bC\oplus C_XM)$ can be naturally seen as an analytic subspace of $\bP(\bC\oplus E)$ while  Proposition \ref{fcprop} identifies 
\[C^f\Ex_X(s(M))\;\mbox{and}\;\bP(\bC\oplus C_XM).\]
via the natural isomorphism $[0]\times_M\bP(\bC\oplus E)\ra \bP(\bC\oplus E)$.

\vspace{0.2cm}

\begin{example} Suppose that the sheaf of ideals defined by the section $s$ is reduced. In other words $X$ is a reduced analytic space. Then the set of regular points  is open and dense in $X$. The following statements are equivalent.
\begin{itemize}
\item[$(i)$] The germ $(X,p)$ admits a regular representative of dimension $j$;
\item[$(ii)$] The ring $\mathcal{O}_{X,p}$ is regular and $\dim\mathcal{O}_{X,p}=j$;
\item[$(iii)$] Given any connection $\nabla$ on $E$, the rank of $(\nabla s)_p$ seen as a linear morphism \linebreak $(\nabla s)_p:T_pM\ra E_p$ is $n-j$.
\end{itemize}
The equivalence $(i)\Leftrightarrow(ii)$ is standard. The equivalence $(ii)\Leftrightarrow (iii)$ is essentially the Rank Theorem and takes into account that $(\nabla s)_p$ for $p\in s^{-1}(0)$ does not depend on $\nabla$ and in local coordinates it becomes the Jacobian matrix of $s$ with respect to the natural basis of $T_pM$ induced by the coordinates.

In this situation, one can prove quite easily, working in local coordinates that $(C_XM)_p$ for $p$ regular coincides with the image of $(\nabla s)_p$, hence $\bP(\bC\oplus C_XM)_p$ is the  projective subspace $\bP(\bC\oplus\Imag (\nabla s)_p)$. 
\end{example}

\vspace{0.3cm}

The archetypal double transgression formula for forms is the next corollary. Particular choices of $\omega$ will deliver important applications in the next sections. 

For $\gamma\in \bC$, let $s_{\gamma}:M\ra \bP(\bC\oplus E)$ be the induced section $m\ra [1:\gamma s(m)]$ whose (fiberwise) graph is $S_{\gamma}$.

\begin{cor}\label{cormt1} Let $\omega$ be a closed form of bidegree $(k,k)$ on $\bP(\bC\oplus E)$ and let $\gamma\in \bC$.
 Then the following double transgression formula holds on $M$:
\begin{equation}\label{genPL} s^*_{\gamma}\omega-s^*_{\infty}\left(\omega\bigr|_{\bP(E)}\right)- Z(s,\omega)=\partial\bar{\partial} T(\gamma,s,\omega)
\end{equation}
where 
\begin{itemize}
\item[(i)]  $s_{\infty}:=\pi^E\circ s\bigr|_{M\setminus X}$ and $s_{\infty}^*\omega$ is a form with locally integrable coefficients\footnote{It can be integrated against any smooth form with compact support on $M$ of complementary degree}  on $M$,
\item[(ii)] $Z(s,\omega)$ is a locally flat current supported on $X$ defined as follows:
 \begin{equation}\label{eqZ} Z(s,\omega) :=\pi_*\left(\omega\wedge \bP(\bC\oplus C_{X}M)\right).\end{equation}   In (\ref{eqZ}) the analytic current $\bP(\bC\oplus C_{X}M)$ lies in $\bP(\bC\oplus E)$ and $\pi:\bP(\bC\oplus E)\ra M$ is the projection.
 
 \item[(iii)] $T(\gamma,s,\omega)$ is a flat current of bidegree $(k-1,k-1)$ (bidimension $(n+1-k,n+1-k)$.
\end{itemize}
\end{cor}
\begin{proof} One has via a change of variables:
\begin{equation}
\label{fdteq1} s_{\gamma}^*\omega(\eta):= \int_{M}s^*_{\gamma}\omega\wedge\eta=\pi_*(\pi_2)_*(\pi_3^*\omega\wedge {S}_{\gamma})(\eta)
\end{equation}
where  $\pi_2,\pi_3$ are projections onto the factors of $E\times_{M}\bP(\bC\oplus E)$\footnote{Recall that there exist already a $\pi_1$, the projection onto the first factor in $\Bl_{\infty\times [0]}(\bP^1\times E)$}.  Since $\pi_2$ is proper, $(\pi_2)_*S_{\gamma}$ makes sense and one has $\supp(\pi_2)_* S_{\gamma}\subset s(M)$ for all $\gamma\in \bC$. The same stays true about the support of 
\[ \lim_{\gamma\ra \infty}(\pi_2)_*S_{\gamma}=(\pi_2)_*(\lim_{\gamma\ra\infty}S_{\gamma})=(\pi_2)_*(\pi_{2,3})_*\tilde{S}_{\infty}\]
Moreover, $\supp(\pi_2)_*(T(\gamma,s))\subset s(M)$ and therefore all these currents in $E$ can be pushed forward via $\pi$ to $M$.

Following  (\ref{fdteq1}), we wedge (\ref{dtker}) with $\pi_3^*\omega$ and push it forward via $\pi_2$ and then via $\pi$. We obtain  the left hand side of (\ref{genPL}) with
\begin{itemize}
\item $s_{\infty}^*\omega=\pi_*(\pi_2)_*\left(\pi_3^*\omega\wedge \tilde{s}_{\infty}(M)\right)$, via a change of variables as $\pi\circ \pi_3\bigr|_{ \tilde{s}_{\infty}(M)}$ is a diffeomorphism onto $M\setminus X$; let $\varphi:=\pi\circ \pi_2$ then for every bounded form (of complementary degree) with compact support $\eta$ in $M$, one has
\[ \infty>\int_{\Bl_{X}s(M)}\varphi^*\eta\wedge\omega=\int_{\tilde{s}(M)}\varphi^*\eta\wedge\omega=\int_{M\setminus X}\eta\wedge s_{\infty}^*\omega.
\]
where the second equality is the change of variables or more precisely the area formula on  preimages of compact sets of $M$. This argument would work fine if $Z=\Bl_X(s(M))$ were a smooth analytic space. In general, one needs to pass to an embedded resolution of singularities of $Z$ in $E\times_M\bP(E)$ and repeat the argument.
\item $ Z(s,\omega):= \pi_*(\pi_2)_*(\pi_3^*\omega\wedge C^f\Ex_X(s(M)))$.
\end{itemize}

The statement about the support of $Z(s,\omega)$ comes out of the information about the support of the fiberwise cone. To see that $Z(s,\omega)$ simplifies to have the form described in (\ref{eqZ}) notice that $\pi_3^*\omega$ restricted to $[0]\times_M\bP(\bC\oplus E)\simeq\bP(\bC\oplus E)$ is just $\omega$ and then use Proposition \ref{fcprop}.

The operations used on the r.h.s. of (\ref{fdteq1}) preserve locally flat currents. 

Finally, since the push-forward and pull-back of forms via holomorphic maps commute with $\partial$ and $\bar{\partial}$ and because $\omega$ is $\partial$ and $\bar{\partial}$ closed we get (\ref{genPL}) from (\ref{dtker}) with the r.h.s. equal to

\begin{equation}\label{spcur}\partial\bar{\partial}(\pi_*(\pi_2)_*(\pi_3^*\omega\wedge T(\gamma,s)))=(\pi_*(\pi_2)_*(\pi_3^*\omega\wedge \partial\bar{\partial}T(\gamma,s)))=:T(\gamma,s,\omega).
\end{equation}
\end{proof}

\vspace{0.2cm}
For the rest of the section we  derive a few more properties about the currents appearing in (ii) and (iii) of Corollary \ref{cormt1}. The next result concerns (ii).

One of the basic properties of the projectivized normal cone $\bP(C_{X}M)$ is that it has (total) dimension $n-1=\dim{M}-1$.  However, in what follows we need to look at the fiber of
\[\pi:\bP(\bC\oplus C_{X}M)\ra X.
\]

The dimension of the fiber of $\pi$ can be larger than the expected dimension which is $n-1-\dim{X}$ as one can easily see from the following example: $M=\bC^2$ , $E=\bC^2\times \bC^2$ and
 \[s(x,y)=(x^2,xy)\] which has codimension $1$ but at the origin has an affine normal cone of dimension $2$. It is easily seen that at a generic point along an irreducible component of $s^{-1}(0)$ of codimension $k$ in $s(M)$ (or in $M$) the dimension of the fiber of $\pi$ is $k$. This is because the total dimension has to be $n$.  Clearly the fiber of the normal cone is an analytic space with possibly different irreducible components and multiplicities.

Now the dimension of the fiber at non-generic points is larger than at generic points.

In the analytic current $A:=\bP(\bC\oplus C_{X}M)$,   the irreducible components of $A$ may come with multiplicities. 

The slices $\langle A,\pi,x\rangle=A\cap \pi^{-1}(x)$ exist for every $x\in X$, and they are  analytic currents as well. However, the correspondence
\[ y\ra \langle A,\pi,y\rangle
\]
is continuous in the flat topology only at those points where the dimension of the fiber $\pi^{-1}(x)$ is constant.  We can be even more precise. The next Proposition relies on King's Fibering Theorem (Theorem 3.3.2 in \cite{K})  and is fundamental for the generalization of Poincar\'e-Lelong.

\begin{prop}\label{cormt2} Let $Z$ be an irreducible component of $X:=s^{-1}(0)$ of complex codimension $k$ in $M$, let $Z^{\gen}$ be the open set of regular points in $Z$ where the dimension of the fiber of $\pi$  is $k$ and let $Z^c$ be the union of the other irreducible components of $s^{-1}(0)$, a closed analytic set.
\begin{itemize}
\item[(i)] If $\deg \omega<2k$ then the support of the current $Z(s,\omega)$ is contained in $Z^c$;
\item[(ii)] If $\deg\omega=2k$ then  the function  $g:Z^{\gen}\ra \bC$:
\[ g(x):=\langle A,\pi,x\rangle\wedge \omega
\]
is continuous and
\[Z(s,\omega)\bigr|_{M\setminus Z^c}= g\cdot Z^{\gen} \]
where $Z^{\gen}$ is considered an analytic space with its reduced structure.
\end{itemize}
\end{prop}
\begin{proof} Part (i) follows from Fubini. The current $A$ is an analytic current, in particular it is rectifiable  and therefore has the structure $(h,\mathcal{H}^{2n}\llcorner{\bP(\bC\oplus C_XM)}, \overrightarrow{\xi})$ where $h$ is an integer valued function, the restriction of the Hausdorff measure $\mathcal{H}^{2n}$ (well-defined on $\bP(\bC\oplus E)$) is to the support of the analytic space $\bP(\bC\oplus C_XM)$, and $\overrightarrow{\xi}$ is an $\mathcal{H}^{2n}\llcorner{\bP(\bC\oplus C_XM)}$ -a.e. defined orientation multivector. To be precise, one can define $\overrightarrow{\xi}$ at the smooth points of the irreducible components of $\bP(\bC\oplus C_XM)$ (considered with the reduced structure). Applied to a test form $\upsilon$, the current $\bP(\bC\oplus C_XM)$ has the following expression:
\[ \int_{\bP(\bC\oplus C_XM)}h\upsilon(\overrightarrow{\xi})~d\mathcal{H}^{2n}
\]
If $\upsilon=\omega\wedge\pi^*\eta$ with $\deg{\omega}$ smaller than the dimension of the fiber then $\upsilon(\overrightarrow{\xi})=0$ a.e.

 Part (ii) follows from the Fibering Theorem of King. The formula without the continuity property of the function $g$ also follows directly from the co-area formula, together with the same argument that was used at (i), that at points on $Z$ where the fiber is "too big", the integral vanishes.
\end{proof}

\vspace{0.5cm}
 We will now work on giving an expression as simple as possible for the current $T(\gamma,s,\omega)$.

On one hand, item (iii) of Theorem \ref{t4r} together with the information that $T(\gamma,s)$ is flat allows us to say via the Flat Support Theorem (see Theorem 2.1.8 in \cite{K}) that $T(\gamma,s)$ comes from a current on the complex manifold $s(M)\times_M\bP(\bC\oplus E)$ . On the other hand, the restriction of $\pi_3$ to $s(M)\times_M\bP(\bC\oplus E)$ is a biholomorphism onto $\bP(\bC\oplus E)$. Let $\alpha$ be the inverse of this restriction and moreover define the following current on $\bP(\bC\oplus E)$:
\[ \widetilde{T(\gamma,s)}:=(\pi_3)_*T(\gamma,s)
\]
Then $\pi\circ \pi_2\circ \alpha:\bP(\bC\oplus E)\ra M$ coincides with $\pi$ and
\[\pi_*(\pi_2)_*(\pi_3^*\omega\wedge T(\gamma,s))=\pi_*(\omega\wedge \widetilde{T(\gamma,s)})
\]
Hence from (\ref{spcur}) we get 
\begin{equation}
\label{Tgso2} T(\gamma,s,\omega)=\pi_*(\omega\wedge \widetilde{T(\gamma,s)}).
\end{equation}

\vspace{0.2cm}
We introduce new notation. Let $U:=M\setminus X$.

 Let $L_s:=s_{\infty}^*\tau'$ where $s_{\infty}:U\ra \bP(E)$ is $s_{\infty}(m):=[s(m)]$ and $\tau'\ra \bP(E)$ is the tautological bundle. Notice that $L_s\hookrightarrow E\bigr|_U$,  comes with an obvious trivialization induced by $s\bigr|_{U}$. Let
 \[ \hat{s}:L_s\ra{\bC} \]
 be this trivialization. We will denote by $\hat{s}$ also the induced fiber bundle isomorphism
  \[\quad\hat{s}:\bP(\bC\oplus L_s)\ra \bP(\bC\oplus \bC)\times U= \bP^1\times U, \qquad \hat{s}^{-1}([z:w],m)=([z:ws(m)],m).\]

  On $\bP^1$ we have the standard meromorphic function:
 \[ \mer([z:w])=\frac{w}{z}
 \]
 which can be lifted to the total space of  the trivial bundle $\bP^1\times U\ra U$.
 
 Denote by $\iota:\bP(\bC\oplus L_s)\hookrightarrow \bP(\bC\oplus E)\bigr|_{U}$ the natural fiber bundle inclusion.

\begin{prop} \label{spark1} The following relation holds:
\begin{equation}\label{Tgso}T(\gamma,s,\omega)\bigr|_{U}=\frac{i}{\pi} \int_{\bP(\bC\oplus L_s)/U}\log{|\tilde{h}^s_{\gamma}|}\iota^*\omega=\frac{i}{\pi}\pi_*(\log{|\tilde{h}^s_{\gamma}|}\iota^*\omega)
\end{equation}
where $\pi:\bP(\bC\oplus L_s)\ra U$ is the projection and ${\tilde{h}_{\gamma}}^s:\bP(\bC\oplus L_s)\ra \bC$ is the  meromorphic function
\[ \tilde{h}^s_{\gamma}= \mer\circ \hat{s}-\gamma.
\]
In particular, $T(\gamma,s,\omega)\bigr|_{U}$ is a smooth form. If $\codim_{\bR}X\geq \deg{\omega}$, then the flat current $T(\gamma,s,\omega)$ is a  form with locally integrable coefficients, uniquely determined by $T(\gamma,s,\omega)\bigr|_{U}$. 

\end{prop}
\begin{proof} 
We take another look at the blow-up $\Bl_{\infty\times [0]}(\bP^1\times E)$ of $\infty\times [0]$ inside $\bP^1\times E$. From
\[ (\mu,\lambda v)\wedge (\theta,w)=0
\]
we get that   the projection $\pi_{2,3}$:
\[\pi_{2,3}:\Bl_{\infty\times [0]}(\bP^1\times E)\ra E\times_{M}\bP(\bC\oplus E)
\]
restricted to $v\neq 0$ is a biholomorphism onto its image.  In fact the image of $\pi_{2,3}$:
\[\Imag_{2,3}:=\{ (v,[\theta:w])\in (E\setminus\{0\})\times_{M}\bP(\bC\oplus E)~|~[\theta:w]=[\mu:\lambda v] \mbox{ for some }[\mu:\lambda]\in \bP^1\}\]  is the fiberwise cone (see Appendix \ref{Ap1}) of the graph of the projection $\pi^E:E\setminus\{0\}\ra \bP(E)$. It follows that the restriction of the projection $\pi_2:E\times_M\bP(\bC\oplus E)\ra E$
to $\Imag_{2,3}$ is a projective line bundle (i.e. the fiber is $\bP^1$) over $E\setminus\{0\}$  naturally isomorphic with
\[ \bP(\bC\oplus (\pi^E)^*\tau')
\]
where $\tau'\ra \bP(E)$ is the tautological bundle via the projection $\pi_3:\Imag_{2,3}\ra \bP(\bC\oplus E)$.

Similarly, when taking the strict transform $\tilde{S}$ of $\bP^1\times s(M)$ in $\Bl_{\infty\times [0]}(\bP^1\times E)$, the projection $\pi_{2,3}$ restricted to  $\tilde{S}\cap\{v\neq 0\}$ is a biholomorphism onto 
\[ \tilde{S}^{\neq 0}:=\{ (s(m),[\mu:\lambda s(m)])~|~m\in M\setminus s^{-1}(0), [\mu:\lambda]\in \bP^1\}\subset (E\setminus\{0\})\times_{M}\bP(\bC\oplus E)
\]
and therefore the projection $\pi_2$ restricted to $\tilde{S}^{\neq 0}$ is in fact a projective line bundle over $s(U)$. The projection $\pi:E\ra M$ identifies $s(U)$ with $U$ and under this identification the projective line bundle can in its turn be identified with $\bP(\bC\oplus L_s)$.  Hence the restriction of the projection $\pi_3:E\times_M\bP(\bC\oplus E)\ra \bP(\bC\oplus E)$
\[\pi_3:\tilde{S}^{\neq 0}\ra \bP(\bC\oplus L_s)
\]
induces an isomorphism of bundles over $U$. Moreover via this identification, the map 
\begin{equation} \label{tildeS}\tilde{S}^{\neq 0}\ra \bC,\qquad H(s(m),[\mu:\lambda s(m)])=\frac{i}{\pi} \log{\left|\frac{\lambda}{\mu}-\gamma\right|}
\end{equation}
becomes $\frac{i}{\pi}\log{|{\tilde{h}}^s_{\gamma}|}$. On the other hand, $H\cdot \tilde{S}^{\neq 0}=\frac{i}{\pi}(\pi_{2,3})_*(\log{|h^{\pi_1}_0(\gamma)|}\cdot \tilde{S}\bigr|_{v\neq 0})$ and therefore
 \begin{equation}\label{Tgso3}\widetilde{T(\gamma,s)}\bigr|_{\pi^{-1}(U)}=\frac{i}{\pi}(\pi_3)_*(\pi_{2,3})_*(\log|h^{\pi_1}_0(\gamma)|\cdot{\tilde{S}\bigr|_{v\neq 0}})=\frac{i}{\pi}\log{|{\tilde{h}}^s_{\gamma}|}\cdot [\bP(\bC\oplus L_s)]\end{equation}
Hence (\ref{Tgso}) is a consequence of (\ref{Tgso2}) and (\ref{Tgso3}).

The smoothness of $T(\gamma,s,\omega)\bigr|_{U}$ can be seen directly from (the obvious change of coordinates induced by $\hat{s}$):
\[ T(\gamma,s,\omega)\bigr|_{U}=\frac{i}{\pi}\int_{\bP^1\times U/U}\log{\left|\frac{\lambda}{\mu}-\gamma\right|}\hat{\omega},\]
where $\hat{\omega}:= (\hat{s}^{-1})^*\iota^*{\omega}, \; [\mu:\lambda]\in \bP^1$.

The uniqueness comes out from the fact that two flat currents of dimension $l$ on $M$ which coincide when restricted to the complement of a closed set of Hausdorff dimension smaller than $l$ are the same (see \cite{Fe}). The hypothesis ensures that the zero locus $X$ has Hausdorff dimension smaller than the dimension of the current $T(\gamma,s,\omega)$.
\end{proof}
\begin{rem} \label{Rema} Notice that if $\gamma=0$ then for $m\in U$:
\[\log|\tilde{h}_{0}^s([\theta:w],m)|=\log{\frac{|w|}{|\theta|}}-\log{|s(m)|},\qquad [\theta:w]\in \bP(\bC\oplus L_{s,m})\]
\end{rem}

\section{The generalized Poincar\'e-Lelong}

We would like to apply the general formula of Corollary \ref{cormt1} to particular forms $\omega$.

In what follows the rank of $E\ra M$ is fixed to be $k$. We assume now that $E$ has a Hermitian metric on it. It follows that there exists a unique connection $\nabla^E$ called the Chern connection, compatible with the metric such that $\nabla^{0,1}=\bar{\partial}$ where $\bar{\partial}$ is the canonical operator induced by the holomorphic structure on $E$. Then $\pi^*E\ra \bP(\bC\oplus E)$ has a metric and this will induce metrics on the tautological  line bundle 
\[ \xymatrix{      \tau \; \ar@{^{(}->}[rr]   \ar[dr] && \bC\oplus \pi^*E  \ar[dl] \\ 
& \bP(\bC\oplus E)&}
\]  and on the quotient bundle $Q:=\bC\oplus\pi^*E/\tau$. Therefore $Q$ has a natural connection $\nabla^Q$ and it is well known that all Chern forms $c_l(\nabla^Q)$ are $\partial$ and $\bar{\partial}$ closed.

 Another natural line bundle is the universal  bundle $\mathcal{O}(1):=\tau^*$ which comes also with a Chern connection.

\begin{rem}\label{Imprem} Along the "zero section", $[1:0]$ of $\bP(\bC\oplus E)$ we have that $Q$ coincides with $\pi^* E$. In fact, $(Q\bigr|_{[1:0]},\nabla^Q\bigr|_{[1:0]})$ equals $(\pi^*E,\pi^* \nabla^E)$ as bundles with connections.

At the other extreme, along the "hyperplane" at $\infty$,  $\bP(0\oplus E)\subset \bP(\bC\oplus E)$, we have that $\tau\subset \pi^*E$ coincides with $\tau'\ra \bP(E)$, the other tautological line bundle  and consequently a natural decomposition $Q=\bC\oplus Q'$ holds where $Q'$ is the quotient bundle $\pi^*E/\tau$. This is true also as bundles with connections, i.e. $(Q\bigr|_{\bP(E)},\nabla^Q\bigr|_{\bP(E)})=(\bC\oplus Q',d\oplus \nabla^{Q'})$, where $\nabla^{Q'}$ is the Chern connection on $Q'$.

 Notice that $Q'\ra \bP(E)$ has rank $k-1$. Clearly, 
 \[ c_k\left(\nabla^Q\bigr|_{\bP(E)}\right)=0 \qquad \mbox{and} \qquad c_j\left(\nabla^Q\bigr|_{\bP(E)}\right)=c_j(\nabla^{Q'}),\;\forall j<k.
 \]
\end{rem}

\vspace{0.2cm}

The next Lemma about characteristic classes is standard. As usual $\pi$ will denote the relevant projection between $\bP(\bC\oplus E), E, \bP(E)\ra M$.

\begin{lem} \label{charcl}  The following holds   at a cohomological level on $\bP(\bC\oplus E)$
\[ c_j(Q)=\sum_{l=0}^j \pi^*c_{j-l}(E) z^l
\]
where $z=c_1(\mathcal{O}(1))=-c_1(\tau)$. In particular, when restricted to a fiber $P_m:=P(\bC\oplus E_m)$ one has:
\begin{equation}\label{eqcoh1} c_j(Q)\bigr|_{P_m}=\left(z\bigr|_{P_m}\right)^j.
\end{equation}
\end{lem}
\begin{proof} Relation \eqref{eqcoh1} is an immediate consequence of $\tau=\mathcal{O}(1)^*$ and $c_*(\tau\oplus Q)=\pi^*c_*(E)$. Notice then that $\pi^*E$ is trivial along $P_m$.
\end{proof}
One can refine this result as follows. 
\begin{lem}\label{Lem2} There exists a smooth form  $\omega$ such that for every $j$ one has:
\begin{equation} \label{Lem2eq1}c_j(\nabla^Q)-\sum_{l=0}^j\pi^*c_{j-l}(\nabla^E)c_1(\nabla^{\tau^*})^l=\partial\bar{\partial} \omega
\end{equation}
In fact, along the fiber $P_m$ one has an equality of forms:
\[c_j(\nabla^Q)\bigr|_{P_m}=c_1(\nabla^{\tau^*})^j\bigr|_{P_m}.
\]
\end{lem}
\begin{proof} The first statement follows from Bott-Chern Proposition 1.5 in \cite{BC}. This particular case of a line bundle  is also contained in  Theorem \ref{mTt} below.

The  equality in the fiber can be proved directly as follows.   The unitary group $U(\bC\oplus E_m)$ acts transitively on the fiber $\bP(\bC\oplus E_m)$. Moreover, both sides of the equality are $U(\bC\oplus E_m)$ invariant forms. Hence one needs to prove the equality at a point only. 

 The curvature of the Chern connection $\nabla^Q$ on $Q\bigr|_{P_m}$ at the point $[1:0]\in\bP(\bC\oplus E_m)$ has a standard description. It can be identified (see \cite{Ci2}) with  the morphism of vector spaces:
 \[ \Lambda^2(E_m)\ra \mathfrak{u}(E_m),\qquad u\wedge w\ra \{v\ra \langle v,w\rangle u-\langle v,u\rangle w\}
 \]
 where $\mathfrak{u}(E_m)$ are skew-adjoint morphisms. In an orthonormal basis $\partial_{z_1},\ldots,\partial_{z_k}$ of $E_m$ this is the matrix  of two forms $A=(dz_i\wedge d\overline{z}_j)_{1\leq i,j\leq k}$. Take $\tilde{A}:=\frac{i}{2\pi}A$.
 
  Then $c_j(\nabla^Q)\bigr|_{P_m}$, the $j$-th $\mathfrak{gl}_k$-invariant polynomial in the entries of $\tilde{A}$, at $[1:0]$ is:
  \begin{equation}\label{calceq1}j!\left(\frac{i}{2\pi}\right)^{j}\sum_{\substack {|J|=j\\ J\subset \{1\ldots k\}}} \prod_{l\in J}dz_l\wedge d\overline{z}_l  \end{equation}

 The canonical symplectic form on $\bP(\bC\oplus E_m)$ equals $c_1({\nabla^{\tau^*}})$ and has at the point $[1:0]$ the expression:
 \[ \eta=\frac{i}{2\pi}\sum_{j=1}^kdz_j\wedge d\overline{z}_j
 \]
 One checks rather easily that $\eta^j$ equals the quantity in (\ref{calceq1}).
\end{proof}
\begin{rem} The form $\omega$ from (\ref{Lem2eq1}) has a rather explicit expression in \cite{BC}, Proposition 4.20.
\end{rem}

\begin{lem}\label{L4} Let $L\ra M$ be a holomorphic and Hermitian line bundle. Then for every $j\geq 1$ there exists a smooth form  $\eta$  on $M$ such that 
 \[ \pi_* \left(c_1(\nabla^{\tau^*})^j \right)- c_1\left(\nabla^{L^*}\right)^{j-1}=\partial\bar{\partial}\eta\]
 where $\tau^*\ra \bP(\bC\oplus L)$ is the hyperplane line bundle and $\pi:\bP(\bC\oplus L)\ra M$ is the projection. If $j=1$ then $\eta$ can be taken to be $0$.
\end{lem}
\begin{proof} The dual tautological bundle $\tau^*\ra\bP(\bC\oplus L)$ comes with a canonical holomorphic section $s^{\tau}$ which is the restriction to $\tau$ of the projection of $\bC\oplus \pi^*L$ to $\bC$. The zero locus of $s^{\tau}$ is $\bP(L)$ and it is not hard to see that $s^{\tau}$ is transverse to the zero section of $\tau^*$ hence the multiplicity of $\bP(L)$ is $1$. On the other hand, we clearly have that $\pi\bigr|_{\bP(L)}$ gives a biholomorphism onto  $M$. We have on $\bP(\bC\oplus L)$ by Griffiths-King that 
\[ c_1(\nabla^{\tau^*})-[\bP(L)]=\frac{i}{\pi}\partial\bar{\partial}\log|s^{\tau}|
\]
We wedge this with $c_1(\nabla^{\tau^*})^{j-1}$. Notice that the restriction of $c_1(\nabla^{\tau^*})$ to $\bP(L)$ coincides with $c_1(\nabla^{(\tau')^*})$ where $\tau'\ra \bP(L)$ is the tautological bundle. Under the identification $\bP(L)\simeq M$ we have that $\tau'=L$.  Hence we can write 
\[c_1(\nabla^{\tau^*})^{j}-c_1(\nabla^{L^*})^{j-1}\wedge M=\frac{i}{\pi}\partial\bar{\partial}\left(\log|s^{\tau}|c_1(\nabla^{\tau^*})^{j-1}\right)
\]

We can push-forward this formula now to $M$. The fiber integral 
\[\int_{\bP(\bC\oplus L)/M}\log|s^{\tau}|c_1(\nabla^{\tau^*})^{j-1}
\]
is a smooth form on $M$ by Lemma \ref{L5} below.

 If $j=1$ both terms on the left equal the constant function $1$.
\end{proof}
\begin{lem}\label{L5} Let $L\ra M$ be a complex  line bundle endowed with a Hermitian metric over a smooth manifold $M$. Let $\nabla^{\tau}$ be any metric compatible connection on the line bundle $\tau\ra \bP(\bC\oplus L)$. Let $s^{\tau}:\tau\ra \bC$ be the map induced by the restriction of the projection $\bC\oplus \pi^*L\ra \bC$. Then 
\[ \pi_*(\log|s^{\tau}|c_1(\nabla^{\tau^*})^j)
\]
is a smooth form on $M$ where $\pi:\bP(\bC\oplus L)\ra M$ is the projection.
\end{lem}
\begin{proof} The statement is local on $M$, hence one can assume $L$ is trivializable. In fact using parallel transport along radial directions from a fixed point on $M$ for a certain metric connection of $L$ one can find a local bundle isometry $\bC\ra L$. Hence one can work on the  bundle with fiber $\bP^1$. For integration over the fiber purposes it is enough to use one chart $\bC\ra \bP^1$:
\[ z\ra [z:1]
\]
covered by a bundle isometry $\bC\ra \tau$ which over a point $z\in \bC$ reads (since $\tau\subset \underline{\bC^2}$):
\[ 1\ra \left(\frac{z}{\sqrt{1+|z|^2}},\frac{1}{\sqrt{1+|z|^2}}\right)
\]
Since $s^{\tau}$ is the projection onto the first coordinate it becomes in this trivialization of $\tau$ :
\[ \underline{\bC}\ra \underline{\bC},\qquad z\ra\frac{z}{\sqrt{1+|z|^2}}
\]
Hence $\log|s^{\tau}|=\log\frac{|z|}{\sqrt{1+|z|^2}}$ which is integrable with respect to $dzd\bar{z}$ at $z=0$.  This function gets multiplied by a smooth and bounded form  which is the pull-back of $c_1(\nabla^{\tau^*})^j$ to $\bC\times M$. It should be clear now that one can differentiate under the fiber integral sign in the horizontal (i.e. the $M$) directions  and since $\log|s^{\tau}|$ does not depend on $m\in M$ one can do this infinitely many times.
\end{proof}
The next more general version of Lemma \ref{L4} is needed.
\begin{lem} Let $E\ra M$ be a holomorphic and Hermitian bundle. There exist smooth forms $\eta_1$ and $\eta_2$ on $M$  such that
\begin{equation}\label{sf01} \pi_*(c_1(\nabla^{\tau^*})^j)-\left(\pi\bigr|_{\bP(E)}\right)_*(c_{1}(\nabla^{(\tau')^*})^{j-1})=\partial\bar{\partial}\eta_1
\end{equation}
\begin{equation}\label{sf2}\pi_*(c_j(\nabla^Q))-\left(\pi\bigr|_{\bP(E)}\right)_*(c_{j-1}(\nabla^{Q'}))=\partial\bar{\partial}\eta_2
\end{equation}
where both projections $\pi:\bP(\bC\oplus E)\ra M$ is the projection.

If $L\subset E$ is a line bundle then one has the following:
\begin{equation}\label{sf3}\left(\pi\bigr|_{\bP(\bC\oplus L)}\right)_*(\iota^*c_j(\nabla^Q))-c_{j-1}(\nabla^{Q'_L})=\partial\bar{\partial}\eta_3
\end{equation}
for a smooth form $\eta_3$ where $\iota:\bP(\bC\oplus L)\hookrightarrow \bP(\bC\oplus E)$ and $Q'_L:=E/L$.

 If $j=1$ then $\eta_1$ and $\eta_2$ can be taken to be $0$.
\end{lem}
\begin{proof} The equality (\ref{sf01}) is obtained  with the same idea as the proof of Lemma \ref{L4} except that now $(s^{\tau})^{-1}(0)=\bP(E).$

 Relation (\ref{sf2}) is a consequence of (\ref{sf01}) and of Lemma \ref{Lem2}. 
 
 The equality (\ref{sf3}) follows from Lemma \ref{Lem2},  Lemma \ref{L4}  and  the fact that there exists a smooth $\eta$ such that
 \[ c_{j-1}(\nabla^{Q'_L})-\sum_{l=1}^{j} c_{j-l}(\nabla^E)c_1(\nabla^L)^{l-1}=\partial{\bar{\partial}}\eta
 \]
  consequence of Bott-Chern Proposition 1.5 in \cite{BC}.
\end{proof}

\vspace{0.3cm}

Suppose now that $s:M\ra E$ is a holomorphic section such that $X:=s^{-1}(0)$ is pure-dimensional of complex codimension $l$ in $M$.  The "expected" codimension of $X$ is equal to $k=\rank{E}\geq l$ and this happens for example when the sheaf of ideals induced by $s$ is a complete intersection.  This  particular case will be separately emphasized in what follows.

Let $U:=M\setminus X$ and $L_s\ra U$ be the line subbundle of $E$ induced by $s$. Let $Q'_s:=E/L_s$ be the resulting quotient bundle on $U$. It is holomorphic and has an induced Hermitian metric.

We emphasize that the analytic current $[s^{-1}(0)]$ generated by the irreducible components of the analytic set $s^{-1}(0)$ contains  the Hilbert-Samuel multiplicities as defined in Fulton\cite{F} Example 4.3.4.

We will apply Corollary \ref{cormt2} to the following  forms that live on  $\bP(\bC\oplus E)$
\begin{itemize}
\item[(a)] $\omega_1:=c_j(\nabla^Q)$;
\item[(b)] $\omega_2:=c_1(\nabla^{\tau^*})^j$
\end{itemize}
for $j\leq \codim s^{-1}(0)$.

\begin{theorem}\label{MT} Let $s:M\ra E$ be a holomorphic section such that $s^{-1}(0)$ is of pure complex codimension $l>0$ and let $\delta_j^l$ be the Kronecker symbol. The following generalizations of the Poincar\'e-Lelong formulas hold for all $j\leq l$:

\begin{equation}\label{E1} c_j(\nabla^E) - c_{j}(\nabla^{Q'_s}) - \delta_{j}^l[s^{-1}(0)] = -\frac{i}{\pi} \partial\bar{\partial}\left( \log{|s|}\left(c_{j-1}{(\nabla^{Q'_s}})+\partial \bar{\partial}\eta_1\right)\right)\end{equation}
\begin{equation}\label{E2}-c_1(\nabla^{L_s^*})^j - \delta_{j}^l[s^{-1}(0)]= -\frac{i}{\pi}\partial\bar{\partial}\left( \log{|s|}\left(c_1(\nabla^{L_s^*})^{j-1}+\partial \bar{\partial}\eta_2\right)\right)\end{equation}
where $\eta_1$ and $\eta_2$ are smooth forms on $U$, all equal to $0$ for $j=1$.

If $s$ does not vanish, then the two equations are identities of smooth forms on $M$ and hold for all $j$. 
\end{theorem}
\begin{proof} We use  formula (\ref{genPL}) of Corollary \ref{cormt1} for $\gamma=0$.

The identification of $Z(s,\omega_i)$, $i=1,2$ for $j=l$ with $[s^{-1}(0)]$ occupies the next subsection. 

The equalities $s_0^*c_j(\nabla^Q)=c_j(\nabla^E)$ and $s_0^*(c_1(\nabla^{\tau^*})^j)=0$ are the object of Remark \ref{Imprem}.

The equalities $s_{\infty}^*(c_j(\nabla^Q))=c_j(\nabla^{Q_s'})$ and $s_{\infty}^*(c_1(\nabla^{\tau^*})^j)=c_1(\nabla^{L_s^*})^j$ are a simple matter of making explicit the pull-backs, using again Remark \ref{Imprem}.

We will therefore concentrate on the right hand side for which we have the preliminary formulas of Proposition \ref{spark1}. We will analyze first the r.h.s. of (\ref{E2}). We use Remark \ref{Rema} and the notation introduced there and split the fiber integral in two pieces:
\begin{equation}\label{inteq35}\int_{\bP(\bC\oplus L_s)/U}\log\left(\frac{|w|}{|\theta|}\right)\iota^*c_1(\nabla^{\tau^*})^j-\int_{\bP(\bC\oplus L_s)/U}\log|s|\iota^*c_1(\nabla^{\tau^*})^j
\end{equation}

Let $\alpha_2:=\int_{\bP(\bC\oplus L_s)/U}\log\left(\frac{|w|}{|\theta|}\right)\iota^*c_1(\nabla^{\tau^*})^j$ whose smoothness follows directly from the similar property of $T(\gamma,s,\omega)\bigr|_{U}$. We argue that $\alpha_2=0$. First, notice that $\alpha_2$ makes sense for any smooth Hermitian line bundle $L$.  Consider thus $L$ endowed with a metric connection $\nabla$ out of which one induces a connection on $\tau^{*}\ra \bP(\bC\oplus L)$. Clearly the question is local on $U$ and  we restrict to a patch where $L$ is trivializable. Due to Lemma \ref{ndy} used with $E_1=\bC$ and $E_2=L$ and $\varphi$ a trivializing unitary isomorphism we are reduced to $L=\bC$. The advantage is that we can use the (orientation preserving)  self-diffeomorphism $C\times \id_U$ of $\bP^1\times U$ where
\[ C([\theta:w])= [w:\theta].
\]
On one hand, the forms $c_1(\nabla^{\tau^*})$ enjoy the symmetry property $C^*c_1(\nabla^{\tau^*})=c_1(\nabla^{\tau^*})$. On the other hand,
 $C^*\log\left(\frac{|w|}{|\theta|}\right)=-\log\left(\frac{|w|}{|\theta|}\right)$ and therefore $\alpha_2=0$.

 Since $s$ depends only on $M$, the second integral in (\ref{inteq35}) equals, using Lemma \ref{L4}:
\[\log|s|\int_{\bP(\bC\oplus L_s)/U}c_1(\nabla^{\tau^*})^j=\log{|s|}(c_1(\nabla^{L^*})^{j-1}+\partial \bar{\partial} \eta_2)
\]
for some smooth form $\eta_2$ on $U$. This finishes the proof of the (\ref{E2}). 

The reasoning for (\ref{E1}) is similar but instead uses (\ref{sf3}).

For $j<l$, the vanishing of the term supported on $[s^{-1}(0)]$ is guaranteed by Proposition \ref{cormt2} part (i).

\end{proof}

The following  change of variables for fiber integration was used in the Theorem \ref{MT}.
\begin{lem}\label{ndy} Let $E_i\ra M$, $i=1,2$ be  smooth Hermitian  bundles of rank $k$ endowed with  metric compatible connections $\nabla^i$ and $\varphi:E_1\ra E_2$ be a  unitary isomorphism such that $\nabla^1=\varphi^{-1}\nabla^2\varphi$. Denote by $\hat{\varphi}:\bP(\bC\oplus E_1)\ra \bP(\bC\oplus E_2)$ the induced bundle diffeomorphism.  Let $\nabla^{\tau_i}$ be the induced connections  on the corresponding tautological line bundles $\tau_i\ra \bP(\bC\oplus E_i)$ and $\nabla^{Q_i}$ the connections on the (tautological) quotient bundles $Q_i\ra \bP(\bC\oplus E_i)$. Let $P(x,y_1,\ldots,y_k)$ be a polynomial in $(k+1)$-variables. Then
\begin{equation}\label{c1nab}\hat{\varphi}^*[P(c_1(\nabla^{\tau_2}),c_1(\nabla^{Q_2}), \ldots,c_{k}(\nabla^{Q_2}))]=P(c_1(\nabla^{\tau_1}),c_1(\nabla^{Q_1}), \ldots,c_{k}(\nabla^{Q_1}))
\end{equation}
Consequently, if $L_i\subset E_i$ are line bundles such that $\varphi(L_1)=L_2$ then
\[\int_{\bP(\bC\oplus L_1)/M}\log\left(\frac{|w_1|}{|\theta|}\right)P(c_1(\nabla^{\tau_1}),c_*(\nabla^{Q_1}))=\int_{\bP(\bC\oplus L_2)/M}\log\left(\frac{|w_2|}{|\theta|}\right)P(c_1(\nabla^{\tau_2}),c_*(\nabla^{Q_2})),\] where $[\theta:w_i]$, $i=1,2$ are  projective coordinates of the fibers in $\bP(\bC\oplus L_i)$.
\end{lem}
\begin{proof} The consequence is a well-known fiber integral property once one quickly checks  that $\hat{\varphi}^*\log\left(\frac{|w_2|}{|\theta|}\right)=\log\left(\frac{|w_1|}{|\theta|}\right)$.  On the other hand, (\ref{c1nab}) will follow if we show it for the cases $P=x$ and $P=y_i$ , $i=1,\ldots,k$. In turn,  it will be enough to show that there exists  unitary isomorphisms $\psi_{\tau}:\tau_1\ra\tau_2$ and $\psi_{Q}:Q_1\ra Q_2$ covering $\hat{\varphi}$ such that
\begin{equation}\label{tauQ} \nabla^{\tau_1}=\psi^{-1}_{\tau}\left(\hat{\varphi}^*\nabla^{\tau_2}\right)\psi_{
\tau}\footnote{If $F_i\ra M_i$ are vector bundles and  $\beta:F_1\ra F_2$ is an isomorpphism covering $\alpha:M_1\ra M_2$ we denote also by $\beta$ the induced isomorphism of bundles $F_1\ra \alpha^*F_2$ over $M_1$.},\qquad\nabla^{Q_1}=\psi^{-1}_{Q}\left(\hat{\varphi}^*\nabla^{Q_2}\right)\psi_Q.
\end{equation}
Let us notice that  if $\pi_i:\bP(\bC\oplus E_i)\ra M$ are the projections  then there exists a natural unitary vector bundles isomorphism $\psi:\pi_1^*E_1\ra \pi_2^*E_2$ covering $\hat{\varphi}$ namely
\begin{equation}\label{psieq13}\psi(m,[\theta:w_1],v)=(m,[\theta:\varphi_m(w_1)], \varphi_m(v))\quad m\in M, \; [\theta:w_1]\in\bP(\bC\oplus E_{1,m}), v\in E_{1,m}
\end{equation}

Then (\ref{tauQ}) will follow from
\begin{equation}\label{psipi}\pi_1^*\nabla^1=\psi^{-1} \left(\hat{\varphi}^*(\pi_2^*\nabla^2)\right)\psi\qquad\Leftrightarrow\qquad \psi(\pi_1^*\nabla^1)=\hat{\varphi}^*(\pi_2^*\nabla^2)\psi
\end{equation}
since $\id_{\underline{\bC}}\oplus \psi:\underline{\bC}\oplus \pi_1^*E\ra \underline{\bC}\oplus\pi_2^*E$ takes $\tau_1$ to $\tau_2$ and $Q_1$ to $Q_2$.
\[\xymatrix{ \pi_1^*E_1 \ar[rr]^{\psi}\ar[d] & & \pi_2^* E_2\ar[d]\\
 \bP(\bC\oplus E_1) \ar[rr]^{\hat{\varphi}} \ar[dr]^{\pi_1} & &\bP(\bC\oplus E_2) \ar[dl]_{\pi_2}\\
   E_1&\ar[l]^{s} M   \ar [r]_{\varphi(s)} & E_2 }\]
   If $s\in\Gamma(E_1)$ then notice that (\ref{psieq13}) implies
   \[\psi(\pi_1^*s)=\hat{\varphi}^*\pi_2^*(\varphi (s))\in \Gamma(\pi_1^*E_2)=\Gamma(\hat{\varphi}^*\pi_2^*E_2)\quad \Leftrightarrow\quad \psi(s\circ \pi_1)=\varphi(s)\circ \pi_1.\]
  Therefore, the following equalities justify (\ref{psipi}) for a  section of $\pi_1^*E$ of type $\pi_1^*s$
  \[\hat{\varphi}^*(\pi_2^*\nabla^2)_X\psi(\pi_1^*s)=\hat{\varphi}^*(\pi_2^*\nabla^2)_X(\hat{\varphi}^*\pi_2^*(\varphi (s)))=(\pi_1^*\nabla^2)_X(\pi_1^*(\varphi(s))=\nabla^2_{d\pi_1(X)}\varphi(s)=\]\[=\varphi(\nabla^1_{d\pi_1(X)}s)=\psi[(\pi_1^*\nabla^1)_X(\pi_1^*s)]
  \]
  since $\psi=\pi_1^*\varphi$ when seen as a morphism $\pi_1^*E_1\ra\pi_1^*E_2$. 
  
  Use now that the map $\Gamma(M;E_1)\otimes_{\Gamma(M;\bC)}\Gamma(\bP(\bC\oplus E_1);\bC)\ra \Gamma(\bP(\bC\oplus E_1);\pi_1^*E)$, $s\otimes f \ra f\pi_1^* s$ is an isomorphism of vector spaces and the Leibniz property of connections in order to conclude that (\ref{psipi}) is valid.
  
  \end{proof}
\begin{rem}
For $j>l$ the formulas tend to be quite more complicated and the structure of the normal cone of $s^{-1}(0)$ becomes fundamental. See Theorem \ref{homzs} for an example.
\end{rem}

\begin{rem} Notice that when $E$ is a line bundle both expressions recover the Poincar\'e-Lelong formula of Griffiths and King \cite{GK}.

Note that when $s^{-1}(0)=\emptyset$ the last statement in Theorem \ref{MT} applied to (\ref{E1}) recovers  Theorem 1 of Bott-Chern in \cite{BC} which says that the Bott-Chern class of $\pi^*c_n(E)$ is zero on the punctured open unit disk bundle $B(E)^*:=\{e\in E~|~|e|<1\}$ where $\pi:E\ra M$ is a holomorphic vector bundle of rank $n$.  Theorem 1 from \cite{BC} follows from the more general Proposition 1.5, op. cit., which infers that for an exact sequence of holomorphic vector bundles
\[ 0\ra E_1\ra E_2\ra E_3\ra 0
\]
one has $c_*(E_2)=c_*(E_1)c_*(E_3)$ in $H^{*,*}_{BC}(M)$.  This is recovered in Section \ref{BCtheorem}.

Furthermore, it is interesting to compare (\ref{E1}) with Bott-Chern Proposition 4.20  equation (4.22) in \cite{BC} in which for $s$  non-vanishing they obtain that
\[ c_n(E)=-\frac{i}{\pi}\partial\bar{\partial}(\log{|s|}c_{n-1}(Q'_s)+\xi)
\]  
where  $\xi$ is an explicit smooth form, a polynomial expression in the curvatures of the Chern connections of $E$ and  of $Q'_s:=E/\langle s\rangle$. It would be interesting to show that the smooth form $\log{|s|}\partial\bar{\partial} \eta_1$ in (\ref{E1}) is equal, mutatis mutandis, with the Bott-Chern form $\xi$.  
\end{rem}

\begin{cor} \label{corm1} Let $s^{-1}(0)$ be a complete intersection, i.e. it has the expected codimension $k$ which coincides with the rank of $E$. Then the first Poincar\'e-Lelong formula simplifies to:
\[c_k(\nabla^E)-[s^{-1}(0)] = -\frac{i}{\pi} \partial\bar{\partial}( \log{|s|}(c_{k-1}{(\nabla^{Q'_s}})+\partial\bar{\partial}{\eta_1})
\]
Moreover if $j<k$ then there exists a flat current $T$ such that:
\[ c_j(\nabla^E)-c_j(\nabla^{Q'_s}) =\partial\bar{\partial}T.
\]
\end{cor}
\begin{proof} The bundle $Q'_s$ has rank one less than $E$ hence $c_{k}(\nabla^{Q'_s})=0$ when $k$ is the rank of $E$.

The second statement follows Proposition \ref{cormt2} part (i).
\end{proof}
As a corollary we get the next well-known result \cite{BaBo}.
\begin{cor}[Baum-Bott] Let $\eta:L\ra TM$ be a  holomorphic map between the line bundle $L$ and the tangent bundle of $TM$ with only isolated zeros. Then
\[ \int_{M} c_n(TM- L)=\sum_{p\in \eta^{-1}(0)}\mu_p
\]
where $\mu_p$ is the Milnor number of $\eta$ at $p$. 
\end{cor}
\begin{proof} One considers $\eta$ as a section of $TM\otimes L^*$. Notice that
\[c_n(TM- L)=c_n(TM\otimes L^*)
\]
Then we are in the situation of Corollary \ref{corm1}. The fact that the multiplicities at the zero points are given by the Milnor numbers is well-known. 
\end{proof}

\subsection{Localization of the multiplicity}
The multiplicity $(e_{X}Y)_Z$ of a complex subspace $X$ of a  space $Y$ along an irreducible component $Z$ of $X$ can be defined algebraically as the coefficient  of the top power of the Hilbert-Samuel polynomial induced by the primary ideal $q$ of $X$ in the localization of $Y$ along $Z$ multiplied with $d!$ where $d$ is the codimension of $Z$ in $Y$. 

The "topological" approach of Fulton and Mac-Pherson looks at these multiplicities via the Segre classes of the normal cone $C_XY$ along $Z$ in the context of Chow groups on which the Segre classes act.

In our case $Y=M$ is regular.  Moreover we take $X\subset M$ to be the zero locus of a section $s:M\ra E$ of a holomorphic vector bundle. The analytic space structure of $X$ is induced by the ideal sheaf $\mathcal{I}$ generated by the components of $s$ in local coordinates. Assume for simplicity that $X$ is purely $k$-codimensional, but is not necessarily reduced, nor irreducible.   In this situation the normal cone $\bP(C_XM)$ of $X$ in $M$ is isomorphic with the exceptional divisor of the strict transform of $\mathcal{I}$ with respect to the blow-up of the zero section of $E$. Since the exceptional divisor of the blow-up of the zero section $\Bl_{[0]}(E)$ is the smooth manifold $\bP(E)$ with the restriction of the blow-down map corresponding to the natural projection of $\bP(E)\ra M,$ the canonical line bundle of $\bP(C_XM)$ as defined by Grothendieck coincides in fact with the pull-back  to $\bP(C_XM)$ via the canonical embedding of the canonical line bundle 
\[\mathcal{O}(1)\ra \bP(E).\]
In what follows we therefore consider $\bP(\bC\oplus C_XM)$ to be an analytic subspace of $\bP(\bC\oplus E)$.

Then the multiplicity $(e_XM)_Z$ of $M$ along every irreducible component $Z$ of codimension $k$ is defined as follows:
\begin{equation}\label{algeq4} \sum_{Z\in \Irr(X)}(e_XM)_Z\cdot [Z]=q_*(c_1(\mathcal{O}(1))^k\cap \bP(\bC\oplus C_XM))
\end{equation}
where $q_*:\bA_*^{\an}(\bP(\bC\oplus E))\ra \bA_*^{\an}(M)$ is the natural push-forward and $\Irr(X)$ are the irreducible components of $X$. To be precise, $q_*$ acts in degree $n-k$ in (\ref{algeq4}) and $q_*(c_1(\mathcal{O}(1))^k\cap \bP(\bC\oplus C_XM))\in \bA_{n-k}^{\an}(X)$. The latter group is generated over $\bZ$ by $\Irr(X)$. 

\begin{rem}\label{Fudefb} There is a drawback in this definition analogous to the one from \cite{F}, namely that the Chow class of $[Z]$ might be zero, which happens for example when $M=\bC^n$. 
\end{rem}

This has a corresponding expression in $H_{n-k,n-k}^{BC}(M)$:
\[\sum_{Z\in \Irr(X)}(e_XM)_Z\cdot [Z]=q_*(c_1(\nabla^{\tau^*})^k\wedge \bP(\bC\oplus C_XM))\]
Due to Lemma \ref{Lem2} we have an alternative way of computing this.

\begin{lem} The following equality holds as currents on $M$:
\[q_*(c_k(\nabla^{Q})\wedge \bP(\bC\oplus C_XM))=q_*(c_1(\nabla^{\tau^*})^k\wedge \bP(\bC\oplus C_XM)).\]
\end{lem}
\begin{proof} Clearly true by the Fibering Theorem of King \cite{K} (see also Proposition \ref{cormt2}) .
\end{proof}

Let $Z\in\Irr(X)$ and suppose $\codim_{M}Z=k$. The generic fiber $C_m$ of  $\bP(\bC\oplus C_XM)\bigr|_{Z}$ has dimension $k$ and is an analytic space. When $X:=s^{-1}(0)$ the fiber is a projective analytic subspace  of $\bP(\bC\oplus E_m)$. As such it has a well-defined degree which is the intersection number of this fiber with a generic $k$-codimensional (projectively) linear subspace of $\bP(\bC\oplus E_m)$. On the other hand, this degree can also be computed by the integral:
\[ \int_{C_m}c_1(\nabla^{\tau^*})^k
\] 
since $c_1(\nabla^{\tau^*})$ is the Poincar\'e dual to a hyperplane section and $\tau$ here is the tautological bundle on $\bP(\bC\oplus E_m)$. Hence, for this particular form $\omega$, the function $g$ of Proposition \ref{cormt2} is continuous and has integer values. Therefore it has to be constant at the generic points of $Z$. This constant is by (\ref{algeq4}) the multiplicity of $M$ along $X$ at $Z$.

We have thus obtained  the following consequence of the Fibering Theorem of King:
\begin{theorem}\label{mulT} Let $E\ra M$ be a holomorphic and Hermitian vector bundle and $s:M\ra E$ be a holomorphic section. Let $Z$ be an irreducible component of $s^{-1}(0)$ of codimension $k$ and $Z^c$ the union of all the other irreducible components. Then 
\[ \pi_*(c_k(\nabla^Q)\wedge \bP(\bC\oplus C_{s^{-1}(0)}s(M)))\bigr|_{M\setminus X^c}=\pi_*((c_1(\nabla^{\tau^*})^k\wedge \bP(\bC\oplus C_{s^{-1}(0)}s(M)))\bigr|_{M\setminus Z^c}=\]\[=(e_{s^{-1}(0)}M)_Z\cdot [Z\setminus Z^c].
\]
Moreover $(e_{s^{-1}(0)}M)_Z$ is the degree of the generic fiber inside $\bP(\bC\oplus E)$ of $\bP(\bC\oplus C_XM)$ over $X$ where generic  means the smooth points of $Z$ for which the dimension of the fiber of the affine normal cone is $n-k$.
\end{theorem}

\begin{example} When $s^{-1}(0)$ is a complete intersection there is no jump in the dimension of the fiber of the normal cone and hence the multiplicity can be computed at every point of the irreducible component (see Griffiths and Harris \cite{GH}, page 130 for a justification in the hypersurface case).
\end{example}

\begin{rem} Due to the drawback expressed in Remark \ref{Fudefb}, it  would  be  desirable to obtain an algebraic proof of the last statement of Theorem \ref{mulT}. 
\end{rem}

\section{Twisting by a line bundle}
In the later sections we construct a theory for holomorphic morphisms of vector bundles in general.
But when $s:L\ra E$ is a holomorphic morphism and $L$ is a line bundle the results of the previous section extend without difficulty  and we collect them in this short section.

There are several ways to go about proving this.  One way is to see that in order to prove the double transgression formulas one needs only work locally on $M$. Due to the sheaf property of currents this is true for the double transgression formula of kernels and also for the general Poincar\'e-Lelong. Another way is to consider the associated section $\tilde{s}:M\ra L^*\otimes E$. The next short presentation  follows in the footsteps of the previous sections. We believe it keeps better track of the objects involved and of the original section $s$.

Let $\bP^{\circ}(L\oplus E):=\{([\beta:v],m)\in \bP(L_m\oplus E_m),\; m\in M~|~\beta\neq 0\}$. Take the closure inside $\bP^1\times \bP^{\circ}(L\oplus E)\times_M\bP(L\oplus E)$ of
\begin{equation}\label{eq61} \{([1:\lambda],[\beta: s(\beta)], [\beta:\lambda s(\beta)],m)~|~m \in M, \beta\in L_m\}\end{equation} inside $\bP^1\times \bP^{\circ}(L\oplus E)\times_M \bP(L\oplus E)$ that can be described as follows. We have a canonical section $[1:0]$ of $\bP(L\oplus E)$ and we can blow-up  $\infty\times [1:0]$ inside $\bP^1\times \bP^{\circ}(\bC\oplus E)$. This can be realized as a subspace of $\bP^1\times \bP^{\circ}(L\oplus E)\times_M \bP(L\oplus E)$. In fact,
\[ \Bl_{\infty\times [1:0]}(\bP^1\times \bP^{\circ}(L\oplus E))=\{([\mu:\lambda],[\beta_1:v],[\beta_2:w])~|~(\mu\beta_1,\lambda v)\wedge (\beta_2,w)=0\}
\]

  The strict transform of $\bP^1\times \{[\beta:s(\beta)]~|~\beta\in L_m \}$ will give the closure of (\ref{eq61}).

 The projective normal cone $\bP(C_{s^{-1}(0)}M)$ of $s^{-1}(0)$ in $M$ is defined locally via the  ideal sheaf induced by $s$ and whose support is $s^{-1}(0)$. The same ideal sheaf is induced by $\tilde{s}$.  The projective normal cone is naturally an analytic subspace of  $\bP(E)\simeq \bP(L^*\otimes E)$, the latter isomorphism being canonical. The fiberwise cone $\bP(\bC\oplus C_{s^{-1}(0)}M)$ of $\bP(C_{s^{-1}(0)}M)$ lives in $\bP(\bC\oplus L^*\otimes E)\simeq \bP(L\oplus E)$. The multiplicity of $M$ along $s^{-1}(0)$ at an irreducible component $X$ is the same as the multiplicity of $M$ along $\tilde{s}^{-1}(0)$ at $X$ and therefore can be computed at a generic point in $X$ as the degree of the fiber of the normal cone inside $\bP(\bC\oplus L^*\otimes E)$ similarly to what was done in Theorem \ref{mulT}.
 
 Clearly $\bP(L\oplus E)\simeq \bP(\bC\oplus L^*\otimes E)$ comes with a tautological bundle $\tau$ and a quotient bundle $Q$. Hermitian metrics on $L$ and $E$ induce Hermitian metrics on $\tau$ and $Q$. The zero section in this context is the inclusion $M\simeq \bP(L)\hookrightarrow \bP(E)$ while the $\infty$ section is $\bP(E)\hookrightarrow \bP(L\oplus E)$. This time, the restriction of $\tau$ to the zero section is isomorphic with $L\ra M$ while the restriction of $Q$ to $\bP(E)$ is isomorphic with the direct sum $L\oplus Q'$ where $Q'$ is as before the universal quotient  bundle on $\bP(E)$. These isomorphisms hold as pairs  (bundle, Hermitian metric) and hence as bundles with Chern connections.
 
 Applying the resulting abstract Poincar\'e-Lelong as in Theorem \ref{MT} to  $\omega_1:=c_k(\nabla^Q)$ and $\omega_2:=c_1(\tau^*)^{k}$  we get the following:
 
\begin{theorem}\label{mTt} Let $s: L\ra E$ be a holomorphic morphism and suppose $L$ and $E$ are endowed with Hermitian metrics. Let  $L_s:=\Imag s$, $Q_s':=E/L_s$ be the line and quotient bundles defined on $M\setminus s^{-1}(0)$ which have metrics induced from $E$.  Suppose $s^{-1}(0)$ is of pure complex codimension $k$. Then the following Poincar\'e-Lelong formulas hold for some locally flat currents $T_1$ and $T_2$.
 \begin{equation} 
 c_k(\nabla^E)-c_{k}(\nabla^L\oplus \nabla^{Q_s'})-[s^{-1}(0)]=\frac{i}{\pi}\partial\bar{\partial}T_1
\end{equation}
\begin{equation}\label{E22}
c_1(L^*)^k-c_1(L_s^*)^k-[s^{-1}(0)]=\frac{i}{\pi}\partial\bar{\partial} T_2
\end{equation}
In case $s^{-1}(0)=\emptyset$ then both formulas hold irrespective of $k$ on all of $M$.
\end{theorem}
Notice that when $s^{-1}(0)=\emptyset$ then $L$ and $L_s$ are isomorphic vector bundles, possibly non-isometric and therefore (\ref{E22}) gives the invariance of the Bott-Chern class under the change of Hermitian metric.

The last statement proves   Proposition 1.5 of \cite{BC} in a particular case:
\begin{cor} [Bott-Chern] If $0\ra L\ra E\ra Q\ra 0$ is an exact sequence of holomorphic vector bundles over $M$ with $L$ a line bundle then 
$c_*(E)=c_*(L\oplus Q)=c_*(L)c_*(Q)$
in $H^{*,*}_{BC}(M)$.
\end{cor}

\section{Bott-Chern duals of Chern-Fulton classes}

The proof of the Poincar\'e-Lelong formula (\ref{E2})  inspires the main result of this section that produces a Bott-Chern dual for any Chern-Fulton class of an analytic subspace defined as the zero locus of a holomorphic section $s:M\ra E$ of a holomorphic $E$. We will look at the non-homogeneous class
\[ \sum_{j\geq 0} c_1(\tau^*)^j=\frac{1}{1+c_1(\tau)}=:c_{ * }(-\tau),
\]
the total Chern polynomial of the \emph{virtual} $-\tau$ vector bundle. By \cite{F}, "integrating" $c_*(-\tau)$ over the normal cone $\bP(\bC\oplus C_{s^{-1}(0)}(M))$ gives the Segre classes of the normal cone. More precisely, Fulton defines
\begin{equation}\label{Segeq1} \Seg_{n- * }(s^{-1}(0),M):=q_*(c_{ * }(-\tau)\cap \bP(\bC\oplus C_{s^{-1}(0)}(M)))\in \bA_{n- * }(s^{-1}(0))
\end{equation}
where $q_*:\bA_{ * }(\bP(\bC\oplus C_{s^{-1}(0)}(M))))\ra \bA_{ * }(s^{-1}(0))$ is the natural push-forward. 
\begin{rem} The Segre classes are interesting only for $ * \geq k$, in other words just $\Seg_j$ with $0\leq j\leq d:=\dim {s^{-1}(0)}=n-k$ are non-trivial.
\end{rem}
When $M$ is projective we have a corresponding class in $H_{*,*}^{BC}(M)$ obtained from the composition of  the morphisms $\bA_{ * }^{\an}(M)\ra H_{ * , * }^{BC}(M)$  and $\bA_*(s^{-1}(0))\ra \bA_*(M)$. 

But even when $M$ is not projective the current $c_{ * }(-\tau)\cap \bP(\bC\oplus C_{s^{-1}(0)}(M))$ makes perfect sense on $\bP(\bC\oplus E)$, once we choose a Hermitian metric on $E$ which will determine a Hermitian metric on $\tau$. We also have a push-forward  induced by the canonical projection denoted here $q:\bP(\bC\oplus E)\ra M$.  We therefore take (\ref{Segeq1}) to be our definition of Segre classes but this time in $ H_{*,*}^{BC}(M)$, opting to ignore the Chern connection in notation.
  \[\Seg_{ * }:=\Seg_{ * }(s^{-1}(0),M)\]

 Define the Chern-Fulton class in the Bott-Chern homology group as
\[ c^F_i(s^{-1}(0)):=[c_*(TM)\wedge \Seg_*(s^{-1}(0),M)]_{d-i}\quad \mbox{i.e.}\]
\[c^F_i(s^{-1}(0))=\sum_{0\leq j\leq i}c_{i-j}(TM)\wedge \Seg_{d-j}\in H_{d-i,d-i}^{BC}(M),\quad \forall 0\leq i\leq d
\]

We will use Corollary \ref{cormt1} for forms 
\[\omega=c_{i-j}(\pi^*(TM))\wedge c_{k+j}(-\tau)=c_{i-j}(\pi^*(TM))c_1(\tau^*)^{k+j},\;\; 0\leq j\leq i\]  in order to obtain the following.

\begin{theorem} For a pure-dimensional set $s^{-1}(0)$ of codimension $k$, the  Chern-Fulton classes satisfy in $H_{n-k-*,n-k-*}^{BC}(M)$
\begin{equation}\label{CF}  c^F_{ * }(s^{-1}(0))=-c_1({L_s}^*)^k(c_{ * }( {TM\otimes L_s^*})),
\end{equation}
where $L_s:=\Imag s\bigr|_{M\setminus s^{-1}(0)}$ is the line subbundle determined by $s$.
\end{theorem}
\begin{proof} Applying the theory to the forms $\omega$ as announced we get relations of type:
\[ -c_{i-j}(TM)c_1(L_s^*)^{k+j}=c_{i-j}(TM)\Seg_{d-j},\qquad 0\leq j\leq i
\]
which we then sum for $0\leq j\leq i$.
\end{proof}
\begin{example}
It is interesting to look at the case  when  $L_s$ extends holomorphically to a line subbundle  $L\subset E$. For example, if $E$ is already a line bundle then $L_s=E$.  Then $s$ is a holomorphic section of $L$ and necessarily $k=1$. 

Suppose moreover that $s$ is actually transverse to the zero section and hence $[s^{-1}(0)]$ is non-singular. Then $L\bigr|_{s^{-1}(0)}$ plays the role of the normal bundle. In other words  one has an exact sequence of holomorphic bundles:
\[ 0\ra TX\ra TM\bigr|_{s^{-1}(0)}\ra L\bigr|_{s^{-1}(0)}\ra 0
\]
with the surjective map given by  $\nabla s$. This implies that $c_*(TX)=\frac{c_{*}(TM)}{c_*(L)}=c_*(TM\otimes L^*)$. 

Moreover $-c_1(L^*)=c_1(L)$ is the Bott-Chern dual of $s^{-1}(0)$ by Poincar\'e-Lelong and therefore the right hand side of (\ref{CF})  is just $c_*(TX)\cap [X]$ in  singular homology. 
\end{example}

\subsection{The homological zero locus} \label{Nrp} Fulton defines for a section $s:M\ra E$ of a vector bundle of rank $k$ a class $\bZ(s)$ in the Chow group $A_{n-k}(M)$ that satisfies 
\[ \bZ(s)=c_k(E)\cap M.
\]  in $\bA_{n-k}(M)$. Notice that $n-k$ is the expected dimension. We can of course take  $\bZ(s)$ and look at it as a class in the Bott-Chern group. But it is more interesting to define it directly. We start with the following  equality, obtained by restricting to  the open set $E\times_ME$ the identity (\ref{dtker}) for $\gamma=1$ (Notation: $\Delta_s=\{s(m),s(m)~|~m\in M\}$)
\[\Delta_s=C_{s^{-1}(0)}s(M).
\]
This of course holds in $H_{n,n}^{BC}(E\times_ME)$. We can push-forward this equality via $\pi_2:E\times_ME\ra E$ in order to get that in $H_{n,n}^{BC}(E)$ the following holds:
\begin{equation}\label{pretty} s(M)=C_{s^{-1}(0)}M
\end{equation}
Define
 \[Z(s)^{\tau}:=\pi_*(C_{s^{-1}(0)}M\wedge \tau)\in\mathcal{D}'_{n-k,n-k}(M)\] where $\tau$ is a bihomogeneous Thom form (see the proof of Proposition \ref{Gysin1}) of $E$ and $\pi:E\ra M$ is the projection. Since by (\ref{core1}) we have $\tau=[M]$ in $H_{n,n}^{BC}(E)$ for any choice of Thom form, it follows that the class of $Z(s)^{\tau}$ in $H_{n-k,n-k}^{BC}(M)$ does not depend on the choice of $\tau$.  Define $\bZ(s)\in H_{n-k,n-k}^{BC}(M)$ to be this class.  
 \begin{theorem}\label{homzs} The following holds in $H_{n-k,n-k}^{BC}(M)$:
 \[ \bZ(s)=c_k(E)
 \]
 \end{theorem}
 \begin{proof} By (\ref{pretty}) we need only check that $\pi_*(s(M)\wedge \tau)=c_k(E)$. On the other hand, 
 \begin{equation}\label{sm0} s(M)=[0]\;\;\; \mbox{in} \;\;H_{n,n}^{BC}(E).\end{equation} This comes out by taking the difference of (\ref{dtker}) for $\gamma=1$ and $\gamma=0$ restricted to $E\times_ME$ which says that $\Delta_s=s(M)\times_M[0]$ and the push-forward via $\pi_2$ gives (\ref{sm0}).
 
  One checks the next equality of currents:
 \[\pi_*([0]\wedge\tau)=\iota^*\tau
 \]
 where $\iota:M\hookrightarrow E$ is the inclusion of the zero section. Take now $\tau=(\rho\circ r) \pi^*c_k(E)+d(\rho\circ r)\wedge \Tc_k(E)$ as in \cite{CiAdvances} where $r$ is the radius function for some Hermitian metric on $E$ and $\rho:[0,\infty)\ra [0,\infty)$ is a compactly supported smooth function  equal to $1$ in a neighborhood of $0$ and $\Tc_k(E)$ is a transgressed form of $c_k(E)$. Clearly $\iota^*\tau=c_k(E) $ as the second term of $\tau$ is zero in a neighborhood of the zero section.
 \end{proof}

\section{Weighted projective actions}\label{Sec8}
Instead of the canonical action of $\bC^*$ on a vector space $V$, or more generally on a vector bundle $E\ra M$ we consider here a weighted homogeneous action. For that we will assume that a decomposition of $E$ into a direct sum of vector bundles
\[ E =E_0\oplus \ldots\oplus E_k.\]
is given. Then an (algebraic) action of $\bC^*$ with non-negative weights is given by
\[ \lambda * (v_0,\ldots,v_k)= (v_0, \lambda^{\beta_1}v_1,\ldots,\lambda^{\beta_k}v_k)
\]
where $0<\beta_1<\ldots<\beta_k$ are  integers. Let $\beta_0:=0$, the weight of the action on $E_0$.  Since  $*$ intertwines with multiplication by scalars we have an induced action  
\[ \bC^*\times \bP(E)\ra \bP(E)
\]
and the reason for taking $\beta_0=0$ is more apparent now. The particular case treated until now corresponds to $E_0=\bC$ and $k=1$.

We will consider a holomorphic section $s:M\ra \bP(E)$ and with the corresponding  ``action":
\[ \bC^*\times M\ra \bP(E),\qquad (\lambda,m)\ra \lambda*s(m).
\]
and we will look at the "graph" of this action $\{(s(m),\lambda s(m))~|~m\in M,\lambda \in \bC^*\}$ aiming to understand the limits $\lim_{\lambda\ra 0(\infty)}$ from a currential point of view. 

\vspace{0.2cm}

\begin{rem} Let $V$ and $W$ be vector spaces. Weighted actions arise naturally when one considers the canonical action of $\bC^*$ on the vector space $\Hom(V,W)$. This extends to an action on the Grassmannian $\Gr_{\dim V}(V\oplus W)$ and via the Plucker embedding to an weighted homogeneous action on a projective space $\bP(\Lambda^{\dim{V}}(V\oplus W))$.  
\end{rem}
Let $V=V_0\oplus \ldots\oplus V_k$ be any fiber (of $E\ra M$) and consider the closure of the graph of the action
\begin{equation}\label{actproj}\bC^*\times \bP(V)\ra \bP(V),\qquad (\lambda,v)\ra \lambda*v
\end{equation}
in $\bP^1\times \bP(V)\times \bP(V)$.  We notice first that this closure is the blow-up of the \emph{non-reduced} ideal
\begin{equation}       \label{wtf} 
\mathcal{I}:=\langle \mu^{\beta_k}v_0,\mu^{\beta_{k}-\beta_1}\lambda^{\beta_1}v_1,\ldots, \mu^{\beta_{k}-\beta_{k-1}}\lambda^{\beta_{k-1}}v_{k-1},\lambda^{\beta_k}v_k\rangle
\end{equation}
This ideal is generated by bihomogeneous polynomials in the variables $[\mu:\lambda]$ and $[v_0:\ldots:v_k]$ separately and as such it determines an algebraic subspace of $\bP^1\times \bP(V)$ and a corresponding coherent sheaf of ideals over the structure sheaf of $\bP^1\times \bP(V)$. The co-support (the zero locus)  of $\mathcal{I}$ is $\{\mu=0,\; v_k=0\}\cup \{\lambda=0,\;v_0=0\}$.

Recall that the blow-up of an ideal $I=\langle f_1,\ldots,f_k\rangle$ in an affine chart, here of type $\bC\times \bC^M$, is the closure of the graph of 
\[ \bC\times \bC^M\setminus Z(I)\ra \bP(\bC^k),\qquad x\ra [f_1(x):\ldots:f_k(x)].
\]
One checks easily that in the standard affine charts of $\bP^1\times \bP(V)$  this is indeed the case for the  sheaf of ideals $\mathcal{I}$, i.e. that  closure of the graph of the action (\ref{actproj}) restricted to these open sets is described as the closure of the graph of the map to $\bP(V)$ naturally determined by the generators of $\mathcal{I}$.

\vspace{0.2cm}
  
 In what follows $[w_0:\ldots:w_k]$ and $[v_0:\ldots:v_k]$  are the coordinates of a point in the first, respectively the second copy of $\bP(V)$ while $[\mu:\lambda]$ is a point in $\bP^1$.
 
  We give now equations for the closure of the graph of the action as follows.   
  
   Let  $\underline{k}:= \{0,\ldots k\}$ and $I\subset \underline{k}$ be a "connected" set, i.e. together with $I\ni a\leq b \in I$ it contains also every $c$ with $a\leq c\leq b$. Clearly I is completely determined by $m_I:=\min{I}$ and $M_I:= \max{I}$. Let $\beta(I)$ be the corresponding connected subset of $\{\beta_0,\beta_1,\ldots,\beta_k\}$ and $\beta_{m_I}$, $\beta_{M_I}$ the corresponding minimum and maximum respectively. For each such $I$ we consider the system of equations:
  \begin{equation}\label{fundeq} (\mu^{\beta_{M_I}-\beta_{m_I}}w_{m_I},\mu^{\beta_{M_I}-\beta_{m_I+1}}\lambda^{\beta_{m_I+1}-\beta_{m_I}}w_{m_I+1},\ldots,\lambda^{\beta_{M_I}-\beta_{m_I}}w_{M_I})\wedge(v_{m_I},\ldots, v_{M_I})=0
  \end{equation}
  
  \vspace{0.2cm}
  
Define $G_k$ to be the projective variety with the \emph{reduced} structure given by  the radical of the  ideal generated by all the systems of equations described via (\ref{fundeq}).
  
  \begin{rem} We do not exclude the possibility that $I=\{i\}$ in which case the corresponding equation is $w_i\wedge v_i=0$. Notice also that rather than taking (\ref{fundeq}) only for connected $I$ we could take the equations for any $I$ non-empty. However, if we let $\hat{I}\subset\underline{k}$ to  be the "connected hull" of $I$, i.e. the connected subset that ranging from $\min{I}$ to $\max{I}$ then (\ref{fundeq}) for $\hat{I}$ implies (\ref{fundeq}) for $I$.
  
  The equations (\ref{fundeq}) are obtained from the "mother" equation for $I=\underline{k}$ by forgetting about the initial and final coordinates and eliminating the $g.c.d.$ of the monomials in $\mu$ and $\lambda$.
    \end{rem}
    
    \begin{rem} The equations (\ref{fundeq}) enjoy a certain symmetry in $\mu$ and $\lambda$ owing to the fact that the set of connected subsets of $\underline{k}$ is invariant under the involution $x\ra k-x$. Notice that there are ${k+2\choose 2}$ such connected sets.
    
Clearly $(\mu,\lambda)\neq 0$ and we will localize at $\mu\neq 0$ and $\lambda\neq 0$. When we do this, the number of equations decreases.    For example when $\lambda\neq 0$ we have that (\ref{fundeq}) for $I$ implies (\ref{fundeq}) for $I\setminus \{{m_I}\}$   since $\lambda$ is invertible.

Hence for $\lambda\neq 0$ it is enough to consider only the connected sets $I$ that start at $0$. There are $k+1$ such (systems of) equations. Similarly for $\mu\neq 0$ one considers the equations corresponding to connected $I$ that end at $k$. 
    \end{rem}
  
We introduce some notation.


For each $0\leq i\leq k$ let
 \begin{equation}\label{Vieq}V_i^+:=\bigoplus_{j\leq i}V_j,\qquad V_i^-:=\bigoplus_{j\geq i} V_j.\end{equation}
 \[  C_i^{\infty},C_i^0\subset \bP(V)\times \bP(V),\]\begin{align}\label{Aieq} C_i^{\infty}:=\{([w],[v])\in \bP(V_i^+)\times \bP(V_i^-)~|~w_i\wedge v_i=0\}\\
 \label{Bieq} C_i^0:=\{[w],[v])\in \bP(V_i^-)\times \bP(V_i^+)~|~w_i\wedge v_i=0\}
 \end{align}
 Clearly, $F_i:=\bP(V_i)$ are the connected components of the fixed points set of the weighted homogeneous action. Then
 \begin{align}\label{SFi} S(F_i):=\{v\in \bP(V)~|~\lim_{\lambda \ra \infty} \lambda* v\in F_i\}=\{v=[v_0:\ldots :v_i:0\ldots:0]\in \bP(V_i^+)~|~v_i\neq 0\}\\
  \label{UFi} U(F_i):=\{v\in \bP(V)~|~\lim_{\lambda \ra 0} \lambda* v\in F_i\}=\{v=[0:\ldots :0:v_i:\ldots:v_k]\in \bP(V_i^-)~|~v_i\neq 0\}
 \end{align}
 are the stable and unstable manifolds of the fixed point components. Obviously $\overline{S(F_i)}=\bP(V_i^+)$ and $\overline{U(F_i)}=\bP(V_i^-)$. We have natural projections
 \[ S(F_i)\ra F_i\leftarrow U(F_i)
 \]
 and natural quasi-projective, regular varieties
  \[S(F_i)\times_{F_i}U(F_i)\subset \bP(V)\times \bP(V)\supset U(F_i)\times_{F_i}S(F_i).\]
  It is not hard to see that in fact for all $i$ one has:
  \[ C_i^{\infty}=\overline{S(F_i)\times_{F_i}U(F_i)},\qquad\qquad C_i^0=\overline{U(F_i)\times_{F_i}S(F_i)}.
  \]
 
\begin{prop} \label{Pn}The closure of the graph of the weighted action is the blow-up of the bihomogeneous ideal (\ref{wtf}) and this can be described as  the variety $G_{k}$ defined by (\ref{fundeq}).  The restriction of the projection $\pi_{1,2}:\bP^1\times \bP(V)\times \bP(V)\ra\bP^1\times \bP(V)$ to $G_k$ is the blow-down map. 

The exceptional divisor consists of two  separate  collections of components:
\begin{equation}\label{f1u} \bigcup_{i=0}^{k-1}\{[0:1]\}\times C_i^{\infty}
\end{equation}
lying over $\{\mu=0, \; w_k=0\}$ and 
\begin{equation}\label{f2u}
\bigcup_{i=1}^{k}\{[1:0]\}\times C_{i}^{0}
\end{equation}
 lying  over $\{\lambda=0,\; w_0=0\}$. All components have  multiplicity $1$ and dimension equal to $\dim{\bP(V)}$.

The intersection $G_k\cap \{\mu=0\}$ consists of the union  (\ref{f1u}) with $\{[0:1]\}\times C_k^{\infty}$, while $G_k\cap\{\lambda=0\}$ of the union (\ref{f2u}) with $\{[1:0]\}\times C_0^0$, each appearing with multiplicity $1$.
\end{prop}
\begin{proof} We recall that the closure of a smooth, connected (irreducible) quasi-projective variety is necessarily an irreducible variety.

Let $\Lambda$ be the graph of the action and $\overline{\Lambda}$  be its  closure in $\bP^1\times \bP(V)\times\bP(V)$.

Clearly, $\Lambda$ is smooth of dimension $\dim{V}$ and easily seen  to coincide with $G_k\cap \{\mu\neq 0\neq \lambda\}$,  an open subset of $G_k$ since for $\mu\neq0 \neq \lambda$ the equation (\ref{fundeq}) for $I=\underline{k}$ implies all the others. Since $G_k$ is Zariski closed it follows that $\overline{\Lambda}\subset G_k$. Due to the symmetry of the situation we will assume that $\lambda\neq 0$ in what follows.

If $\mu=0$ and $w_k\neq0$ then (\ref{fundeq}) for $I=\underline{k}$ implies that $v_0=\ldots=v_{k-1}=0$ and $w_k\wedge v_k=0$ and all the other equations for $I=\underline{j}$, $0\leq j< k$ are fulfilled.  Hence $([w],[v])\in C_k^{\infty}$. In fact, $G_k\bigr|_{w_k\neq 0,\lambda\neq0}$ is a graph induced by $I=\underline{k}$.  On the other hand, it is easy to see that $\Lambda$ extends to a graph over $\{\mu=0;\;w_k\neq 0\}$ and the extension is given exactly by
 $$\{[0:1]\}\times (C_k^{\infty}\setminus\{([w],[v])~|~w_k=0\}).$$ This justifies that
  $$G_k\cap\{w_k\neq0,\lambda\neq 0\}=\Lambda\cap \{w_k\neq 0,\lambda\neq 0\}$$ and that $\{[0:1]\}\times A_k$ has multiplicity $1$ since every regular subvariety has multiplicity $1$ within a regular ambient variety.

Recall that by definition  the ideal of the equations defining $G_k$ is reduced.

 We will argue that
\begin{itemize}
\item[(i)] the support\footnote{By support of a subscheme defined by an ideal sheaf $\mathcal{I}$ in a scheme $X$ we understand of course the support of $\mathcal{O}_X/\mathcal{I}$.} of $G_k$ is contained in the topological closure of $\Lambda$,  therefore the support of $G_k$ will be contained in the Zariski closure of $\Lambda$;
\item[(ii)] the support of $[G_k\cap\{\mu=0,w_k=0\}]$ and $[G_k\cap\{\lambda=0,w_0=0\}]$ respectively coincide with (\ref{f1u}) and (\ref{f2u}) respectively.
\item[(iii)] each of the subvarieties $\{[0:1]\}\times C_i^{\infty}$ and $\{[1:0]\}\times C_{i}^{0}$ appears with multiplicity $1$ in the intersection $G_k\cap \{\mu=0,w_k=0\}$ and $G_k\cap \{\lambda=0,w_0=0\}$.
\end{itemize} 

Items (i) and (ii) go together. We use induction. Again we treat just $\lambda\neq 0$. First  looking at (\ref{fundeq}) for $I=\underline{k-1}$ we deduce 
 \[G_k\cap \{\mu=0,w_k=0,w_{k-1}\neq 0\}=\{[0:1]\}\times (C_{k-1}^{\infty}\setminus \{([w],[v])~|~w_{k-1}=0\}).\]
  For $w_{k-1}=0$ use that 
  \[G_{k-1}\cap\{\mu=0,w_{k-1}\}=\{[0:1]\}\times\bigcup_{i=0}^{k-2}\{([w],[v])\in \bP(V_i^+)\times \bP(V_i^-/V_k))~|~w_i\wedge v_i=0\} \]
and the fact that $v_k$ is free since $w_k=0$, in order to conclude that
\[G_k\cap \{\mu=0,w_k=0,w_{k-1}= 0\}=\{[0:1]\}\times\bigcup_{i=0}^{k-2}\{([w],[v])\in \bP(V_i^+)\times \bP(V_i^-)~|~w_i\wedge v_i=0\} 
\]
Notice that the union on the right contains for $i=k-2$ also the missing piece from $C_{k-1}^{\infty}$, i.e.
\[\bP(V_{k-2}^{+})\times \bP(V_{k-1}^-)=\{([w],[v])\in C_{k-1}^{\infty}~|~w_{k-1}=0\}.\]

The fact that all points in (\ref{f1u}) and (\ref{f2u}) are in the topological closure of $\Lambda$ is a straightforward exercise, noticing that each $S(F_i)\times_FU(F_i)$ can be seen as pairs of points $(p,q)$ lying on the same (once) broken trajectory of the complex flow determined by the $\bC^*$ action (see \cite{HL1}, \cite{Lat} or \cite{Ci2}).

Part (iii) is a consequence of the real theory. The $\bC^*$ action induces a Morse-Bott flow on $\bP(V)$. This is obtained by considering on $\bP(V)$ the function
\[ f(L)=\Real \Tr AR_L=\Tr AR_L=\frac{\sum_{j=0}^k \beta_j|v_j|^2}{\sum_{j=0}^k|v_j|^2}
\]
where $A=\beta_0\id_{V_0}\oplus\ldots\oplus \beta_k\id_{V_k}$, $L=[v_0:\ldots:v_k]$ and $R_L$ is the orthogonal projection of $L$ in $V$ for some fixed metric. The gradient of $f$ is $L\ra A-R_LAR_L$ where the r.h.s. belongs in fact to $T_{R_{L}}\iota(\bP(V))$, $\iota$ being the  embedding $\bP(V)\ra \End(V)$ induced by taking  reflections. The integral curves of the gradient are of type $t\ra [e^{t\beta_0}v_0:\ldots:e^{t\beta_k}v_k]$ recovering thus the action for $\lambda\in (0,\infty)$.

 The flow out of the diagonal $\Delta$ of $\bP(V)\times \bP(V)$ via the gradient flow induced by $f$ in the second component has a limit in the flat topology, when $t\ra \infty$ equal to $\sum_{i=0}^k S(F_i)\times_{F_i}U(F_i)$ (see \cite{Lat, Ci2,Mi}). The main point is that the multiplicities are all equal to $1$ which will therefore also be the multiplicity of every $C_i^{\infty}$.\end{proof}

Proposition \ref{Pn} describes what happens in a single fiber.   The bihomogeneous ideal $\mathcal{I}$ from (\ref{wtf}) clearly has a fiber bundle equivalent as a sheaf of bihomogeneous ideals. We will keep the same notation we used so far also in the fiber bundle case. 

  For example, we will regard $G_k$ as an analytic current in $\bP^1\times \bP(E)\times_M\bP(E)$. Also the analytic subspaces $C_i^{\infty}\subset \bP(E_{\leq i})\times_M\bP(E_{\geq i})$ and $C_{i}^{0}\subset \bP(E_{\geq i})\times \bP(E_{\leq i})$ will give components of $G_k\cap \{\mu=0\}$ and $G_k\cap \{\lambda=0\}$ respectively.

For $[\mu:\lambda]\in \bP^1$, let $F_{[\mu:\lambda]}$ be the slice $\pi_1^{-1}(\{[\mu:\lambda]\})\cap G_k$ in the projection $\bP^{1}\times \bP(E)\times_M\bP(E)\ra \bP^1$. As we saw this is the graph of the action at "time" $[\mu:\lambda]$ when $[\mu:\lambda]\in \bC^*$.  Note also that $F_{[0:1]}$ and $F_{[1:0]}$ have quite explicit descriptions due to the last statement of Proposition \ref{Pn}.

\vspace{0.4cm}

Now let $s:M\ra \bP(E)$ be a holomorphic  section.

Just like in Section \ref{sec3}, let $\tilde{S}$ be the strict transform of $\bP^1\times s(M)$ with respect to the blow-up of $\mathcal{I}$. One can speak of a strict transform precisely because $\bP^1\times s(M)$ is  not contained in the co-support of $\mathcal{I}$. Clearly, then $\tilde{S}$ has  dimension equal to $\dim{M}+1=n+1$. 
\begin{rem}\label{altdesc} An alternative description of $\tilde{S}$ is as the closure of the graph of
\[ \bC^*\times s(M)\ra \bP(E),\qquad (\lambda, s(m))\ra \lambda* s(m)
\]
inside $\bP^1\times \bP(E)\times \bP(E)$.
\end{rem}
\begin{rem}\label{Remexcdiv} The exceptional divisor is typically defined as the intersection of $\tilde{S}$ with the exceptional divisor of the blow-up of $\mathcal{I}$, described in Proposition \ref{Pn}. It is convenient though in our context to extend the definition and call exceptional divisor the intersection  $\tilde{S}\cap (\{\mu=0\}\cup \{\lambda=0\})$, including thus also $\tilde{S}\cap \{[1:0]\}\times C_k^{\infty}$ and $\tilde{S}\cap \{[0:1]\}\times C_k^0$.
\end{rem}
Let $ \mathcal{F}l(G_k)$  be the space of flat currents in $\bP^1\times \bP(E)\times_M\bP(E)$ with support contained in the support of $G_k$. 
\begin{theorem}\label{ths5} The following equalities of currents in $\bP^1\times\bP(E)\times_M \bP(E)$ holds:
\begin{align} \label{Fmu1} F_{[\mu:\lambda]}\wedge \tilde{S}-F_{[0:1]}\wedge \tilde{S}=\partial\bar{\partial} T_1[\mu:\lambda],\qquad \forall [\mu:\lambda]\in \bC^*\cup\{[1:0]\}\\
\label{Fmu2} F_{[\mu:\lambda]}\wedge \tilde{S}-F_{[1:0]}\wedge \tilde{S}=\partial\bar{\partial} T_2[\mu:\lambda],\qquad\forall [\mu:\lambda]\in \bC^*\cup\{[0:1]\}
\end{align}
where $T_1:\bP^1\setminus\{[0:1]\}\ra \mathcal{F}l(G_k)$ and $T_2:\bP^1\setminus\{[1:0]\}\ra  \mathcal{F}l(G_k)$ are  continuous maps.
\end{theorem}
\begin{proof} The proof is similar with that of Theorem \ref{PL3} except that one applies Hardt Slicing Theorem for the projection $\pi_1$ on $N:=\bP^1\times \bP(E)\times_M\bP(E)$  in order to be able to speak of currents in the first place, the current being $T:=G_k$.

Then one wedges with $\tilde{S}$. This last step needs a bit of care.  The analogue of Proposition \ref{Psp1} is Theorem \ref{Th95} in the next section but it turns out one need not be that explicit at this point. The intersection of $\tilde{S}$ with $\bigcup_{i=0}^{k-1}\{[0:1]\}\times C_i^{\infty}$ has dimension at most $n$, being the exceptional divisor of the strict transform of $\bP^1\times s(M)$ which has dimension $n+1$. It will have dimension exactly $n$ if it is non-empty.

 The intersection of $\tilde{S}$ with $\{[0:1]\}\times C_k^{\infty}$ is either the empty set, if $s^{-1}(S(F_k))=\emptyset$ or it is the closure of $\{[0:1]\}\times s(s^{-1}(S(F_k)))\times_{F_{k}}s_{\infty}(s^{-1}(S(F_k)))$ where $s_{\infty}(p)=\pi_k^s(s(p))$ for all $p\in s^{-1}(S(F_k))$, $\pi_k^s:S(F_k)\ra F_k$ being the natural projection.

Now $s^{-1}(S(F_k))$ is an open subset of $M$, the complement of $\{s_k=0\}$ where $s=(s_0,\ldots,s_k)$. On the other hand
\begin{equation}\label{identeq1}\{[0:1]\}\times s(s^{-1}(S(F_k)))\times_{F_{k}}s_{\infty}(s^{-1}(S(F_k)))\simeq\Imag {s_{\infty}}\subset \bP(E_k)
\end{equation}
with $s_{\infty}:M\setminus \{s_k=0\}\ra F_k$, $s_{\infty}(p)=[s_k(p)]$ and the identification in (\ref{identeq1}) is given by projection onto the last coordinate. The closure of the image of $s_{\infty}$ is the blow-up of $\{s_k=0\}$ (see Remark \ref{beforetric}). Hence $\tilde{S}\wedge\{[0:1]\}\times C_k^{\infty}\simeq \Bl_{\{s_k=0\}}(M)$ and therefore this piece is also of dimension $n$, if non-empty. 

The reasoning is quite analogous for $\tilde{S}\wedge \{[1:0]\}\times C_0^0\simeq \Bl_{\{s_0=0\}}(M)$.
\end{proof}

For $0\leq i\leq k$, we define the following analytic varieties:
\begin{align} \label{Siinfty}\tilde{S}_i^{\infty}:=\pi_{2,3}((\{[0:1]\}\times C_i^{\infty})\cap \tilde{S})\subset \bP(E_{\leq i})\times_M\bP(E_{\geq_i})\subset\bP(E)\times_M\bP(E),\\ \label{Sizero}\tilde{S}_i^0:=\pi_{2,3}((\{[1:0]\}\times C_{i}^{0})\cap \tilde{S})\subset \bP(E_{\geq i})\times_M\bP(E_{\leq_i})\subset \bP(E)\times_M\bP(E).
\end{align}
The projection onto the first coordinate in $\bP(E)\times_M\bP(E)$ induces maps:
\[ \pi_1:\tilde{S}_i^{\infty}\ra s^{-1}(\bP(E_{\leq i})),\qquad \pi_1:\tilde{S}_i^0\ra s^{-1}(\bP(E_{\geq i})).
\]
For an alternative description of the components $\tilde{S}_i^{\infty}$ see the next section.
\begin{cor}\label{cormt3} Let $\omega\in \Omega^*(\bP(E))$  a form which is $\partial$ and $\bar{\partial}$-closed and let $s:M\ra \bP(E)$ be a section. Then there exist double transgression formulas:
\begin{align}\label{Feq3} s^*\omega-\lim_{\lambda\ra \infty}(\lambda* s)^*\omega=\partial\bar{\partial} T_1(s,\omega)\\
\label{Feq4} s^*\omega-\lim_{\lambda\ra 0}(\lambda* s)^*\omega=\partial\bar{\partial} T_2(s,\omega)
\end{align}
for some flat currents $T_1(s,\omega)$ and $T_2(s,\omega)$. The limits are described as follows:
\[\lim_{\lambda\ra \infty}(\lambda* s)^*\omega=\sum_{i=0}^k (\pi_1)_*(\pi_2^*\omega\wedge \tilde{S}_i^{\infty}),\quad \;\;\lim_{\lambda\ra 0}(\lambda* s)^*\omega=\sum_{i=0}^k (\pi_1)_*(\pi_2^*\omega\wedge \tilde{S}_i^{0}).
\]

 The current $\lim_{\lambda\ra \infty}(\lambda* s)^*\omega$ is smooth if $s(M)\subset S(F_j)$ for some $j$ and the current $\lim_{\lambda\ra 0}(\lambda* s)^*\omega$ is smooth if $s(M)\subset U(F_j)$ for some $j$. When this is the case, the currents $T_1(s,\omega)$ and $T_2(s,\omega)$ are smooth as well.
\end{cor}
\begin{proof} Consider the action-graph of $s$, 
\[\{(\lambda,s(m),\lambda*s(m))~|~\lambda\in \bC^*,\;m\in M\}\subset \bP^1\times \bP(E)\times_M\bP(E)\] whose closure is $\tilde{S}$. We use Theorem \ref{ths5} and push-forward via the proper projection $\pi_{2,3}$
the equations (\ref{Fmu1}) and (\ref{Fmu2}) to two double transgressions in $\bP(E)\times_M\bP(E)$.

The current $F_{[\mu:\lambda]}\wedge \tilde{S}$ pushes forward to 
\begin{align}\label{feq1}\left \{ (s(m),\left(\frac{\lambda}{\mu}\right)*s(m))~|~m\in M\right\} &  \mbox{ when } \;  [\mu:\lambda]\in \bC^*\subset\bP^1\\
\label{feq2} \lim_{\lambda\ra \infty} \{ (s(m),{\lambda}*s(m))~|~m\in M\} & \mbox{ when } \; \mu=0\\
 \label{feq3} \lim_{\mu\ra 0} \{ (s(m),\mu*s(m))~|~m\in M\} & \mbox{ when }  \; \lambda=0
\end{align}
Indeed (\ref{feq1}) is clear. The continuity of Hardt's Theorem ensures that $F_{[1:\lambda]}\wedge\tilde{S}\ra F_{[1:0]}\wedge \tilde{S}$ when $\lambda\ra 0$. By pushing this forward via $\pi_{2,3}$ one gets (\ref{feq2}). Then (\ref{feq3}) is analogous.

 For the transition to (\ref{Feq3}) and (\ref{Feq4}) one proceeds as in the proof of Corollary \ref{cormt1}.

If $s(M)\subset S(F_j)$ then 
\[\lim_{\lambda\ra \infty}\Gamma_{\lambda*s}=s(M)\times_M\tilde{s}_{\infty}(M)\simeq M
\]
where $s_{\infty}(m)=[s_j(m)]\in \bP(E_j)=F_j$.
\end{proof}

\vspace{0.2cm}

\begin{example} The strict transform $\tilde{S}$ of $\bP^1\times s(M)$ exists always. However, if $s(M)$ is contained in a connected component $\bP(E_i)=F_i$ of the fixed point locus, then by Remark \ref{altdesc}, $\tilde{S}=\bP^1\times s(M)\times_Ms(M)$. Then the pushforward via $\pi_{2,3}$ of (\ref{Fmu1}) and (\ref{Fmu2}) give a trivial double transgression $0=0$.
\end{example}

\begin{rem}\label{Gr19} If $G\subset \bP(E)$ is a flow-invariant, proper,  regular fiber subbundle  then one can apply the machinery to $\partial$ and ${\bar{\partial}}$ closed forms $\omega\in \Omega^*(G)$  and holomorphic sections $s:M\ra G$ and the results hold just the same as in Corollary \ref{cormt3}.
\end{rem}

A particular case of Remark \ref{Gr19} is the case of the Pl\"ucker embedding $\Gr_{k}(E\oplus F)\hookrightarrow \bP(\Lambda^k(E\oplus F))$ where $k=\dim{E}$. The natural action $\bC^*\times \Hom(E,F)$ extends to an action of $\bC^*$ on $\Gr_k(E\oplus F)$ that intertwines with the weighted homogeneous action of $\bC^*$ on $\bP(\Lambda^k(E\oplus F))$ induced by the direct sum of the actions
\[\bC^*\times \Lambda^iE\otimes \Lambda^jF\ra \Lambda^iE\otimes \Lambda^jF,\qquad (\lambda, e\otimes f)\ra \lambda^je\otimes f
\]
where $i+j=k$. It is useful to describe the fixed points, stable and unstable manifolds of the $\bC^*$ action on $\Gr_k(E\oplus F)$.
\begin{prop} \label{Fixpeq} The fixed points manifolds of the action of $\bC^*$ on $\Gr_k(E\oplus F)$ are
\[ F_i:=\bigcup_{i}\Gr_i(E)\times \Gr_{k-i}(F).
\]
The stable and unstable manifolds are:
\[ \Sigma_i:=\{L\in\Gr_k(E\oplus F)~|~\dim{L\cap E}=i\},\qquad \Upsilon_i:=\{L\in\Gr_k(E\oplus F)~|~\dim{L\cap F}=k-i\}
\]
with projections
\[\Sigma_i\ra F_i\leftarrow \Upsilon_i,\qquad L\ra (L\cap E,\pi^F(L)),\;\mbox{resp.}\;\; (\pi^E(L),L\cap F)\leftarrow L
\]
where $\pi^E,\pi^F$ are the projections onto the factors of $E\oplus F$.
\end{prop}
\begin{proof} See \cite{HL2}.
\end{proof}
As a consequence of Corollary \ref{cormt3} we get:
\begin{prop}\label{grpr} Let $s:M\ra \Hom(E,F)$ be a holomorphic section and $\omega\in \Omega^*(\Gr_k(E\oplus F))$ a smooth form, $\partial$ and $\bar{\partial}$ closed. Then there exists a double transgression
\[ s^*\omega-\lim_{\lambda\ra \infty}(\lambda s)^*\omega=\partial\bar{\partial} T(s,\omega)
\]
where $\displaystyle\lim_{\lambda\ra \infty}(\lambda s)^*\omega$ is a sum of $k+1$ flat currents each of which supported on 
\[\Sigma_i(s):=\{p\in M~|~\dim{\Ker s(p)}=i\}=s^{-1}(\Sigma_i).\] If $\dim{\Ker s(p)}=i$ for all $p\in M$, then $\Sigma_i(s)$ and the form $\lim_{\lambda\ra \infty}(\lambda s)^*\omega$ is smooth.
\end{prop}
\begin{rem} If $s\pitchfork \Sigma_i$, for all $i$, one can perform explicit computations for the current $\displaystyle\lim_{\lambda\ra \infty}(\lambda s)^*\omega$ (see \cite{CiHl3}).
\end{rem}

\section{Weighted blow-ups and the limit of the kernels}\label{Sec9}
Proposition \ref{Psp1} completely describes the limiting kernels when $\lambda\ra \infty$ in the "homogeneous" case of a single weight. We would like to have an analogous result in the context of weighted homogeneous actions and give an alternative description of the components $\tilde{S}_i^{\infty}$ from (\ref{Siinfty}) and their cousins $\tilde{S}_i^0$. More concretely the starting question is: can each $\tilde{S}_i^{\infty}$ be described by some sort of cone construction over the exceptional divisor of some blow-up? It turns out that the situation is  more complicated in the weighted-homogeneous case and, in a sense, this can be "blamed" on the non-local character of the orbits, or better said on the non-Hausdorff nature of the orbit space.

By definition, the components $\tilde{S}_i^{\infty}$ lie within $\overline{S(F_i)\times_{F_i}U(F_i)}$. As opposed to $S(F_i)\times_{F_i}U(F_i)$ which is biholomorphically equivalent with a vector bundle over $F_i=\bP(E_i)$, the closure $\overline{S(F_i)\times_{F_i}U(F_i)}$ does not fiber over $F_i$. So it seems reasonable to first understand the intersection 
\[\tilde{S}_i^{\infty}\cap (S(F_i)\times_{F_i}U(F_i))
\] 
which we will call the "affine" part of $\tilde{S}_i^{\infty}$. In fact, if $s\pitchfork S(F_i)$ for all $i$, then the real theory \cite{Ci2} implies that the part of $\tilde{S}_i^{\infty}$ that lies in $\overline{S(F_i)\times_{F_i}U(F_i)}\setminus S(F_i)\times_{F_i}U(F_i)$ is "negligible", more precisely it has Hausdorff dimension at most $n-1$. This, however is not  true anymore if we remove the transversality hypothesis. 
\begin{example} Let $B:=B(0,1)$ be the open unit ball in $\bC$. Consider $\bP^3(\bC)$ with the action of $\bC^*$ given by the weights $\beta_i=i$, $i\in \{0,1,2,3\}$. Let $s:B\ra \bP^3(\bC)$ be the section:
\[ s(z)=[1+z:z^{a_1}:z^{a_2}:z^{a_3}]
\]
where $0<a_1<a_2<a_3$ positive integer satisfying the inequalities
\[ a_2>2a_1,\qquad a_3>\frac{3}{2}a_2,\qquad a_3+a_1>2a_2.
\]
Note that the first two also imply that $a_3>3a_1$. The only point $p\in B(0,1)$ where at least one of the projective coordinates of $s(p)$ vanishes is $p=0$. We compute all possible limits
 \begin{equation}\label{divideeq} \lim_{\substack{z\ra 0\\ \lambda \ra \infty}}[(1+z):\lambda z^{a_1}:\lambda^2z^{a_2}:\lambda^3z^{a_3}]
 \end{equation}
 by separating into the cases:
 \begin{itemize}
  \item[(i)] $\lim\lambda z^{a_1}= \mbox{ finite}\Rightarrow\left\{\begin{array}{ccccc} \lim\lambda^2z^{a_2} &= & 0& \mbox{since} & \lambda^2z^{a_2}=(\lambda z^{a_1})^2z^{a_2-2a_1} \\
    \lim \lambda^3z^{a_3}& = & 0& \mbox{since} &  \lambda^3z^{a_3}=(\lambda z^{a_1})^3z^{a_3-3a_1} \end{array}\right.$
    
  This implies that all the points $[1:x:0:0]$ are limits ($x\in \bC$).
  \item[(ii)] $\lambda z^{a_1}\ra\infty$ and $\lim\lambda z^{a_2-a_1}=\lim(\lambda z^{a_1}) z^{a_2-2a_1}= \mbox{ finite} \Rightarrow \lim\lambda^2z^{a_3-a_1}= 0$ since $(\lambda^2 z^{a_3-a_1})=(\lambda z^{a_2-a_1})^2z^{a_3+a_1-2a_2}$.
 
Divide (\ref{divideeq}) by $\lambda z^{a_1}$ to conclude that  all the points $[0:1:x:0]$ are limits.
 
  \item[(iii)] $\lim \lambda z^{a_2-a_1}= \infty$ and $\lim\lambda z^{a_3-a_2}=\lim(\lambda z^{a_2-a_1})\cdot z^{a_3+a_1-2a_2}=\mbox{ finite}$.   Notice also that $\lim\lambda z^{a_2-a_1}= \infty \Rightarrow \lim\lambda^2 z^{a_2}= \lim(\lambda z^{a_2-a_1})^2z^{-a_2}= \infty$.
  
  Divide (\ref{divideeq}) by $\lambda^2z^{a_2}$ to get all points $[0:0:1:x]$ as limits.
  \item[(iv)] $\lim\lambda z^{a_3-a_2}=\infty\Rightarrow \lim\lambda^3z^{a_3}=\lim(\lambda z^{a_3-a_2})^3 z^{-2a_3+3a_2}=\infty$. Then the only limit point is $[0:0:0:1]$.
 \end{itemize}
 Therefore the set of all possible limits is a union of three projective lines denoted $X$ in $\bP^3(\bC)$. We conclude that $\tilde{S}^{\infty}$ is supported in
  \[\{[1:0:0:0]\}\times X \cup s(B(0,1))\times \{[0:0:0:1]\}\]
  The main point about this example is that $\tilde{S}_0^{\infty}$ which is supported in $\{[1:0:0:0]\}\times X$ cannot be  entirely obtained from its affine part, namely $\tilde{S}_0^{\infty}\cap \{[1:0:0:0]\}\times \{z_0\neq 0\}$. Neither is this captured by the intersections $\tilde{S}^{\infty}\cap S(F_i)\times_{F_i}U(F_i)$ since $F_0\not\subset S(F_i)$  $\forall~ i\geq 1$. 
\end{example}

\vspace{0.3cm}

For the rest of the section we partially answer the starting  question, namely we describe the affine part of $\tilde{S}_i^{\infty}$ as the exceptional divisor of a weighted cone  of a weighted blow-up.

We start with some simple observations concerning several related processes. Let $V$ be a vector space with a decomposition
 \[V=V_1\oplus \ldots\oplus V_k\] and with a weighted action of $\bC^*$ on it  with positive integer weights only $0<\beta_1<\ldots<\beta_k$. This extends to $\{0\}\in \bC$ trivially and we will still speak of an "action" $\bC\times V\ra V$. Then one can consider the closure $G^V$ of the graph  of the action $\bC^*\times V\ra V$ in $\bP^1\times V\times V$. Under these hypothesis it is not hard to see that in fact $G^V$ is a algebraic subset, e.g.  by "compactifying" the  action of $\bC^*$ to an action on $\bP(\bC\oplus V)$ and proceeding as in the previous section.  By projecting $G^V$ onto the last two coordinates, i.e. on $V\times V$ we get an algebraic set $\pi_{2,3}(G^V)\subset V\times V$ since this projection is proper.
 
 Notice that rather than taking the graph of the action $\bC\times V\ra V$, one can do the same with the graph of the action $\bC\times A\ra V$ where $A\subset V$ is an analytic subset and get an analytic subset $G^A$ of $\bP^1\times A\times V$. Indeed, one does not need to close $A$ in $\bP(\bC\oplus V)$, which might be problematic if $A$ is not algebraic, but one can consider the closure of the graph of the action, only that this time in $\bP^1\times V\times \bP(\bC\oplus V)$. This can again be seen as a blow-up of a sheaf of ideals (the restriction of the sheaf $\mathcal{I}$ in (\ref{wtf}) to the open set $\bP^1\times V$) and the closure of the graph of $\bC\times A\ra \bP(\bC\oplus V)$ is just the strict transform of $\bP^1\times A$ with respect to this blow-up. The restriction $G^A$ of this analytic set to the open set $\bP^1\times A\times V$ is then of course analytic. 
 
We will give now another description to the analytic set $\pi_{2,3}(G^A)\subset A\times V$.

Let $\beta=(\beta_1,\ldots,\beta_k)$. The weighted projective space $\bP^{\beta}(V)$ is the analytic space obtained by taking the quotient $(V\setminus\{0\})/\bC^*$. Alternatively, one can work algebraically \cite{Dol} and take $\Proj[X_1,\ldots X_k]$ where $\deg{X_i}=\beta_i$ and since this is a projective variety, one can take the associated analytic variety via the GAGA \cite{Se} correspondence. 

The weighted blow-up of the origin in $V$ is by definition the closure in $V\times \bP^{\beta}(V)$ of the graph of the projection:
\[ V\setminus \{0\}\ra \bP^{\beta}(V)
\]
This gives an analytic subset in $V\times \bP^{\beta}(V)$ whose equations can be described as follows. An analytic subset of $V\times \bP^{\beta}(V)$ with coordinates $(v,[w])$ (here $[w]$ with $w\in V$ is an equivalence class just like for the usual projective space) is given by equations which are weighted homogeneous with weights $\beta=(\beta_1,\ldots,\beta_k)$ in the variables $w=(w_1,\ldots,w_k)$. The relation $v\in [w]$ translates into the following equations. For every pair $(\beta_i,\beta_j)$ with $i\neq j$ let
 \[ \beta_i^j:=\beta_i/\gcd(\beta_i,\beta_j)\qquad  \beta_j^i:=\beta_j/\gcd(\beta_i,\beta_j)\]
Obviously $\beta_j\beta_i^j=\beta_i\beta_j^i$. Then the equations describing the blow-up are:
\[ v_i^{\beta_j^i}w_j^{\beta_i^j}=v_j^{\beta_i^j}w_i^{\beta_j^i},\qquad 1\leq i\neq j\leq k
\]
where, for notational convenience, $v_i$ and $w_j$ each represent here just one coordinate of the possibly higher dimensional vector spaces $V_i$ and $V_j$.

\begin{rem} The reason why one has to pass to $\gcd$ of the weights has to do with the gap sheaf and Ritt's Lemma \cite{Fi}, page 41. 
\end{rem}

The weighted blow-up of $0$ in $V$ can also be applied to a reduced  analytic subset $A\subset V$, under the adequate name of weighted strict transform. By definition, this is the closure in $V\times \bP^{\beta}(V)$ of the graph of  $\pi^{\beta}\bigr|_{A\setminus \{0\}}$ where
\[ \pi^{\beta}:V\setminus \{0\}\ra \bP^{\beta}(V)
\] 
is the natural projection.

This gives an analytic space $\Bl_{0}^{\beta}(A)\subset V\times \bP^{\beta}(V)$. One can take the  (weighted) affine cone of $\Bl_{0}^{\beta}(A)$, meaning that one looks at the same equations that define $\Bl_{0}^{\beta}(A)$ but forgets about being weighted-homogeneous in the $w$ variable and thus obtains a set $G^A_1$ in $V\times V$. 
\begin{lem} \label{Lemga} As complex spaces $G^A_1=\pi_{2,3}(G^A)$.
\end{lem}
\begin{proof} Unwinding the definitions it all comes down to  showing that the closure of $\{(v,\lambda* v)~|~\lambda \in \bC, v\in A\}$ in $V\times V$ is the projection of the closure of $\{(\lambda, v,\lambda* v)~|~\lambda \in \bC, v \in A \}$ in $\bP^1\times V\times V$ onto the last two variables. This is obvious.
\end{proof}

There exists a special analytic subspace of $G^A_1$ and this is the exceptional divisor, i.e. the intersection $EG^A_1:=G^A_1\cap (\{0\}\times V)$. 

\vspace{0.3cm}

Everything we said makes sense if we replace $V$ by a vector bundle 
\[E'=E_1\oplus \ldots \oplus E_k\ra M\]  endowed with a linear weighted action of $\bC^*$ with weights $\beta_1,\ldots,\beta_k$. The weighted blow-up of the zero section is an analytic subset of $E'\times_M\bP^{\beta}(E')$. One can then take the (fiberwise, weighted) affine cone in order to get an analytic subset of $E'\times_ME'$. The affine cone of the weighted blow-up makes sense for an analytic subset $A\subset E'$. We will denote again by $G^A_1$ the resulting analytic subset in $E'\times E'$ and by $EG^A_1$ the exceptional divisor.  This is related to the results of the previous section as follows. 

The zero section $0\subset E'$ can be seen as  the fixed point set $\bP(\bC)\hookrightarrow\bP(\bC\oplus E')$ via the standard embedding $E'\hookrightarrow \bP(\bC\oplus E')$. The action on $\bP(\bC\oplus E')$ is the obvious one. In this context, if $A=s(M)$ is the image of a section $s:M\ra E'\subset \bP(\bC\oplus E')$ then  with the notation (\ref{Siinfty}) the following holds:
\begin{equation}\label{eqEG} EG^{s(M)}_1=\tilde{S}_0^{\infty}\cap (\bP(\bC\oplus E')\times_M E')\subset \bP(\bC)\times_ME' 
\end{equation}
where $E'\times_M E'\subset \bP(\bC\oplus E')\times_ME'$ via the standard embedding. To see that (\ref{eqEG}) holds use Lemma \ref{Lemga} and the fact that the rest of the components of the exceptional divisor of $\tilde{S}$ are supported in $\bP((\bC\oplus E')_{\leq i})\times_M\bP((\bC\oplus E')_{\geq i})$ with $i\geq 1$ hence their intersection with the open $E'\times_M E'$ is empty. 

This takes care of the case $E_0=\bC$. 

For higher dimensional $E_0$ we take $\tau_0\ra \bP(E_0)$ to be the tautological bundle and consider the biholomorphism onto the image (we will refer to this as a "chart"):
\begin{equation}\label{HE0} \Hom(\tau_0,\pi_0^*E')\ra \bP(E_0\oplus E'),\qquad (\ell_p,A:\ell_p\ra E'_p) \ra \Gamma_{A},\qquad \ell_p\in \bP(E_{0,p}),\;p\in M
\end{equation}
where  $\pi_0:\bP(E_0)\ra M$ is the projection. The image of (\ref{HE0}) is $ \bP(E_0\oplus E')\setminus \bP(E')$. Notice that the weighted action of $\bC^*$ on $E'$ induces an action on $\Hom(\tau_0,\pi_0^*E')$ which is linear in the fibers and turns (\ref{HE0}) into an equivariant map. The  induced action is obtained by letting $\bC^*$ act trivially on $\tau_0$:
\[\lambda*(\ell_p,A):=(\ell_p,( \lambda^{\beta_1}A_1,\ldots,\lambda^{\beta_k}A_k)), \qquad \ell_p\in \bP(E_{0,p}),\;p\in M
\]
where $A_i:\ell_p\ra E'_i$ are the components of $A$. To see that $\Gamma_{\lambda* A}=\lambda*\Gamma_A$ take $v_0\in \tau_{[v_0]}$ a non-zero vector. The graph of $\lambda* A$  is a line which can be represented, independently of the choice of $v_0\in \ell_p$ by 
\[[v_0:\lambda^{\beta_1}A_1(v_0):\ldots:\lambda^{\beta_k}A_k(v_0)]=\lambda*[v_0:A_1(v_0):\ldots:A_k(v_0)].\]

Once one has an action on a vector bundle, it makes sense  to perform the weighted blow-up of the zero section in $\Hom(\tau_0,\pi_0^*E')$ and take the (fiberwise) \emph{affine cone}  in order to get a subspace of $\Hom(\tau_0,\pi_0^*E')\times_{\bP(E_0)} \Hom(\tau_0,\pi_0^*E')$.  We  transfer this analytic set to $\bP(E_0\oplus E')\times_M \bP(E_0\oplus E')$ via the product of the chart (\ref{HE0}) with itself 
\[\Hom(\tau_0,\pi_0^*E')\times_{\bP(E_0)} \Hom(\tau_0,\pi_0^*E')\substack{\sim\\\longrightarrow} (\bP(E_0\oplus E')\setminus \bP(E'))\times_M((\bP(E_0\oplus E')\setminus \bP(E'))
\]

By definition, this is the  affine cone of the weighted blow-up of $\bP(E_0)$ inside the open set $\bP(E_0\oplus E')\setminus \bP(E')$. The projection onto the first component has generically one-dimensional fibers, which is to be expected for a cone over a blow-up. The exceptions are the fibers corresponding to points on $\bP(E_0)\subset \bP(E_0\oplus E')\setminus \bP(E')$.

\vspace{0.2cm}

 Now, if $A\subset \bP(E_0\oplus E')\setminus \bP(E')$ is a closed analytic set, by working in $\Hom(\tau_0,\pi_0^*E')$, one can take the weighted strict transform of $A$ and the fiberwise affine cone and (after returning to $\bP(E_0\oplus E')\times_M\bP(E_0\oplus E')$) obtain thus an analytic space  $G^A_{1}\in(\bP(E_0\oplus E')\setminus \bP(E'))\times_M(\bP(E_0\oplus E')\setminus \bP(E'))$ with exceptional locus:
 \[ EG^A_1:=G^A_{1}\cap [\bP(E_0)\times_M(\bP(E_0\oplus E')\setminus \bP(E'))]\subset \bP(E_0)\times_M(\bP(E_0\oplus E')\setminus \bP(E')).
\]

\begin{prop}\label{Pr} Let $s:M\ra \bP(E_0\oplus E')$ be a section and $A:=s(M)\cap  (\bP(E_0\oplus E')\setminus \bP(E'))$. Then
\[EG^{A}_1=\tilde{S}_0^{\infty}\cap \left(\bP(E_0\oplus E')\setminus \bP(E')\times_M \bP(E_0\oplus E')\setminus \bP(E')\right)
\]
\end{prop}
\begin{proof}  Lemma \ref{Lemga} has a fiber bundle correspondent in this context. Let $\tilde{E}:=\Hom(\tau_0,\pi_0^*E')$, a fiber bundle over $\bP(E_0)$ with a linear action of $\bC^*$, extendable to $\bC$. Let $A\subset \tilde{E}$ be an analytic subset. Then one can take the closure of the graph $G^A$ of $\bC\times A\ra \tilde{E}$ in $\bP^1\times \tilde{E}\times_{\bP(E_0)}\tilde{E}$. This is an analytic subset and the proper projection onto the last two components, gives an analytic set which alternatively can be described as the affine cone over the weighted strict transform of the (weighted) blow-up of the zero section in $\tilde{E}$. Then the two exceptional divisors coincide as well.  

Notice  that  $M_1:=s^{-1}({\bP(E_0\oplus E')\setminus \bP(E')})$ is open in $M$ and $s_1:=s\bigr|_{M_1}$ is a holomorphic section such that $A=s_1(M_1)$. Let $\hat{S}^{\infty}_0$ be the intersection of the closure of \[\{s_1(m),\lambda*s_1(m)~|~m\in M_1\}\] in ${\bP(E_0\oplus E')\setminus \bP(E')}\times_{M_1}{\bP(E_0\oplus E')\setminus \bP(E')}$ with  $\bP(E_0)\times_{M_1}{\bP(E_0\oplus E')\setminus \bP(E')}$. It is this object that has a description as the exceptional divisor of the affine cone of a weighted strict transform as argued in the previous paragraph.

On the other hand, the equality  $\hat{S}_0^{\infty}=\tilde{S}_0^{\infty}\cap \left(\bP(E_0\oplus E')\setminus \bP(E')\times_M \bP(E_0\oplus E')\setminus \bP(E')\right)$ is obvious since the points in the image of  $s\bigr|_{M\setminus M_1}$ cannot converge to a point on $\bP(E_0)$.
\end{proof}

We will use a similar recipe as in Propostion \ref{Pr} in order to get an analogous result about $\tilde{S}_i^{\infty}$ for $i>0$.

\vspace{0.3cm}

 Let $E:=E_0\oplus E'$, and $\tau_i\ra \bP(E_i)$ be the $i$-th line tautological bundle with projection $\pi_i:=\pi\bigr|_{\bP(E_i)}$. Just as before, there exists a biholomorphism (chart):
\begin{equation}\label{eqch2} \Hom(\tau_i,\pi_i^*E_{<i}\oplus \pi_i^*E_{>i})\ra \bP(E)\setminus \bP(E_i'),
\end{equation}
where $E_i':=E_{<i}\oplus E_{>i}$. It becomes equivariant if $\bC^*$ acts on $E_{<i}\oplus E_{>i}$ as follows:
\[ \lambda*(\ldots, v_j,\ldots)= (\ldots,\lambda^{\beta_j-\beta_i}v_j,\ldots),\qquad \forall j\neq i.
\]
Notice that the weights are negative on $E_{<i}$. Rather than taking the weighted  blow-up of the zero section (as a vector bundle over $\bP(E_i)$) in $\Hom(\tau_i,\pi_i^*E_{<i}\oplus \pi_i^*E_{>i})$, one blows-up the invariant vector subbundle $\Hom(\tau_i,\pi_i^*E_{<i})$. This means taking the analytic space 
\[ \Bl_i^{w}:=\{\left((v_{<i},v_{>i}),[w]\right)\in \Hom(\tau_i,\pi_i^*E_{<i}\oplus \pi_i^*E_{>i})\times_{\bP(E_i)} \bP^{\beta_{>i}}(\Hom(\tau_i,\pi_i^*E_{>i}))~|~v_{>i}\in [w]\}.
\]
where here $\beta_{>i}:=(\beta_{i+1}-\beta_i,\ldots,\beta_{k}-\beta_i)$. The (fiberwise) affine cone over this weighted blow-up is an analytic subset of $\Hom(\tau_i,\pi_i^*E_{<i}\oplus \pi_i^*E_{>i})\times_{\bP(E_i)}\Hom(\tau_i, \pi_i^*E_{>i})$ with exceptional divisor contained in  $\Hom(\tau_i,\pi_i^*E_{<i})\times_{\bP(E_i)}\Hom(\tau_i,\pi_i^*E_{>i})$. We use now the chart, i.e. the product of (\ref{eqch2}) with itself:
\begin{equation}\label{homtau}\Hom(\tau_i,\pi_i^*E_{<i}\oplus \pi_i^*E_{>i})\times_{\bP(E_i)} \Hom(\tau_i,\pi_i^*E_{<i}\oplus \pi_i^*E_{>i})\substack{\sim\\ \longrightarrow} (\bP(E)\setminus \bP(E_i'))\times_M(\bP(E)\setminus \bP(E_i'))
\end{equation}
to push this to an analytic subset in $\bP(E)\times_M\bP(E)$.  By definition, this  is  the  affine cone of the weighted blow-up of $S(F_i)=\bP(E_{\leq i})\setminus \bP(E_{<i})$ inside $\bP(E)\setminus \bP(E_i')$. 

The exceptional divisor gets  mapped biholomorphically by (\ref{homtau}) to
 \begin{equation}\label{EBli0}\EBl_i^w\subset(\bP(E_{\leq i})\setminus \bP(E_{<i}))\times_M (\bP(E_{\geq i})\setminus\bP(E_{>i})).\end{equation}
  \begin{lem} With the same notation as in (\ref{SFi}) and (\ref{UFi})
 \begin{equation}\label{EBli} \EBl_i^w=S(F_i)\times_{F_i}U(F_i)
 \end{equation}
 \end{lem}
 \begin{proof} Note that in the r.h.s. of (\ref{EBli0}) the fiber product is over $M$ while in r.h.s of (\ref{EBli}) it is over $F_i=\bP(E_i)$. The exceptional divisor $\EBl_i^w$ is the image of $\Hom(\tau_i,\pi_i^*E_{<i})\times_{\bP(E_i)}\Hom(\tau_i,\pi_i^*E_{>j})$ under the chart map to $\bP(E)\setminus \bP(E_i')\times_M \bP(E)\setminus\bP(E_i')$. One need only observe then that
 \[\Hom(\tau_i,\pi_i^*E_{<i})\simeq \bP(E_{\leq i})\setminus \bP(E_{<i})=S(F_i),
 \]
 and
 \[ \Hom(\tau_i,\pi_i^*E_{<i})\simeq \bP(E_{\geq i})\setminus \bP(E_{>i})=U(F_i)
 \]
 via (\ref{eqch2}).
 \end{proof}
 This intermediary step gives a hint on how to proceed.  By working in the chart (\ref{eqch2}), every analytic subset $A\subset \bP(E)\setminus \bP(E_i')$ has a weighted strict transform with respect to the weighted blow-up of $S(F_i)$ inside $\bP(E_i)\setminus \bP(E_i')$. The affine cone over this weighted strict transform $(G^A_1)_i$ lives in $(\bP(E)\setminus \bP(E_i'))\times_M(\bP(E)\setminus \bP(E_i'))$ while its exceptional divisor is by definition:
 \[ (EG^A_1)_i:=(G^A_1)_i\cap  \EBl_i^w.
 \]
 
 For the next result we will assume that $s(M)$ is not contained in any of the stable manifolds $S(F_i)$. This particular case is disposed of in Corollary \ref{cormt3}. The fact that $s(M)\not\subset S(F_i)$ implies that the weighted strict transform of $s(M)$ when we take the weighted blow-up of $S(F_i)$ is well-defined.
 \begin{theorem} \label{Th95} The strict transform $\tilde{S}$ of $\bP^1\times s(M)$ in the blow-up with respect to the sheaf of ideals (\ref{wtf}) intersected with the open and dense (in $C_i^{\infty}$) set $S(F_i)\times_{F_i}U(F_i)$  coincides with the exceptional divisor of the affine cone of the  weighted strict transform of $A=s(M)\cap \bP(E)\setminus \bP(E_i')$ in the weighted blow-up of $S(F_i)$ inside $\bP(E)\setminus \bP(E_i')$, in other words
 \[\tilde{S}_i^{\infty}\cap (\bP(E)\setminus \bP(E_i')\times_{M}\bP(E)\setminus \bP(E_i'))=(EG^A_1)_i
 \] 
 \end{theorem}
 \begin{proof} Let $\tilde{E}^-:=\Hom(\tau_i,\pi_i^*E_{<i})$ and $\tilde{E}^+:=\Hom(\tau_i,\pi_i^*E_{>i})$ be bundles over $\bP(E_i)$. Then $\bC^*$ acts linearly in every fiber with negative respectively positive weights on $\tilde{E}^-$ and on $\tilde{E}^+$ respectively.  Let us consider the  vector space counterpart of the problem we want to solve. 

 Suppose $V=V^-\oplus V^+$ is a vector space on which $\bC^*$ acts with weights $-\alpha_1>-\alpha_2>\ldots >-\alpha_{k^-}$ on $V^-$ and with weights $\gamma_1<\ldots<\gamma_{k^+}$ on $V^+$ where $\alpha_i,\;\gamma_j$ are positive integers. There are several related actions one can consider with different closures
 \begin{itemize} 
 \item[(a)] $\bC^*\times V\ra V$ with closure in $\bP^1\times V\times V$;
 \item[(b)] $\bC^*\times V\ra \bP(\bC\oplus V)$ (just embed the codomain from (a) into $\bP(\bC\oplus V)$) with closure in $\bP^1\times V\times \bP(\bC\oplus V)$, 
 \item[(c)] $\bC^*\times \bP(\bC\oplus V)\ra \bP(\bC\oplus V)$ (projectivize the resulting action on $\bC\oplus V$ obtained by putting together (a)  with the trivial action on $\bC$) with closure in $\bP^1\times \bP(\bC\oplus V)\times\bP(\bC\oplus V)$; .
 \end{itemize}
 We are familiar with (c) from the previous section. Notice that the action on $\bP(\bC\oplus V)$ can alternatively be described as an action with the following $k^-+k^++1$ non-negative weights:
 \[0=\alpha_{k^-}-\alpha_{k^-},\;\;\;\;\alpha_{k^-}-\alpha_{k^--1}\;\;\;\ldots\;\;\;\; \alpha_{k^-}-\alpha_{1},\;\;\;\;\; \alpha_{k^-},\;\;\;\;\;\gamma_1+\alpha_{k^-}\;\;\;\; \ldots\;\;\; \;\;\gamma_{k^+}+\alpha_{k^-}
 \]
 with the weight $\alpha_{k^-}$ corresponding to the action on $\bC$. Hence the closure at (c) can be described as a blow-up of a sheaf of ideals in $\bP^1\times \bP(\bC\oplus V)$. Then the closure at (b) is the same blow-up restricted to the preimage of the open set $\bP^1\times V$ while the closure at (a) is simply the intersection of the closure of (b) with $\bP^1\times V\times V$.
 
  Suppose $A\subset V$ is a reduced analytic set. Then the closure  of the graph of the action $\bC^*\times A\ra \bP(\bC\oplus V)$ is an analytic subset  of $\bP^1\times V\times \bP(\bC\oplus V)$. We restrict it to the open set $\bP^1\times V\times V$ and denote this object $B(A)$. The projection of $B(A)$ onto $\bP^1\times V$ is a birational map with exceptional divisor (see Remark \ref{Remexcdiv}) denoted $E(A)$ contained in 
 \[ (\{[0:1]\}\times V\times V)\cup (\{[1:0]\}\times V\times V)
 \]
 One can be even more precise. The exceptional divisor  is in fact contained in
 \[ (\{[0:1]\}\times S(\bP(\bC))\times U(\bP(\bC)))\cup (\{[1:0]\}\times U(\bP(\bC))\times S(\bP(\bC)))=
 \]
 \[ =\{[0:1]\}\times V^-\times V^+\cup \{[1:0]\}\times V^+\times V^-.
 \]
 In other words, when intersecting with $\{\mu=0\}\times V\times V$ only one component of the exceptional divisor in $\bP^1\times \bP(\bC\oplus V)\times \bP(\bC\oplus V)$ out of the $k^-+k^++1$ possible survive the intersection, namely the one corresponding to the fixed point $\bP(\bC)$. For the other $C_i^{\infty}$'s (using notations (\ref{Vieq}), (\ref{Aieq})), either $\bP(V_i^+)\subset \bP(V)$ or $\bP(V_i^-)\subset \bP(V)$.  
 
 Denote 
 \[ E(A)_{\infty}:=E(A)\cap \{[0:1]\}\times V^-\times V^+
 \]
 The  main claim is that $E(A)_{\infty}$ does not depend on the choice of weights $\alpha_1,\ldots,\alpha_k$ and in fact coincides with the exceptional divisor of the affine cone of the weighted strict transform of $A$ under the weighted blow-up of $V^-$ inside $V$. The weighted blow-up of $V^-$ in $V$ is the analytic space:
 \[ \{(v_-,v_+,[w])\in (V^-\oplus V^+)\times \bP^{\gamma}(V^+)~|~v^+\in [w]\}
 \]
 The affine cone over the weighted strict transform of $A$ is the closure of 
 \[\{(v_-,v_+,\lambda^{\gamma}v_+\footnote{Notation: $\lambda^{\gamma}v^+:=(\lambda^{\gamma_1}v^+_1,\lambda^{\gamma_2}v^+_2,\ldots)$})~|~v=(v^-,v^+)\in A,\; \lambda\in \bC^*\} \mbox{ in } V\times V^+\subset V\times V.\] The exceptional divisor is obtained by intersecting with  $\{v_+=0\}$.
 
 We will prove the claim in two steps.
 
 \noindent
 \emph{Step 1.} Consider the following map:
 \[\varphi:\bC^*\times V\ra V,\qquad \varphi:=(\varphi^-,\varphi^+),\;\;\; \varphi^-\equiv0, \;\;\;\varphi^+(\lambda,v^-,v^+):=\lambda^{\gamma}v^+
 \] 
 This is not an action. However for every reduced closed subset $A\subset V$, the closure of the graph of $\varphi\bigr|_{\bC^*\times A}$ in $\bP^1\times V\times V$ is an analytic subset.  This can be seen by considering, equivalently, the closure of the graph of $\varphi^+\bigr|_{\bC^*\times A}$ in $\bP^1\times V\times V^+$. Notice that $\varphi^+$ has an extension to a map $\bC^*\times \bP(\bC\oplus V)\ra \bP(\bC\oplus V^+)$
 and the closure of the graph of this map in $\bP^1\times \bP(\bC\oplus V)\times \bP(\bC\oplus V^+)$ is the blow-up of a bihomogeneous ideal in $\bP^1\times \bP(\bC\oplus V)$, namely
 \[ \mathcal{J}=\langle \mu^{\gamma_{k^+}} \theta, \mu^{\gamma_{k^+}-\gamma_1}\lambda^{\gamma_1}w_1,\ldots\mu^{\gamma_{k^+}-\gamma_{k^+-1}}\lambda^{\gamma_{k^+-1}}w_{k-1},\lambda^{\gamma_{k^+}}w_k\rangle
 \]
 where $([\mu:\lambda],[\theta:v_1:\ldots v_{k^-}:w_1:\ldots w_{k^+}])$ are the coordinates of a point in $\bP^1\times \bP(\bC\oplus V)$. We have an analogous relation between the closures as in (a)-(c) above.
 
 Denote the closure of the graph of  $\varphi\bigr|_{\bC^*\times A}$ in $\bP^1\times V\times V$ by $A^{\varphi}$. It has an exceptional divisor $EA^{\varphi}$ which  has two distinct components: one denoted $EA^{\varphi}_{\infty}$ and lying in $\{[0:1]\}\times V\times V$   and another one in $\{[1:0]\}\times V\times V$. Let us prove that
 \begin{equation} \label{EA} EA^{\varphi}_{\infty}=E(A)_{\infty}
 \end{equation}
 Consider the following maps:
 \[ \bC^*\times  V\times V\ra \bC^*\times V\times V,\qquad (\lambda,v^-,v^+,w^-,w^+)\ra(\lambda,v^-,v^+,\lambda^{-\alpha}v^-,w^+)
 \]
 \[ \bC^*\times  V\times V\ra \bC^*\times V\times V,\qquad (\lambda,v^-,v^+,w^-,w^+)\ra(\lambda,v^-,v^+,0,w^+)
 \]
 Clearly both maps extend holomorphically at $\lambda=\infty$. The two extensions at $\lambda=\infty$ are inverse to each other when taking the restrictions:
 \begin{itemize}
 \item[-] of the first one to the closure of the graph of $\varphi$ (but away from $\{\lambda=0\}\times V\times V$) 
 \item[-] of  the second one to the closure of the graph of the original action on $V$ (also away from $\{\lambda=0\}\times V\times V$) 
 \end{itemize}
 Clearly the first holomorphic map takes the graph of $\varphi\bigr|_{\bC^*\times A}$ to the graph of the action on $A$ and hence it will take the closure at $\infty$ to the closure at $\infty$. Notice also that at $\infty$ the two maps become identity when restricted to $V\times V^+$. Therefore (\ref{EA}) holds. \vspace{0.2cm}
 
\noindent
\emph{Step 2:} We claim that $A^{\varphi}$ is isomorphic  via the projection $\bP^1\times V\times V\ra V\times V$ with the affine cone over the weighted strict transform of $A$. This will of course imply that $EA^{\varphi}_{\infty}$ is the exceptional divisor of the affine cone and finish the proof of the main claim. However, this statement is similar to Lemma \ref{Lemga}. \vspace{0.2cm}

 Just like Proposition \ref{Pr} has a vector space counterpart from which it can be adapted, the same is true here. In our case if we translate the main claim to the bundle $\tilde{E}=\tilde{E}^-\oplus \tilde{E}^+$ where the reduced analytic set $A$ is the preimage of  $s(M)$ via the chart $\tilde{E}\ra \bP(E)$, we get the claim of the Theorem.  A similar observation as at the end of the proof of Proposition \ref{Pr} applies here as well, i.e. one works  with $s\bigr|_{s^{-1}(\bP(E_i)\setminus \bP(E_i'))}$ and notices that 
$ \tilde{S}_i^{\infty}\cap (\bP(E)\setminus \bP(E_i')\times_{M}\bP(E)\setminus \bP(E_i'))$ is entirely determined by this restriction.
   \end{proof}

\section{Applications in the weighted  case}
\subsection{Bott-Chern formula}\label{BCtheorem}

As a first application we prove Bott-Chern \cite{BC} Proposition 1.5. For a holomorphic vector bundle $E\ra M$, let $\hat{c}_t(E)$ represent the total Chern class of $E$ in the Bott-Chern cohomology groups. 

\begin{prop} \begin{itemize}
\item[(i)] If $h_0$ and $h_1$ are two Hermitian metrics over $E$ then 
\[ c_t(\nabla^{h_0})-c_t(\nabla^{h_1})=\partial\bar{\partial}\omega
\]
for some smooth form $\omega$ where $\nabla^{h_0}$ and $\nabla^{h_1}$ are the Chern connections.
\item[(ii)] Let $0\ra E'\ra E\ra E''\ra 0$ be an exact sequence of holomorphic vector bundles over $M$. Then in $H^{*,*}_{BC}(M)$ the following holds
\[ \hat{c}_t(E)=\hat{c}_t(E'\oplus E'')=\hat{c}_t(E)\hat{c}_t(E'')
\]
\end{itemize}
\end{prop}
\begin{proof} (i) Suppose the rank of $E$ is $k$ and consider $\Gr_k(E\oplus E)$, the Grassmannian fiber bundle of $k$-linear subspaces in $E\oplus E$.  Let $\tau\subset \pi^*E\oplus \pi^* E\ra \Gr_k(E\oplus E)$ be the tautological bundle. On $\pi^*E\oplus \pi^*E$ put the Hermitian metric $\pi^*h_0\oplus \pi^*h_1$. This induces a metric $h^{\tau}$  on $\tau$. Consider the Chern connection $\nabla^{\tau}$ on $\tau$ induced by the metric $h^{\tau}$. 

Notice that $\Gr_k(E\oplus E)\ra M$ comes with two other distinguished sections, namely the one represented by the $k$-dimensional subspaces $E\oplus 0$ denoted $\iota_0$ and $0\oplus E$  denoted $\iota_{\infty}$, respectively. One sees easily that  $(\tau,\nabla^{\tau})\bigr|_{\iota_0(M)}$ is isomorphic with $(E,\nabla^{h_0})$ while $(\tau,\nabla^{\tau})\bigr|_{\iota_{\infty}(M)}\simeq (E,\nabla^{h_1})$.

Consider the action of $\bC^*$ on $\Hom(E,E)$, $A\ra \lambda A$. This extends to an action on $\Gr_k(E\oplus E)$ which in fact can be seen as the restriction via the Plucker embedding of a weighted action on $\bP(\Lambda^k(E\oplus E))$. The flow of the diagonal (i.e. graph of $\id_E$ considered as a holomorphic section of $\pi:\Gr_k(E\oplus E)\ra M$) equals $\iota_0(M)$ when $\lambda\ra 0$, and equals $\iota_{\infty}(M)$ when $\lambda\ra \infty$. By Proposition \ref{grpr}  it follows that
\[ c_t(\nabla^{h_1})-c_t(\nabla^{h_0})=\partial\bar{\partial}T
\]
and this proves the metric invariance of the Bott-Chern classes. 

For (ii) we use the same idea, only slightly different. Suppose $E'$ and $E$ come endowed with Hermitian metrics $h^{E'}$ and $h^{E}$. Let $k$ be the rank of $E'$ and consider $\Gr_{k}(E'\oplus E)$ which comes with a holomorphic section given by the graph of the injective map $\iota:E'\hookrightarrow E$ of the exact sequence
We consider again the complex flow induced by rescaling on $\Hom (E',E)$.  Now let $Q\ra \Gr_k(E'\oplus E)$  be the universal quotient bundle, with the induced metric, namely the one obtained by identifying $Q$ with the orthogonal complement of $\tau$ in $\pi^* (E'\oplus E)$. We will use the total Chern form $c_t(\nabla^Q)$. 

  There are again two natural sections of $\Gr_k(E'\oplus E)$, namely $E'\oplus 0$  and $0\oplus \Imag \iota\in\Gr_k(E)$ denoted again $\iota_0$ and $\iota_{\infty}$. It is not hard to check that as currents
\[ \lim_{\lambda\ra 0}\Gamma_{\iota}=\iota_0(M),\qquad \lim_{\lambda\ra \infty}\Gamma_{\iota}=\iota_{\infty}(M)
\]
Notice that $(Q,\nabla^Q)\bigr|_{\iota_0(M)}\simeq (E,\nabla^{h^E})$ while $(Q,\nabla^Q)\bigr|_{\iota_{\infty}(M)}\simeq (E',\nabla^{h^{E'}})\oplus (Q_0,\nabla^{Q_0})\bigr|_{\iota_{\infty}{(M)}}$ where $Q_0\subset \pi^*E$ is the universal quotient bundle over $\Gr_{k}(E)$ with the induced Hermitian metric. Then $\iota_{\infty}^*Q_0\simeq E''$ because $\iota_{\infty}^*\tau=\Imag\iota$ seen as a subbundle of $E$. If we endow $E''$ with the quotient metric then the isomorphism is of holomorphic bundles with Hermitian metrics. 

Applying the theory to the form $\omega=c_t(\nabla^Q)\in \Omega^*\Gr_k(E'\oplus E)$ we get the double transgression relation that justifies (ii).

\end{proof}

\subsection{Quillen's question} Let $E^+\ra M\leftarrow E^-$ be two holomorphic and Hermitian vector bundles over $M$. Let $\nabla:=\nabla^+\oplus \nabla^-$ be the Chern connection on $E:=E^+\oplus E^-$. Let $\tilde{A}:E^+\ra E^-$ be a holomorphic morphism  and let $A_z:=\left(\begin{array}{cc} 0& \overline{z}\tilde{A}^*\\
 z\tilde{A}& 0\end{array}\right)$ act on $E$ where $z\in \bC$.  Consider the following family of superconnections (see \cite{Q0}):
\begin{equation}\label{spcon} \bA_z:=\nabla+ A_z,
\end{equation}
i.e. odd, first order differential operators on $\Gamma(\Lambda^*T^*M\otimes E)$ that verify the Leibniz rule:
\[ \bA(\omega\otimes s)=d\omega \otimes s+(-1)^{\deg\omega}\omega\wedge \bA s
\] Note that $A_z(\omega\otimes s):=(-1)^{\deg \omega}\omega\otimes A_zs$.

Motivated by the computation of the Chern character of the relative $K$-theory class determined by $A_1$, Quillen asked for a description of 
\[\lim_{z\ra \infty}\ch(\bA_z)
\]
where $\ch(\bA_z)=\str(e^{F(\bA_z)})\in \Gamma(\Lambda^{\even}T^*M)$ is the Chern supercharacter form determined by $\bA_z$.

When $M$ is a smooth manifold, $t$ is real and $A$ satisfies certain transversality conditions, we were able to give a quite explicit answer in \cite{CiHl3}. We explain why these results can be extended to the context of Bott-Chern currential homology. 

Let us return to the holomorphic context. We need the following.

\begin{prop}\label{delclos} The forms $\ch(\bA_z)$ belong to $\bigoplus_{p} \Omega^{p,p}(M)$, are $d$-closed and therefore $\partial$ and $\bar{\partial}$-closed.
\end{prop}

\begin{proof} The fact that $\ch(\bA_z)$ is $d$-closed is standard. We need only prove that its components belong to $\Omega^{p,p}(M)$. For fixed integer $i$ let $\Omega^{*}_i(M):=\bigoplus_{p}\Omega^{p,p+i}(M)$.

 First we look at the curvature $F(\bA_z)$:
\[F(\bA_z):=\left(\begin{array}{cc} F(\nabla^+)+|z|^2\tilde{A}^*\tilde{A} & [\nabla, (z\tilde{A})^*] \\
         \;[\nabla,z\tilde{A}] &   F(\nabla^-)+|z|^2\tilde{A}\tilde{A}^*
\end{array}\right)
\]
By definition 
\[[\nabla,z\tilde{A}](s):=\nabla^-(\tilde{A}s)-\tilde{A}(\nabla^+s).\]
and one checks easily that the connection on $\Hom(E^+,E^-)$ is compatible with the holomorphic structure and the Hermitian metric. Now if $s$ is a  holomorphic section and $\nabla$ is the Chern connection it follows that $\nabla s\in \Omega^{1,0}(M;E)$. Similarly $[\nabla,z\tilde{A}]\in \Omega^{1,0}(M,\Hom(E^+,E^-))$. This implies that $[\nabla,\bar{z}\tilde{A}^*]\in \Omega^{0,1}(M,\Hom(E^+,E^-))$. Indeed, the metric compatibility, generally given for real vector fields $X\in \Gamma(TM)$ has the following expression for general complex vector fields $X\in \Gamma(TM\otimes_{\bR}\bC)$:
\[ X\langle s_1,s_2\rangle=\langle \nabla_Xs_1,s_2\rangle +\langle s_1, \nabla_{\overline{X}}s_2\rangle
\]
Then one gets immediately that
\[\nabla_{X} (\tilde{A}^*)=(\nabla_{\overline{X}}\tilde{A})^*
\]
and from here the claim. It follows then by a quick induction that the block components of $F(\bA)^k$ are of type:
\[ \left(\begin{array}{cc} A & B\\
  C&D \end{array}\right)
\] where $A,D\in \Omega^{*}_0(M)$, $B\in \Omega^{*}_1(M)$, $C\in\Omega^{*}_{-1}(M)$. Clearly $\str(e^{F(\bA_z)})\in\Omega^{*}_0(M)$ and this finishes the proof.
\end{proof}

The set-up for applying the theory is as follows. The holomorphic fiber bundle with projective fiber over $M$ will be $\Gr_{k^+}(E^+\oplus E^-)$ where $k^+$ is the rank of $E^+$ and the holomorphic sections is $\Gamma_A$, which associates to $m\in M$ the linear graph $\Gamma_{A_m}$. We need a universal Chern supercharacter form on $\Gr_{k^+}(E^+\oplus E^-)$ which is $\partial$ and $\bar{\partial}$ closed. Over the dense, open set $\Hom (E^+,E^-)$ we have such a form. Indeed let $\pi:\Gr_{k^+}(E^+\oplus E^-)\ra M$ and let $\pi^*E^+$ and $\pi^*E^-$ be the corresponding pull-backs. Then $\pi^*E^+$ and $\pi^*E^-$ come with Chern connections and over $\Hom(E^+,E^-)$ there exists a tautological section $\tilde{s}^{\tau}$ of $\Hom(\pi^*E^+,\pi^*E^-)$, and hence over $\Hom(E^+,E^-)$ we have a superconnection:
\[ \bA^{\tau}=\pi^*\nabla^+\oplus\pi^* \nabla^-+s^{\tau}
\] 
where $s^{\tau}$ and $\tilde{s}^{\tau}$ share the same relation as $A_1$ and $\tilde{A}$.
Quillen showed in \cite{Q1} how to extend $\ch(\bA^{\tau})=\str(e^{-F(\bA^{\tau})})$ to be defined everywhere on $\Gr_{k^+}(E^+\oplus E^-)$ by using the Cayley transform. We will keep the notation $\ch(\bA^{\tau})$ for the extension. It follows from Proposition \ref{delclos} that the extension is $\partial$ and ${\bar{\partial}}$
closed.

If $\tilde{A}:M\ra \Hom(E^+,E^-)$ is a section (not necessarily holomorphic) then
\[(z\tilde{A})^*\ch(\bA^{\tau})=\ch(\bA_z)
\]

Now we can use Proposition \ref{grpr} to deduce:
\begin{theorem} Let $E:=E^+\oplus E^-\ra M$ be a holomorphic super vector bundle with a compatible Hermitian metric and corresponding Chern connection $\nabla$. Let $\tilde{A}:M\ra\Hom(E^+,E^-)$ be a holomorphic section and $\bA_z:=\nabla+A_z$ be the corresponding family of superconnections on $E$ as in (\ref{spcon}). Then there exists a double transgression formula
\[ \ch(\nabla)-\lim_{z\ra\infty}\ch(\bA_z)=\partial{\bar{\partial}} T(\tilde{A})
\]
where $\lim_{z\ra\infty}\ch(\bA_z)$ is a sum of flat currents, each  supported on $\{p\in M~|~\dim{\Ker\tilde{A}}= i\}$, where $1\leq i\leq k^+$.
\end{theorem}
\begin{proof} The reason for the vanishing of the component of the limit supported on \linebreak $\{p\in M~|~\dim{\Ker\tilde{A}}= 0\}$ is a certain vanishing property of the extended form $\ch^{\bA^{\tau}}$ (see (3.12) in \cite{Q1} and also  Section 4 in \cite{Q0}). 
\end{proof}
\begin{rem} Again, under transversality conditions, explicit computations of $\lim_{z\ra \infty}\ch(\bA_z)$ can be performed.
\end{rem}

\section{Addition of linear correspondences} 
Let $E$ and $F$ be two vector spaces of dimensions $p$ and $q$. In this section we will deal with the question of whether the addition operation on  $\Hom(E,F)$ can be compactified to a holomorphic operation
\[\Gr_p(E\oplus F)\times \Gr_p(E\oplus F)\ra \Gr_p(E\oplus F).
\] 

One sees immediately for $p=q=1$ that this is too optimistic since clearly on $\bP^1\times \bP^1$ the map:
\[ [\lambda:v], [\mu:w]\ra [\lambda \mu:\mu v+\lambda w ]
\]
which extends $([1:v],[1:w])\ra [1:v+w]$ does not make sense when $\lambda=\mu=0$. On the positive side, the operation is holomorphic on $\bP^1\times \bP^1\setminus\{([0:1],[0:1])\}$ which is larger than  the original domain of the operation. We would like to find the Grassmannian analogue of this extension. 

For that we will consider the following open conditions for pairs $(L_1,L_2)\in \Gr_p(E\oplus F)\times \Gr_p(E\oplus F)$:
\begin{align}\label{op1} \pi^E(L_1)+\pi^E(L_2)=&E\\
                    \label{op2}       L_1\cap L_2\cap F=\;&\{0\}\quad
\end{align}
where $\pi^E$ is the projection relative $E\oplus F$. They can be rephrased as
\begin{itemize}
\item[(i)] $\pi^E\bigr|_{L_1+L_2}$ is surjective;
\item[(ii)] $\pi^E\bigr|_{L_1\cap L_2}$ is injective.
\end{itemize}
In the case when $E$ and $F$ are endowed with Hermitian metrics then (\ref{op1}) becomes
 \begin{equation}\label{op11}
   L_1^{\perp}\cap L_2^{\perp}\cap E=\{0\}
   \end{equation} by using the fact that $\pi^E(L_1)^{\perp}=L_1^{\perp}\cap E$ where the orthogonal complement $\pi^E(L_1)^{\perp}$ is taken in $E$. This makes it  clearer that this condition is open. For ease of presentation we will assume that $E$ and $F$ have fixed metrics although the extension of the addition map will not depend on this. 

Denote by $\Gr^+_{p,p}(E\oplus F)$ the space of pairs $(L_1,L_2)\in \Gr_p(E\oplus F)\times \Gr_p(E\oplus F)$ that satisfy (\ref{op1}) and (\ref{op2}). 
\begin{example} If  $L_1$ (or $L_2$)  is in $\Hom(E,F)$ then (\ref{op11}) and (\ref{op2}) are satisfied automatically since $L_1^{\perp}\cap E=\{0\}$ and $L_1\cap F=\{0\}$. 

Hence for $E=\bC=F$ we are getting indeed $\bP^1\times \bP^1\setminus\{([0:1],[0:1])\}$.
\end{example}
\begin{example} There exists an interesting algebraic, closed set contained in $\Gr^+_{p,p}(E\oplus F)$. This is the set of of pairs $(L_1,L_2)$ satisfying:
\begin{align} \label{op10} L_1\cap E\supset \pi^E(L_2^{\perp})\\
                   \label{op20}  L_2^{\perp}\cap F\supset \pi^F(L_1) 
\end{align}
To see that (\ref{op10}) implies (\ref{op1}) notice that $L_2^{\perp}\cap E \subset \pi^E(L_2^{\perp})$ and hence $L_2^{\perp}\cap E \subset L_1\cap E$ but $L_1\cap E\cap L_1^{\perp}=\{0\}$. Similarly, (\ref{op20}) implies (\ref{op2}).

More concretely, if $L_1=\Gamma_{A}$ and $L_2=\Gamma_{B}$ then $L_2^{\perp}=\Gamma_{-B^*}$ and $\pi^E(L_2^{\perp})=\Imag B^*$ while $L_1\cap E=\Ker{A}$, hence (\ref{op10}) translates to
\[ AB^*=0
\]
while (\ref{op20}) says that
\[ B^*A=0.
\]
\end{example}

We will define $L_1*L_2$ by composition of three holomorphic maps:
\begin{itemize}
\item[(a)] The first  map is  $\oplus:\Gr^+_{p,p}(E\oplus F)\ra \Gr_{2p}(E\oplus F\oplus E\oplus F)$: 
\[ (L_1,L_2)\ra L_1\oplus L_2;
\]
\item[(b)] Let $G:=E\oplus F\oplus E\oplus F$, $E_1:=E\oplus E\subset G$, $F_1:=F\oplus F\subset G$, $\Delta^{E_1}$ the diagonal in $E\oplus E$ and $\Delta^{F_1}$ the diagonal in $F_1$. 

Denote by $\Gr_{2p}^{E_1,F_1}(G)\subset\Gr_{2p}(G)$ the open subset of subspaces $V$ that satisfy
\begin{align} \label{op3} V^{\perp}\cap \epsilon^{E_1} \Delta^{E_1}=\{0\}\\
                      \label{op4} V\cap \epsilon^{F_1}\Delta^{F_1} =\{0\}               
\end{align} 
where  $\epsilon^{E_1} \Delta^{E_1}$ and $\epsilon^{F_1} \Delta^{F_1}$ are the antidiagonals in $E_1$ and $F_1$. Then it is easy to check that $\oplus$ takes pairs $(L_1,L_2)$ that satisfy (\ref{op11}) and (\ref{op2}) to $L_1\oplus L_2\in \Gr_{2p}^{E_1,F_1}(G)$. 

Define the second map
\[ \cdot\cap (\Delta^{E_1}\oplus F_1):\Gr_{2p}^{E_1,F_1}(G)\ra \Gr_{p}(\Delta^{E_1}\oplus F_1),\qquad V\ra V\cap (\Delta^{E_1}\oplus F_1)
\]
To see that the map is well-defined notice that $[V\cap (\Delta^{E_1}\oplus F)]^{\perp}=V^{\perp}+\epsilon^{E_1}\Delta^{E_1}$ and the later has dimension $2q+p$ due to (\ref{op3}). Hence $V\cap (\Delta^{E_1}\oplus F)$ has dimension $p$. The fact that this map is holomorphic under these conditions is standard.
\item[(c)] Consider the map 
\[ \theta:\Delta^{E_1}\times F_1\ra E\oplus F,\qquad (v,w_1,v,w_2)\ra (v,w_1+w_2),\qquad v\in E, \;w_1,w_2\in F
\]
The kernel of this map is $\epsilon\Delta^{F_1}$. Hence we get an induced well-defined map:
\[ \Theta:\Gr^{\Ker\theta}_p(\Delta^{E_1}\times F_1)\ra \Gr_{p}(E\oplus F),\qquad \Theta(W):=\theta(W)
\]
where $\Gr^{\Ker\theta}_p(\Delta^{E_1}\times F_1)$ is the open set of subspaces $W$ satisfying $W\cap \Ker\theta=\{0\}$. The fact that the image of $\cdot \cap (\Delta^{E_1}\oplus F_1)$ lands in $\Gr^{\Ker\theta}_p(\Delta^{E_1}\times F_1)$ is a consequence of (\ref{op4}).
\end{itemize}

\begin{prop} If $L_1=\Gamma_{A}$ and $L_2=\Gamma_B$ with $A,B\in\Hom(E,F)$ then
\[ L_1*L_2=\Gamma_{A+B}.
\]
\end{prop}
\begin{proof} Notice that
\[ (L_1\oplus L_2)\cap  (\Delta^{E_1}\oplus F_1)=\{(v,Av,v,Bv)~|~v\in E\}
\]
Therefore 
\[ \Theta\left((L_1\oplus L_2)\cap  (\Delta^{E_1}\oplus F_1)\right)=\Gamma_{A+B}.
\]
\end{proof}
\begin{rem}
We deduce that the map $\cdot*\cdot:\Gr^+_{p,p}(E\oplus F)\ra \Gr_p(E\oplus F)$ is a well-defined holomorphic map that extends the addition operation. From an algebraic point of view this operation does not seem that interesting even if it is commutative and there exists a "neutral element", namely $L=E$. For example $L*L$ is only defined if $L$ is a graph. 
\end{rem}

We   look now at extending the operation
 \[\Hom(E,F)\times\Hom(F,E)\ra \Hom(E,F),\quad (A,B)\ra A+B^*.\]
Since  $B\ra B^*$ is an operation that extends to a real analytic bijection $\Gr_q(E\oplus F)\ra \Gr_{p}(E\oplus F)$ we just need to trace what happens with the extension already discussed after composing with this bijection. We state the facts directly. On the open subset of pairs $(L_1,L_2)\in \Gr_p(E\oplus F)\times\Gr_q(E\oplus F)$ that satisfy the conditions:
\begin{align}\label{op110} L_1^{\perp}\cap L_2\cap E=\{0\}\\
  \label{op220} L_1\cap L_2^{\perp}\cap F=\{0\}
\end{align}
we define an operation $L_1\diamond L_2$ by composition of the maps:
\[ (L_1,L_2)\ra L_1\oplus L_2^{\perp}\in \Gr_{2p}^{\diamond}(G)
\]
\[\Gr_{2p}^{\diamond}(G)\ni V\ra V\cap (\Delta^{E_1}\oplus F_1)\in\Gr_p(\Delta^{E_1}\oplus F_1)
\]
and the map $\Gr_p^{\ker\theta^-}(\Delta^{E_1}\oplus F_1)\ra \Gr_p(E\oplus F)$ induced by
\[{\theta}^-: \Delta^{E_1}\oplus F_1\ra E\oplus F,\qquad (v,w_1,v,w_2)\ra (v,w_1-w_2).
\]
The open $\Gr^{\diamond}_{2p}(G)\subset \Gr_{2p}(G)$ consists of $V$ such that $V^{\perp}\cap \epsilon\Delta^{E_1}=\{0\}$ and $V\cap \Delta^{F_1}=\{0\}$. The operation $\diamond$ extends real analytically the  operation $A+B^*$.
\begin{example} The closure in $\Gr_{p}(E\oplus F)\times\Gr_q(E\oplus F)$ of the algebraic set in $\Hom(E,F)\times \Hom(F,E)$ defined by the "chain" equations:
\[ A\circ B=0\qquad B\circ A=0
\]
is the set of pairs $(L_1,L_2)$ satisfying 
\begin{align}\label{op111} L_1\cap E\supset\pi^F(L_2) \\
 \label{op222}L_2\cap F\supset \pi^E(L_1)
\end{align} 
Notice that (\ref{op111}) implies (\ref{op110}) and (\ref{op222}) implies (\ref{op220}).
\end{example}
\vspace{0.2cm}
\subsection{The Bismut-Gillet-Soul\'e double transgression}
We now come to the main application of this digression.

Let  $E$ and $F$ be holomorphic vector bundles over $M$ of ranks $p$ and $q$ endowed with Hermitian metrics and Chern connections $\nabla^E$ and $\nabla^F$.  Let  $v:=\left(\begin{array}{cc} 0& B\\
  A& 0\end{array}\right)\in \End^-(E\oplus F)$ be an odd vector bundle morphism such that $B\circ A=0$ and $A\circ B=0$. Look at the family of superconnections parametrized by $\lambda\in \bC$:
  \begin{equation}\label{balambda} \bA_{\lambda}:=\nabla^E\oplus \nabla^F +(\lambda v+\bar{\lambda}v^*)
  \end{equation}
  and take their Chern supercharacters $\str(e^{-F(\bA_{\lambda})})$. 
  
  We see quickly that $\str(e^{-F(\bA_{\lambda})})$ is the pull-back via the section $\lambda v$ of a universal form that lives on $\End^-(E\oplus F)$ described as follows. Let
   \[\pi:\End^-(E\oplus F)\ra M\]  be the projection and take $\pi^*(\End^-(E\oplus F))=\End^-(E\oplus F)\times_M\End^-(E\oplus F)$. This is the vector bundle  of odd morphisms on $\pi^*E\oplus \pi^*F$ that lives over $\End^-(E\oplus F)$. It has a tautological section 
   \[\hat{s}^{\tau}=(s_E^{\tau},s_F^{\tau}):\End^-(E\oplus F)\ra \Hom(\pi^*E,\pi^*F)\times_M\Hom(\pi^*F,\pi^*E)\] 
  induced by the diagonal embedding $\End^-(E\oplus F)\hookrightarrow\End^-(E\oplus F)\times_M\End^-(E\oplus F)$. 
   
   Moreover  $\pi^*E\oplus \pi^*F$ has a natural connection and one can form the  superconnection:
  \begin{equation}\label{Ahat} \hat{\bA}^{\tau}:= \pi^*\nabla^E\oplus \pi^*\nabla^F +(\hat{s}^{\tau}+(\hat{s}^{\tau})^*)=\left(\begin{array}{cc}\pi^*\nabla^E & s_F^{\tau}+(s_E^{\tau})^* \\ s_E^{\tau}+(s_F^{\tau})^*& \pi^*\nabla^F\end{array}\right)
  \end{equation}
The form $\str(e^{-F(\hat{\bA}^{\tau})})$ lives on $\End^-(E\oplus F)$. We have  the obvious
\[ (\lambda v)^*\str(e^{-F(\hat{\bA}^{\tau})})=\str(e^{-F(\bA_{\lambda})}).\]
 Since we need to take $\lambda\ra \infty$ we have to pass to a compactification of $\End^-(E\oplus F)=\Hom(E,F)\oplus \Hom(F,E)$. Clearly, the natural compactification is $\Gr_p(E\oplus F)\times_M \Gr_q(E\oplus F)$. 
 
 Denote by $Z\subset \Gr_p(E\oplus F)\times_M \Gr_q(E\oplus F)$ the closure of the set of "chain equations" $A\circ B=0$ and $B\circ A=0$. As we mentioned, every fiber $Z_m$ with  $m\in M$ can be described by the following incidence relations:
 \[ Z_m=\{(L_1,L_2)\in \Gr_p(E_m\oplus F_m)\times \Gr_q(E_m\oplus F_m) ~|~L_1\cap E_m\supset \pi^{E_m}(L_2),\;\;L_2\cap F_m\supset \pi^{F_m}(L_1)\}
 \]
 Incidentally, this is actually an old acquaintance. We use the notation of Proposition \ref{Fixpeq}.
 \begin{prop} The set $Z$ coincides with $\bigcup_{i=0}^k\overline{S(F_i)\times_{F_i}U(F_i)}$.
 \end{prop}
 \begin{proof} The points in a fiber of $S(F_i)\times_{F_i}U(F_i)$ are described by the incidence relations:
 \[ L_1\cap E_m=\pi^{E_m}(L_2),\qquad L_2\cap {F_m}=\pi^{F_m}(L_1).
 \]
 \end{proof}
 
 The next result is what allows the extension of the Bismut-Gillet-Soul\'e double transgression theory using the tools of this article.
 \begin{prop}\label{BGSQext} The form $\str(e^{-F(\hat{\bA}^{\tau})})$ extends to a smooth form in an open neighborhood of the set $Z\subset \Gr_p(E\oplus F)\times_M\Gr_q(E\oplus F)$.
 \end{prop}
 \begin{proof} We use  the map $(A,B)\ra A+B^*$. We know that this map extends to a real analytic map from the open set $U$ of pairs $(L_1,L_2)\in\Gr_p(E\oplus F)\times_M\Gr_q(E\oplus F)$ satisfying (\ref{op110}) and (\ref{op220}) to $\Gr_p(E\oplus F)$. Moreover $U\supset Z$. 
 
 We mentioned in the previous section that Quillen showed how the Chern supercharacter of the superconnection associated to the superbundle $\pi^*E\oplus\pi^*F\ra\Hom(E,F)$
 \begin{equation}\label{Ataudef} \bA^{\tau}=\left(\begin{array}{cc} \pi^*\nabla^E& 0\\
 0& \pi^*\nabla^F\end{array}\right)+\left(\begin{array}{cc} 0& (s^{\tau})^*\\
  s^{\tau}& 0\end{array}\right)
 \end{equation}
extends to a smooth form $\ch(\bA^{\tau})$ on $\Gr_p(E\oplus F)$. By comparing (\ref{Ahat}) and (\ref{Ataudef}) we conclude that
\[ \gamma^*\bA^{\tau}=\hat{\bA}^{\tau}
\]
 where $\gamma:\Hom(E,F)\times \Hom(F,E)\ra \Hom(E,F)$ is $\gamma(A,B)=A+B^*$. Since $\gamma$ extends smoothly to $U$ and $\ch(\bA^{\tau})$ extends as well the proof is complete.
  \end{proof}

 The formalism of Bismut and Bismut-Gillet-Soul\'e from \cite{Bi,BiGiS,BiGiS2} uses a finite sequence of holomorphic and Hermitian vector bundles 
\[ E_i\ra M,\qquad i=0,\ldots, m
\]
with morphisms $v_i:E_i\ra E_{i-1}$, that satisfy the chain condition $v_{i-1}\circ v_i=0$ such that the induced complex of sheaves of holomorphic sections is acyclic away from a complex submanifold $M'$, possibly disconnected. More precisely there exists a holomorphic and Hermitian vector bundle $\eta\ra M'$ and a restriction map  $r:E_0\bigr|_{M'}\ra \eta$ such that
\[ 0\ra\mathcal{O}_M(E_m)\ra \ldots \ra \mathcal{O}_M(E_0)\ra\iota_*\mathcal{O}_{M'}(\eta)\ra0
\]
is exact where $\iota:M'\ra M$ is the inclusion.

Let $E^+:=\oplus_{k} E_{2k}$, $E^-:=\oplus E_{2k+1}$, $E=E^+\oplus E^-$, $v=v^+\oplus v^-\in \End^-(E)$ be the natural, holomorphic, odd morphism induced by the entire collection of $v_i$'s and $v^*$ be its adjoint. Notice that
\[ v^+\circ v^-=0,\qquad v^-\circ v^+=0
\] 

Together with the (direct sum of) Chern connections on $E^+$ and $E^-$ one gets a family of super-connections on $E$:
\begin{equation}\label{bAageq1}\bA_{u}:=\left(\begin{array}{cc}\nabla^{E^+}&0\\ 0& \nabla^{E^-}\end{array}\right)+\sqrt{u}\left(\begin{array}{cc}0& v^-+(v^+)^*\\
 v^++(v^-)^*&0\end{array}\right)=\nabla^E+\sqrt{u}(v+v^*)
\end{equation}
Under a rather technical condition on the Hermitian metrics, called condition (A), Bismut-Gillet-Soul\'e computed the limit $\lim_{u\ra \infty} \ch(\bA_{u})$ explicitly and derived a double transgression formula in Theorem 2.5 of \cite{BiGiS}.  We come now to the main result of this section generalizing the existence part of Bismut-Gillet-Soul\'e Theorem.

\begin{theorem}\label{mainc10} Let $v_i:E_i\ra E_{i-1}$ be a finite sequence of holomorphic chain morphisms between holomorphic and Hermitian vector bundles and for $\lambda\in\bC$ let $\bA_{\lambda}$ be the superconnection obtained from $\bA_u$  in  (\ref{bAageq1}) by replacing the morphism part with $\lambda v+\bar{\lambda}v^*$. Then there exists a double transgression formula:
\[ \lim_{\lambda\ra \infty} \ch(\bA_{\lambda})-\ch(\bA_1)=\partial\bar{\partial} T(v)
\]
\end{theorem}
\begin{proof} Let rank $E^+$ be $p$ and rank of $E^-$ be $q$. The compactification of 
\[\End^-(E)=\Hom(E^+,E^-)\oplus\Hom(E^-,E^+)\] we will consider is 
\[\Gr_{p,q}(E):=\Gr_p(E^+\oplus E^-)\times_M \Gr_q(E^-\oplus E^+)\] which is easily seen to be the closure in $\Gr_{p+q}(E\oplus E)$ of the graph embedding $\End^{-}(E)\hookrightarrow \Gr_{p+q}(E\oplus E)$. The rescaling operation on $\End^-(E)$ extends to an action of $\bC^*$ on the compactification which is simply the diagonal action resulting from the  by now standard, complex flows on $\Gr_p(E^+\oplus E^-)$ and $\Gr_q(E^-\oplus E^+)$ respectively.

Take the holomorphic section $v:M\ra \End^-(E)\hookrightarrow\Gr_{p,q}(E)\hookrightarrow \Gr_{p+q}(E\oplus E)$. By using the Pl\"ucker embedding and proceeding as in Corollary \ref{cormt3} (see also Remark \ref{Gr19}) we deduce that if $\omega$ is a $\partial$ and $\bar{\partial}$ closed form on $\Gr_{p,q}(E)$ then we get a double transgression:
\[ \lim_{\lambda\ra \infty} (\lambda v)^*\omega-\omega=\partial\bar{\partial}T(\omega).
\]
We will use for $\omega$ the extension of  Chern supercharacter $\str(e^{-F(\hat{\bA}^{\tau})})$ of Proposition \ref{BGSQext}. Now, this form is not defined everywhere on $\Gr_{p,q}(E)$, but is defined on an open neighborhood of the closure $Z$ in $\Gr_{p,q}(E)$ of the "chain morphism equations" $A\circ B=0$ and $B\circ A=0$. Since the limiting current $\lambda v$ has support in $Z$ this is enough for the theory to work.
\end{proof}

\begin{rem} Notice that several strong hypothesis are removed from  Theorem 2.5 of \cite{BiGiS} in the statement of Theorem \ref{mainc10}. In particular we do not assume  the fact that the complex is acyclic away from a submanifold, which among other things implies that rank $E^+$ equals rank of $E^-$. Moreover the metric compatibility condition (A) is  removed and the limit parameter is complex.
\end{rem}

 \appendix
  \section{Cones and blow-ups}
  \label{Ap1}
  
   The material here is standard, in the analytical category one can find most of it for example in \cite{Fi} but we included it in order to help with the lecture of this article.
 
 \subsection{Blow-ups} 
  We start by reviewing the process of blow-up. Given an  analytic, closed subvariety $Z$ of an ambient variety $Y$ such that $Z$ is analytically rare\footnote{For simplicity, take  $Y$ to be reduced or even smooth while $Z$ to be any closed subspace which does not contain any irreducible component of $Y$, but can itself be non-reduced.} (\cite{Fi}, Section 0.43) defined by a coherent sheaf of ideals $\mathcal{I}_Z\subset\mathcal{O}_Y$,  the sheaf of Rees algebras
  \[ R(\mathcal{I}_Z):=\bigoplus_{k\geq 0} \mathcal{I}_Z^k
  \]
  defines a complex space,  which is an (affine) cone over $Y$. Locally, over an open $U\subset Y$, a cone is isomorphic with a subspace of $U\times \bC^l$ invariant under multiplication by scalars (see \cite{Fi} page 44). Taking $\bP(
 \cdot)$ of this cone (procedure described in \cite{Fi} Section 1.3) one gets the blow-up space $\Bl_Z(Y)$.  To understand better, notice that locally, if $U\subset Y$ is isomorphic to a local model of an complex subspace in $\bC^n$ with (global) ring $R$ and $I:=\mathcal{I}_Z(U)$ is generated by $P_1,\ldots,P_l$ in $R$, then the Rees algebra is the image of the morphism of  \emph{graded rings} 
  \begin{equation}\label{Req1} R[y_1,\ldots y_l]\ra R[T],\qquad y_i\ra P_i T.
  \end{equation}
  The kernel of (\ref{Req1}) is a homogeneous ideal and by a theorem of Cartan (see \cite{Fi} Section 1.2) all cones over $U$ are induced by such homogeneous ideals. 
  
\begin{rem}\label{beforetric}  It turns out that when $Z$ is \emph{globally} defined by a section $s$ \emph{non identically zero on each irreducible component of $Y$} of a locally free sheaf (vector bundle) $\pi:E\ra Y$, the construction of $\Bl_Z(Y)$ can be explained in more geometric terms. Namely, let $Z:=s^{-1}(0)\subset Y$ and
  \[  \bP({s}):Y\setminus Z\ra \bP(E),\qquad p\ra [s(p)]
  \]
  be the induced (partial) section of $\bP(E)$. Then, $\Bl_Z(Y)$ is \emph{the closure of the image} of $\bP({s})$ in $\bP(E)$ and this will be our definition. 
  \end{rem}

  \begin{rem}\label{tricky} One should be careful  to consider the closure of the graph of \emph{any} holomorphic map $\hat{s}:M\setminus N\ra \bP(V)$ in $M\times \bP(B)$ where $M$ is a manifold, $N$ is a closed analytic set and $V$ a vector space. Take for example $M=\bC$, $N=\{0\}$, $V=\bC^2$ and $\hat{s}(z)=[1:e^{1/z}]$. This clearly is no blow-up. It is fundamental that $N$ is zero locus of  the same map from $M$ to $V$ that induces $\hat{s}$. 
  \end{rem}
  
  We notice it is enough to look locally on $Y$ around points $p\in Z$ where $E$ trivializes. Then in such a neighborhood $U$, the closure of the graph of the map  $\bP(s):(Y\setminus Z)\cap U\ra \bP(\bC^l)$ satisfies the (universal) properties of the blow-up as one can see \cite{Fi} Section 4.1\footnote{for the algebraic counterpart see \cite{EH}, Prop 4.22}.
  
  Notice that the $\Bl_Z(Y)$ comes with a natural projection to $Y$ given by the restriction of $\pi:\bP(E)\ra Y$. This  \emph{blow-down map} is  
    \begin{itemize}
  \item[(i)]  proper, since $\pi:\bP(E)\ra Y$ is proper and 
  \item[(ii)] a biholomorphism away from $Z$ since $\bP(s)$ is biholomorphic onto its image.
  \end{itemize}
  \begin{example} We look at the blow-up of $Y$ in $E$ as the zero section. The section $s$ is now the tautological section $s^{\tau}:E\ra \pi^*E=E\times_YE$ given by the diagonal $p\ra (p,p)$. In other words, we are looking at the closure in $\pi^*(\bP(E))=E\times_Y\bP(E)$ of the image of the map which is  $(p,v)\ra (p, v,[v])$ for every $p\in M$ and $v\in E_p$. Hence we are taking the fiberwise blow-up of the origin in $E_p$, for all $p$ and therefore get
  \[ \Bl_{[0]}(E)=\{(v,[w])\in E\times_Y\bP(E)~|~v\wedge w=0\}
  \] 
  where $\wedge$ denotes colinearity and the relation is obviously considered fiberwise.
  \end{example}
  
  \begin{example} Suppose $Y$ is an affine scheme modelled on the ring $R= \bC[x_1,\ldots, x_m]/\mathcal{I}_Y$ for some ideal $\mathcal{I}_Y$ of polynomials and $E\simeq \underline{\bC}^l$ is the trivial bundle. Therefore the ideal $\mathcal{I}_Z$ has naturally $l$ generators, the components $s_1,\ldots,s_l$ of the section $s$. Then the Rees algebra is the quotient of the polynomial ring $R[y_1,\ldots,y_l]$ by the homogeneous ideal $\mathcal{I}^h(V)$ generated by \emph{all} homogeneous polynomials in the $y$ variables that vanish on the set of points $V:=\{y_1-s_1=0,\; y_2-s_2,\;\ldots,y_l-s_l=0\}.$ But $V$ is exactly the set of points that lie in the image of $\bP(s)$ and the closure of $V$ in $\bP(R[y_1,\ldots,y_l])=Y\times \bP(\bC^l)$ is generated by such an ideal.  
  
  Now $\mathcal{I}^h(V)$ is closely related, but not isomorphic in general with $S^2(\mathcal{I}_Z)$ which is the homogeneous ideal in $R[y_1,\ldots,y_l]$ generated by the $2\times 2$ minors of
  \[\left(\begin{array}{cccc}
  y_1 & y_2 & \ldots & y_n\\
  s_1& s_2 & \ldots & s_n
  \end{array}\right)
  \]
 Clearly each of the $2\times 2$ minors is an element of $\mathcal{I}^h(V)$ and in the case of a complete intersection we have that $S^2(\mathcal{I}_Z)=\mathcal{I}^h(V)$\footnote{For a more general situation when the Rees algebra is isomorphic with $\Sym(I)$ see \cite{Hu}.}.
  \end{example}
  \vspace{0.3cm}

We will use the word strict transform as a synonym for "subblow-up" as explained now. If $\iota:S\hookrightarrow Y$ is another closed subvariety, this time not necessarily given by a section of $E$,  then we can consider the restriction $\iota^*E\ra S$ and $\iota^*s:=s\circ \iota$ becomes a section of  $\iota^*$ whose zero locus determines the subvariety $S\cap Z$. Taking the closure of the image of $\bP(\iota^*s)$ in $\iota^*\bP(E)$ gives the ( sub)blow-up space $\Bl_{S\cap Z}(S)$. Now $\iota^*\bP(E)$ is embedded in $\bP(E)$.  One can look at the closure of the image of $\bP({s})\bigr|_{S\setminus S\cap Z}$ inside $\bP(E)$. This by definition is the strict transform of $S$ in $\Bl_{Z}(Y)$. This closure obviously coincides, at least at  a "topological" level with $\Bl_{S\cap Z}(S)$ via the embedding $\iota^*\bP(E)\ra \bP(E)$. One can check this is also the case as complex spaces (see \cite{EH} Proposition IV-21 in the algebraic case) and we will identify the two spaces inside $\bP(E)$.

It should be clear that $\Bl_Z(Y)$  depends only on the isomorphism type of the pair $(Y,Z)$. In other words if $\varphi:Y\ra Y'$ is an isomorphism such that $\varphi\bigr|_{Z}$ restricts to an isomorphism onto $Z'\subset Y'$, then $\Bl_Z(Y)\simeq \Bl_{Z'}(Y')$. We check this in a particular case.

 Let $s:Y\ra E$ be a section. Then the projection $\pi:E\ra Y$ induces an isomorphism of pairs of analytic spaces $(s(Y),s(s^{-1}(0)))\simeq (Y,s^{-1}(0))$. We will abuse notation and write $s^{-1}(0)$ also for $s(s^{-1}(0))$ and denote both spaces by $X$. Then $\Bl_X(s(Y))$ is the strict transform of $s(Y)$ inside $\Bl_{[0]}(E)$ and hence an analytic subvariety of $E\times_Y\bP(E)$ where the blow-down map is the restriction of  the projection $\pi_1:E\times_Y\bP(E)$ to $\Bl_X(s(Y))$. At a closer inspection we see that $\Bl_X(s(Y))$ is an analytic subvariety of $s(Y)\times_Y\bP(E)$. Now the projection
 \[ \pi_2:E\times_Y\bP(E)\ra \bP(E)
 \]
 induces an isomorphism $s(Y)\times_Y\bP(E)\simeq \bP(E)$ which makes the next diagram commutative
 \begin{equation}\label{blid}\xymatrix{s(Y)\times_Y\bP(E) \ar[rr]^{\pi_2} \ar[dr]^{\pi\circ \pi_1} & & \bP(E)\ar[dl]^{\pi}\\
   &Y&}
 \end{equation}
 where we again abused notation and called $\pi$ both the projection $E\ra Y$ and $\bP(E)\ra Y$. It is not hard to see that $\pi_2$ induces an isomorphism 
 \[\Bl_{X}(s(Y))\simeq \Bl_X(Y)\]
 that commutes with the blow-down maps, here the vertical arrows. 
 \vspace{0.5cm}

\subsection{Cones}  
   Let $V$ be a complex vector space. A closed  analytic space\footnote{not necessarily reduced nor irreducible} $A\subset \bP(V)$ of dimension $k$ is defined by a finite number of homogeneous equations over $V$. Then the projective cone $CA\subset \bP(\bC\oplus V)$ over $A$ is defined by the exactly the same equations, hence it has the same codimension. 
   
 There are some alternative geometric descriptions.
 
For the first one consider  the affine cone 
 \[ C^aA\subset V,\qquad C^aA:=\pi^{-1}A\cup \{0\}
 \]
Then considering $V$  an open subset of $\bP(\bC\oplus V)$ one defines $CA$ as the closure of $C^aA$, i.e. the "smallest" complex analytic space that contains $C^aA$.
 The affine cone $C^aA$ is then  intersection $C^aA=CA\cap V$.
 
  The  direct sum of two cones $C_1\subset V$, $C_2\subset W$ is a well-defined cone $C_1\oplus C_2\subset V\oplus W$. Therefore, in order to obtain $CA$,  instead of taking the closure of $C^aA$ in $\bP(\bC\oplus V)$ one might as well "projectivize" the cone $\bC\oplus C^aA$, i.e. take the image under the  projection $\bC\oplus V\setminus \{0\}\ra \bP(\bC\oplus V)$.

Another description that avoids mentioning the affine cone   is  as follows:   
\[ CA=\overline{\pi^{-1}(A)}=\pi^{-1}(A)\cup\{[1:0]\}\subset \bP(\bC\oplus V)
\]
where $\pi:\bP(\bC\oplus V)\setminus \{[1:0]\}\ra \bP(V)$ is the "stereographic" projection
\[ [w] \ra [ (\bC w+\bC(1,0))\cap V].
\]

Yet another description is contained in the following.

\begin{prop}\label{P1} Let $A\subset \bP(V)$ be an analytic set. Then the cone $CA$ coincides with the exceptional divisor of the blow-up  $\Bl_{\infty\times 0}(\bP^1\times C^aA)$.
\end{prop}
\begin{proof} The exceptional divisor of the strict transform of a  cone $C\subset V$ in $\Bl_0(V)$ is the projectivization of the cone $\bP(C)$. The exceptional divisor of the blow-up $\Bl_{\infty\times 0}(\bP^1\times C^aA)$ is the same as the exceptional divisor of $\Bl_{0\times 0} (\bC\oplus C^a{A})$. Therefore this is
\[\bP(\bC\oplus C^a{A})\subset \bP(\bC\oplus V)
\] 
one possible description of $CA$.
\end{proof}
\vspace{0.5cm}

If $F\ra M$ is a vector bundle and $A\subset \bP(F)$ is a complex analytic space, then the fiberwise cone $C^fA\subset \bP(\bC\oplus F)$ is obtained  as follows:
\begin{itemize} 
\item[(a)] take the closure $C^{fa}A$ of $\tilde{\pi}^{-1}(A)$ in $F$ where $\tilde{\pi}:F\setminus \{0\}\ra \bP(F)$ is the projection; this is the fiberwise affine cone.
\item[(b)] take the closure $C^fA$ of $C^{fa}A$ in $\bP(\bC\oplus F)$ by regarding $F$ as an open subset of $\bP(\bC\oplus F)$. 
\end{itemize}

Alternatively, notice that $C^{fa}A$ is a cone over $\pi(A)$ (where $\pi:\bP(F)\ra M$), a well-defined analytic space. Then one can take the direct sum of cones $\bC\oplus C^{fa}A$ and finally $\bP(\bC\oplus C^{fa}A)$. It  easily  seen that this produces an analytic subspace of $\bP(\bC\oplus F)$.

The equations of $C^fA$ are locally the same as the equations of $A$.

\vspace{0.5cm} Let $\pi:E\ra M$ be another holomorphic vector bundle. Then $E\times_M\bP(E)$ is in fact the bundle $\pi^*\bP(E)\ra E$. Hence, for any analytic space $A\subset E\times_M\bP(E)$ one has a well-defined fiberwise cone $C^fA\subset \pi^*\bP(\bC\oplus E)=E\times_M\bP(\bC\oplus E)$  and a corresponding affine cone $C^{fa}A:=C^fA\cap (E\times_ME)$.

\vspace{0.5cm}

One useful example of a complex set $A\subset E\times_M\bP(E)$ is  $A=\Ex_{[0]\cap B}(B)$ where $B$ is an analytic variety in $E$ and $\Ex_{[0]\cap B}(B)$ is the exceptional divisor of the strict transform of $B$ with respect to the blow-up of the zero section. Notice that in this situation $A\subset [0]\times_{M}\bP(E)\simeq \bP(E)$ and one might as well think of $A$ as an analytic subvariety of $\bP(E)$. The fiberwise affine cone is then $C^{fa}A\subset E$.

 We have a fiberwise analogue of Proposition \ref{P1}.
\begin{prop}\label{P2} Suppose $A\subset [0]\times_M\bP(E)$.  Then the fiberwise cone $C^fA$ coincides with the exceptional divisor of the blow-up $\Bl_{\infty\times ([0]\cap C^{fa}A)}(\bP^1\times C^{fa}A)$\footnote{also the exceptional divisor of the strict transform of $\bP^1\times C^{fa}{A}$ inside $\Bl_{\infty\times[0]}(\bP^1\times E)$.}  where $C^{fa}A$ is considered inside $E$.
\end{prop}
\begin{proof} One has directly a fiberwise version of the argument in the proof of Proposition \ref{P1}.
\end{proof}
The case that interests us is the following.
\begin{cor} \label{C1} Let $A=\Ex_{0\cap B}(B)\subset E\times_M\bP(E)$ for some analytic subspace $B\subset E$. Then $C^fA$ coincides with the exceptional divisor of $\Bl_{\infty\times (0\cap B)}(\bP^1\times B)$.
\end{cor}
\begin{proof} Use Proposition \ref{P2} and the fact that the exceptional divisors of $\Bl_0(B)$ and $\Bl_0(C^{fa}{A})$ coincide.
\end{proof}

Let $s:M\ra E$ be a holomorphic section and let $X:=s^{-1}(0)$ with the complex analytic structure induced by $s$ itself in local coordinates. We will abuse the notation and make no distinction between $X$ and $s(X)\subset [0]$.

Let $\Bl_{X}(s(M))$ be the strict transform of $s(M)$ of the blow-up of $0$ in $E$. Notice that $\Bl_{X}(s(M))\subset E\times_{M}\bP(E)$. It is well-known (see \cite{EH}) that $\Bl_{X}(s(M))$ can be described as a $\sigma$-process in its own right, namely as the blow-up of the complex analytic intersection $[0]\wedge s(M)=s(X)$ inside the complex manifold $s(M)$. Therefore $\Bl_{X}(s(M))$ is a complex analytic set of dimension $n$ with an exceptional divisor $\Ex_{X}(s(M))\subset [0]\times_M \bP(E)$ a complex analytic space of dimension $n-1$. 

\begin{prop}\label{fcprop} The fiberwise cone $C^f\Ex_X(s(M))\subset [0]\times_M \bP(\bC\oplus E)$ is a complex analytic space of dimension $n$ and is naturally isomorphic to $\bP(\bC\oplus C_{X}M)$, where $C_{X}M$ is the affine normal cone of $X$ in $M$.

Moreover, via this isomorphism, the natural projection $\bP(\bC\oplus C_{X}M)\ra X$ gets identified with the restriction of the projection $\pi_1:[0]\times_M \bP(\bC\oplus E)\ra [0]=M$ to $C^f\Ex_X(s(M))$.
\end{prop}
\begin{proof}  By the previous corollary, the fiberwise cone $C^f\Ex_X(s(M))$  coincides with the exceptional divisor of the blow-up $\Bl_{\infty\times X}(\bP^1\times s(M))$.  The other statements are  consequences of the commutative diagram (\ref{blid}).
\end{proof}

\appendix

\end{document}